\documentclass[11pt]{amsart}

\usepackage{geometry}
\usepackage{amssymb}
\usepackage{hyperref}
\usepackage{color}

\usepackage{amsmath}
\usepackage{amsthm}
\usepackage{amsfonts}
\usepackage{bbm}
\usepackage{color}
\usepackage{amssymb}
\usepackage{mathrsfs}

\newcommand{\chk}{} %

\DeclareFontFamily{U}{mathx}{\hyphenchar\font45}
\DeclareFontShape{U}{mathx}{m}{n}{
<5> <6> <7> <8> <9> <10>
<10.95> <12> <14.4> <17.28> <20.74> <24.88>
mathx10
}{}
\DeclareSymbolFont{mathx}{U}{mathx}{m}{n}
\DeclareFontSubstitution{U}{mathx}{m}{n}
\DeclareMathAccent{\widecheck}{0}{mathx}{"71}

\usepackage[explicit,indentafter]{titlesec}
\titleformat{\section}{\centering\normalfont\scshape}{\thesection.}{.5em}{#1}
\titleformat{\subsection}[runin]{\normalfont\itshape}{\textnormal{\thesubsection.}}{.5em}{#1.}
\titleformat{\subsubsection}[runin]{\normalfont\itshape}{\thesubsubsection.}{.5em}{#1.}
\titlespacing{\section}{0em}{1em}{0.5em}
\titlespacing{\subsection}{0em}{.5em}{0.5em}

\hypersetup{
	colorlinks=true,
	citecolor=[rgb]{0,0,0.5},
	linkcolor=[rgb]{0,0,0.5},
	urlcolor=[rgb]{0,0,0.75},
	pdfpagemode=UseNone,
	pdfstartview=FitH,
	pdfdisplaydoctitle=true,
	pdftitle={Trilinear smoothing inequalities and a variant of the triangular Hilbert transform},
	pdfauthor={M. Christ, P. Durcik, J. Roos},
	pdflang=en-US
}

\begin{document}
\setcounter{tocdepth}{1}

\newtheorem{theorem}{Theorem}[section]
\newtheorem{thma}{Theorem}
\newtheorem{proposition}[theorem]{Proposition}
\newtheorem{corollary}[theorem]{Corollary}
\newtheorem{lemma}[theorem]{Lemma}
\newtheorem{sublemma}[theorem]{Sublemma}
\newtheorem{conjecture}[theorem]{Conjecture}
\newtheorem{variant}[theorem]{Variant}
\newtheorem{fact}[theorem]{Fact}
\newtheorem{observation}[theorem]{Observation}
\def\distance{\operatorname{distance}}

\theoremstyle{remark}
\newtheorem{definition}[theorem]{Definition}
\newtheorem{remark}[theorem]{Remark}
\newtheorem*{remarka}{Remark}

 \numberwithin{theorem}{section}
 \numberwithin{equation}{section}

\newcommand{\norm}[1]{ \|  #1 \|}

\def\bk{\mathbf k}
\def\kernel{\operatorname{kernel}}
\def\eps{\varepsilon}

\def\deltaprime{\delta'}
\def\bart{\bar t}

\def\R{\mathbb{R}}
\def\reals{\mathbb{R}}
\def\complex{\mathbb{C}}
\def\integers{\mathbb{Z}}
\def\Z{\mathbb{Z}}
\newcommand{\A}{\mathcal{A}}
\newcommand{\B}{\mathcal{B}}
\def\g{\mathbbm{g}}
\def\h{\mathbbm{h}}
\def\md{\mathscr{D}}

\def\p{\mathrm{p}}
\def\q{\mathrm{q}}
\def\r{\mathrm{r}}

\def\scriptt{\mathcal{T}}
\def\scriptm{\mathcal{M}}
\def\bff{\mathbf f}

\subjclass[2020]{42B15, 42B20, 42B25, 05D10}

\title[A  triangular Hilbert transform]{Trilinear smoothing inequalities and\\a variant of the triangular Hilbert transform}

\author[M. Christ, P. Durcik, J. Roos]{Michael Christ \ \ Polona Durcik \ \ Joris Roos}

\address{
        Michael Christ\\
        Department of Mathematics\\
        University of California \\
        Berkeley, CA 94720, USA}
\email{mchrist@berkeley.edu}

\address{
        Polona Durcik\\
        Schmid College of Science and Technology\\
        Chapman University\\
        Orange, CA 92866, USA
        }
\email{durcik@chapman.edu}

\address{Joris Roos\\
Department of Mathematical Sciences\\
University of Massachusetts Lowell\\
Lowell, MA 01854, USA\\
\& School of Mathematics\\
The University of Edinburgh\\
Edinburgh EH9 3FD, UK}
\email{joris\_roos@uml.edu}

\thanks{The first author was supported by National Science Foundation grant
DMS-1901413.}

\begin{abstract}
Lebesgue space inequalities are proved 
for a variant of the triangular Hilbert transform involving curvature.
The analysis relies on a crucial trilinear smoothing inequality developed herein, 
and on bounds for an anisotropic variant of the twisted paraproduct.

The trilinear smoothing inequality also leads to Lebesgue space bounds for a corresponding maximal function
and a quantitative nonlinear Roth-type theorem concerning patterns in the Euclidean plane. 
\end{abstract}

\date{November 13, 2020.}

\maketitle

\def\LocT{\Lambda}
\def\LocTOp{T_{\mathrm{loc}}}

\section{Introduction}

Consider the trilinear form
\begin{equation}
\scriptt(f_1,f_2,f_3) =  \int_\scriptm \prod_{j=1}^3 f_j(x_j)\,d\mu(x),
\end{equation}
where $x = (x_1,x_2,x_3)\in (\reals^d)^3$, 
$\scriptm$ is a submanifold of $(\reals^d)^3$ of dimension strictly less
than $3d$, and $\mu$ is a compactly supported measure on $\scriptm$
with smooth density. 
Inequalities of the form
\begin{equation} \label{generaltriineq}
|\scriptt(f_1,f_2,f_3)| \leq C \prod_j \norm{f_j}_{W^{p,s}}
\ \text{ for some $s<0$},
\end{equation}
where $W^{p,s}$ is the Sobolev space of functions having $s$
derivatives in $L^p$, have recently been investigated
in the case $d=1$ and $\dim(\scriptm)=2$ \cite{Chr20}.
In that work, microlocal orthogonality relations were combined with
upper bounds for Lebesgue measures of sublevel sets to establish
inequalities of this form in nearly maximal generality
for real analytic relations $\scriptm$. 
A particular example was analyzed earlier by Bourgain \cite{Bou88}. The paper \cite{Chr20} also establishes
certain results for $d>1$; in those results,
$\dim(\scriptm)=4$ when $d=2$.

In the present paper we begin the study of more singular situations. We focus on a key example with $d=2$, $\dim(\scriptm)=3$, that goes beyond \cite{Chr20} and establish for it an inequality of the type 
\eqref{generaltriineq}
in Theorem~\ref{mainresult}.  In the analysis of inequalities of the type \eqref{generaltriineq}
for $d>1$, a phenomenon arises that was not encountered for $d=1$.
We introduce two additional ingredients to the framework developed in \cite{Chr20}
in order to treat this phenomenon. 
While Theorem \ref{mainresult} deals with a particular case that is motivated by its applications, we believe that there will be further applications of the underlying
method of proof.

As the main application of Theorem \ref{mainresult} we derive Lebesgue space bounds for a modulation invariant bilinear singular integral operator, which is related to the triangular Hilbert transform; see Theorem \ref{thm:singint}.
In addition, we develop applications to associated maximal functions in Theorem \ref{thm:maxfct}, 
and to the existence of certain Roth-type patterns in measurable sets 
in Theorem~\ref{thm:patterns}.
Another application concerning almost everywhere convergence of certain continuous-time nonlinear ergodic averages associated 
to two commuting $\reals$--actions was established in joint work of the authors with Kova\v{c} \cite{cdkr}.

We first indicate these applications, before discussing the foundational inequality on which they rely. 

\subsection{A triangular Hilbert transform with curvature}

Consider the bilinear singular integral operator  
\begin{equation}\label{eqn:thtc}
T(f_1, f_2)(x,y) = \mathrm{p.v.}\int_{\R} f_1(x+t, y) f_2(x, y+t^2)\,\frac{dt}t,
\end{equation}
defined \emph{a priori} for test functions $f_1, f_2: \R^2\to\complex$. 
We obtain the following result.
\begin{thma}\label{thm:singint}
Let $p,q\in (1,\infty)$, $r\in [1, 2)$ satisfy $p^{-1}+q^{-1}=r^{-1}$.
Then $T$ extends to a bounded operator $L^p\times L^q\to L^r$.
\end{thma}

The operator $T$ is a 
variant of
\begin{equation}
\label{tht}
    (f_1,f_2)\mapsto \mathrm{p.v.}\int_{\R} f_1(x+t, y) f_2(x, y+t)\,\frac{dt}t,
\end{equation}
which is known as the {\em triangular Hilbert transform}. 
Determining whether the triangular Hilbert transform satisfies any 
Lebesgue norm bounds is a significant open problem. Partial progress for a  Walsh model  
was obtained by Kova\v{c}, Thiele and Zorin-Kranich \cite{KTZ15}. Cancellation estimates concerning a truncated
triangular Hilbert transform with a finite number of scales, with improvement on the trivial
bound by the number of scales, were addressed in 
\cite{Zor17}, \cite{DKT16}.

Theorem \ref{thm:singint} unifies two previously known
inequalities.  
First, it implies (see \S \ref{sec:bht}) $L^p$ bounds for a variant of the bilinear Hilbert transform with curvature,
\begin{equation}\label{eqn:bht0}
(f_1,f_2)\mapsto \mathrm{p.v.}\int f_1(x+t)f_2(x+t^2)\,\frac{dt}t.
\end{equation} 
The $L^2\times L^2\to L^1$ bound for this operator was first proved by Li \cite{Li13}. An alternative proof in a more general context was given by Lie \cite{Lie15}; also see \cite{LX16, Lie18} for $L^p$ bounds.  
The operator \eqref{eqn:bht0} is a nonlinear variant (in the sense that the mapping $(x,t)\to x+t^2$ is nonlinear) 
of the bilinear Hilbert 
transform, 
studied by Lacey and Thiele 
\cite{laceythiele1}, \cite{laceythiele2}. 

Second, Theorem \ref{thm:singint} also yields (see \S \ref{sec:sw}) $L^p$ bounds for the operator
\begin{equation}\label{eqn:sw0}
f\mapsto \sup_{N\in\R} \Big| \mathrm{p.v.} \int_{\R} f(x-t) e^{iN t^2}\,\frac{dt}t\Big|,
\end{equation}
which were proved by Stein \cite{Ste95} (also see work of Stein and Wainger \cite{SW01}). Note that replacing $t^2$ by $t$ in the phase gives Carleson's operator.

The operators $T$ and \eqref{tht} both exhibit certain general modulation symmetries that are not present  in \eqref{eqn:bht0} and also not in the classical bilinear Hilbert transform \cite{laceythiele1}, \cite{laceythiele2}; see also   comments after Theorem \ref{mainresult}. 
The study of singular integral operators with general modulation symmetries and a certain bipartite structure was initiated  by Kova\v{c}  in the work on the twisted paraproduct \cite{Kov12}, which is a degenerate case of the two-dimensional bilinear Hilbert transform, \cite{DT10}. 
In Theorem \ref{thm:singint}  we are concerned with an operator $T$ which features a more singular  structure.
In addition to trilinear smoothing inequalities,  Theorem \ref{thm:singint} also relies on bounds for   an   anisotropic variant  of the operator in \cite{Kov12}.
 
Let $\alpha$ and $\beta$ be positive integers. Let $m$ be a 
smooth function on $\mathbb{R}^2\setminus\{(0,0)\}$ which satisfies 
\begin{align}\label{symest}
|\partial_\xi^k \partial_\eta^\ell m(\xi,\eta)| \leq C_{\alpha,\beta}\, (|\xi|^{1/\alpha}+|\eta|^{1/\beta})^{-\alpha k - \beta \ell}
\end{align}
for all  $k,\ell \geq 0$ up to a large finite order.
For test functions $f_1,f_2$ on $\mathbb{R}^2$ we let
\[ T_m(f_1,f_2)(x,y) = \int_{\R^2} f_1(x+s,y) f_2(x,y+t) K(s,t)\,ds\,dt \]
with $K$ the distribution satisfying $m=\widehat{K}$.

\begin{thma}\label{thm:anisotp}
Let $p,q\in (1,\infty)$, $r\in (\tfrac12, 2)$ be such that $p^{-1}+q^{-1}=r^{-1}$. Assume that $m$ satisfies \eqref{symest}. Then $T_m$ extends to a bounded operator %
$L^p\times L^q\rightarrow L^r$.
\end{thma}

A  very  recent work of Kova\v{c} \cite{Kov20} concerning certain anisotropic point configurations in Euclidean space also features an estimate for a certain multilinear form related to the operator $T_m$.  Some variants of $T_m$ with general dilation structure have also been studied in \cite{KS15}.

\begin{remarka}
Our methods also apply to the operators  
\[ (f_1,f_2)\longmapsto \mathrm{p.v.}\int_{\R} f_1(x+[t]^\alpha, y) f_2(x, y+[t]^\beta)\,\frac{dt}t,\]
where $\alpha,\beta\ge 1$, $\alpha\not=\beta$ are real numbers and $[t]^\alpha$ stands for $|t|^\alpha$ or $\mathrm{sgn}(t)|t|^\alpha$. We have chosen $\alpha=1$, $\beta=2$ in order not to clutter the presentation.
\end{remarka}

\subsection{Further applications}

Consider the bilinear maximal operator
\begin{equation}\label{eqn:maxfct}
M(f_1, f_2) (x,y) = \sup_{r>0} \frac{1}{2r} \int_{-r}^r |f_1(x+t,y) f_2(x,y+t^2)|\,dt.
\end{equation}
\begin{thma}\label{thm:maxfct}
For every $p,q\in (1,\infty)$, $r\in [1,\infty)$ with $p^{-1}+q^{-1}=r^{-1}$ there exists $C\in (0,\infty)$ such that for all test functions $f_1, f_2$,
\[ \|M(f_1, f_2)\|_{L^r} \le C \|f_1\|_{L^p} \|f_2\|_{L^{q}}. \]
\end{thma}
Note that for $r>1$, the claim in Theorem \ref{thm:maxfct} follows immediately from H\"older's inequality and the Hardy-Littlewood maximal theorem (in fact, this argument also gives $L^p\times L^{p'}\to L^{1,\infty}$ bounds for all $p\in (1,\infty)$).

Another application
is the following combinatorial result on the existence of 
patterns in subsets of $[0,1]^2$ of positive measure.

\begin{thma}\label{thm:patterns}
Let $\varepsilon\in(0,\tfrac12)$ and $E\subset [0,1]^2$ a measurable set of Lebesgue measure at least $\varepsilon$. Then there exist 
\[ (x,y), (x+t,y), (x,y+t^2)\in E \]
with $t>\exp(-\exp(\varepsilon^{-C}))$ for some constant $C>0$ not depending on $E$ or $\varepsilon$.
\end{thma}
Observe that the corresponding result without the postulated lower bound on $t$ follows by an application of the Lebesgue density theorem. A  result with worse $\varepsilon$-dependence can also be deduced from a general theorem of Bergelson and Leibman \cite{BL96}. 

Note that the theorem recovers the following quantitative non-linear Roth theorem due to Bourgain \cite{Bou88}: for every $E\subset [0,1]$ with Lebesgue measure at least $\varepsilon$ there exist $t>\exp(-\exp(\varepsilon^{-c}))$ and $x$ with
\[ x, x+t, x+t^2 \in E. \]
This follows by applying Theorem \ref{thm:patterns} to the set $\widetilde{E}=\{(x,y)\in [0,1]^2\,:\,x-y\in E\}$.

\subsection{A trilinear smoothing inequality}
The next result lies at the heart of our analysis, powering the applications
indicated above.
Let $\zeta$  be a smooth function  with compact support
in $\R^2\times (\R^1\setminus\{0\})$.  
Consider the trilinear form
\begin{equation}\label{eqn:locformdef}
\Lambda(f_1,f_2,f_3)  = \int_{\R^3} f_1(x+t,y) f_2(x,y+t^2)f_3(x,y) \zeta(x,y,t) \,dx\,dy\,dt.
\end{equation}
\begin{thma} \label{mainresult} 
There exist constants $C>0$ and $\sigma>0$ so that for all test functions $f_1,f_2, f_3$,
\begin{equation}\label{eqn:local-gain1}
| \Lambda(f_1,f_2,f_3)|  \le C \|f_1\|_{H^{(-\sigma,0)}}
\|f_2\|_{H^{(0,-\sigma)}} \|f_3\|_{L^\infty},
\end{equation}
where the constant $C$ only depends on $\zeta$, and 
\[ \|f\|_{H^{(a,b)}}^2 = 
\int_{\reals^2} |\widehat{f}(\xi_1,\xi_2)|^2
\,(1+|\xi_1|^{2})^{\frac{a}2}\, (1+|\xi_2|^{2})^{\frac{b}2}\,
d\xi_1\,d\xi_2.\] 
\end{thma}
 Theorem~\ref{mainresult} is related to an inequality of Bourgain \cite[(2.4)]{Bou88} roughly in the same way as Theorem \ref{thm:singint} is related to Li's theorem \cite[Theorem 1.1]{Li13}.
The inequality \eqref{eqn:local-gain1} is equivalent to the estimate
\[|\Lambda(f_1,f_2,f_3)| \le C \lambda^{-\sigma}  \|f_1\|_{L^2}
\|f_2\|_{L^{2}}\|f_3\|_{L^\infty}\]
valid for all $\lambda\ge 1$ under the assumption that $\widehat{f_j}$ is supported
where $|\xi_j|\asymp\lambda$ for at least one index $j=1,2$.

The trilinear form $\Lambda$ has a rich symmetry structure,
related to the symmetry structure of $T_m$ and \eqref{tht}. Indeed,
for any unimodular functions $\varphi,\psi:\R\to\mathbb{C}$,
$$\Lambda(f_1,f_2,f_3) = \Lambda(\varphi(y)f_1,\psi(x)f_2,\overline{\varphi(y)\psi(x)}f_3),$$
where $\varphi(y)f_1$ indicates the function $(x,y)\mapsto \varphi(y)f_1(x,y)$,
and so on. Thus there is no valid version of \eqref{eqn:local-gain1}
involving $\norm{f_1}_{H^{(0,-\sigma)}}$ or
$\norm{f_2}_{H^{(-\sigma,0)}}$. Nor is there a formulation 
involving a negative order norm of $f_3$.

The proof of Theorem \ref{mainresult} begins with  microlocal
 decompositions, as in \cite{Chr20}.
Two issues arise in the proof  
that were not encountered in the precursor work \cite{Chr20}. 
Firstly, the trilinear form that arises here is more singular
than those studied in \cite{Chr20}. To compensate, one is led to eliminate
one of the three functions, reducing matters 
by a standard Cauchy-Schwarz--type manipulation to a bilinear form.
There are situations in which this manipulation sacrifices
essential information, as a consequence of which no satisfactory bound is obtained.
To deal with these situations requires an inverse theorem and a resulting decomposition, which
recoups useful information in usable form.
Thus two structured cases arise in the analysis, one bilinear and one trilinear.

Secondly, Roth's archetypal proof \cite{Rot53} of the existence of arithmetic progressions of length three
in subsets of the integers with positive upper density
exploits a lower bound for $\sup_\xi |\widehat{f}(\xi)|$
in terms of a positive function of the density. 
In the proof of Theorem~\ref{mainresult} an analogous step arises, 
but $f$ is a periodic function of two real variables,
and ({\em grosso modo}) a lower bound is obtained for $\norm{\widehat{f}}_{\ell^2(\Gamma)}$
for a certain linear graph $\integers^2\supset\Gamma=\{(k_1,k_2): k_2=tk_1\}$.

This analysis leads to a sublevel set inequality
 in the spirit of \cite{Chr20}, the proof of which involves further technical difficulties.

\subsection*{Structure of the paper}
\begin{itemize}
\item In \S \ref{sec:lwpaley}, \S \ref{sec:high-red}, \S \ref{sec:low-red}, \S \ref{sec:mixed-red} we prove Theorem \ref{thm:singint} using Theorem \ref{thm:anisotp} and Theorem \ref{mainresult}. %

\item In \S \ref{sec:max-red} we modify these reductions to deduce Theorem \ref{thm:maxfct} from Theorem \ref{mainresult}.

\item In \S \ref{sec:localop} we prove the trilinear smoothing   inequality of Theorem \ref{mainresult}. The inverse theorem and a resulting decomposition are
developed in \S \ref{section:structure} and \S \ref{sec:local-applystructure}, respectively. 
The sublevel set inequality
 is analyzed
in \S \ref{sec:sublevel}.

\item In \S \ref{sec:anisotp} we give the proof of Theorem \ref{thm:anisotp}, adapting ideas from \cite{Kov12}, \cite{Ber12}, \cite{Dur14}. %

\item In \S \ref{sec:patterns} we deduce Theorem \ref{thm:patterns} from Theorem \ref{mainresult} following Bourgain \cite{Bou88}. 

\item In \S \ref{sec:bht}, \S \ref{sec:sw} we show how Theorem \ref{thm:singint} implies bounds for the operators \eqref{eqn:bht0}, \eqref{eqn:sw0}.

\item In \S \ref{section:openproblems} we indicate several open problems.
\end{itemize}

\subsection*{Notation} Fourier transforms will be denoted by $\widehat{f}(\xi) = \int_{\R^d} f(x) e^{-2\pi i x\cdot \xi} dx$. 
The letters $c,C$ are reserved for constants that may change from line to line. Dependence of the constants on various parameters will be understood from context or made explicit by appropriate subscripts. 
We write $A\lesssim B$ to denote existence of a constant $C$ such that $A\le C\cdot B$ and $A\gtrsim B$ similarly. The notation $A\asymp B$ signifies that $A\lesssim B$ and $A\gtrsim B$. In addition we use the notation $A=B+O(X)$ to denote that $|A-B|\lesssim X$. For a function $f$ on $\R^2$ and a function $\varphi$ on $\R$ we define the partial convolutions
\[(f*_1 \varphi) (x,y) = \int_\R f(x-u, y) \varphi(u) du,\quad (f*_2 \varphi)(x,y) = \int_\R f(x,y-u) \varphi(u) du .\]
For a measurable set $E\subset \R^d$ we denote the $d$-dimensional Lebesgue measure of $E$ by $|E|$ and make no reference of the dimension in the notation.
$L^p$ norms will be denoted by $\|f\|_p$, $\|f\|_{L^p}$, or $\|f\|_{L^p(\R^d)}$, depending on context.

\section{Preliminary reductions}\label{sec:prelim}
In this section we reduce the proof of Theorem \ref{thm:singint} to those of Theorems \ref{mainresult} and \ref{thm:anisotp}. In \S \ref{sec:max-red} we give a variant of this reduction which also gives us Theorem \ref{thm:maxfct} (which only uses Theorem \ref{mainresult}, but not Theorem \ref{thm:anisotp}). 

\subsection{Littlewood--Paley decompositions}\label{sec:lwpaley}
Let $\varphi$ be a smooth even function on $\R$ supported in the set $\{|\zeta|\le 2\}$ which is equal to one on $\{ |\zeta|\le 1\}$ and satisfies $0\le \varphi\le 1$. Let $\psi(\zeta)=\varphi(\zeta)-\varphi(2\zeta)$ and for $j\in \Z$ let $\varphi_j(\zeta) = \varphi(2^{-j} \zeta)\;\text{and}\;\psi_j(\zeta) =\psi(2^{-j} \zeta)$. Then $\psi_j$ is supported in $\{2^{j-1}\le |\zeta|\le 2^{j+1}\}$ and $\sum_{j\in\Z} \psi_j(\zeta) = 1$ holds for all $\zeta\not=0$. We define the partial Littlewood--Paley operators $\Delta^{(\ell)}_j$ and corresponding partial sums $S_j^{(\ell)}$ by
\[ \widehat{\Delta^{(\ell)}_j f} (\xi) = \psi_j(\xi_\ell) \widehat{f}(\xi),\quad \widehat{S^{(\ell)}_j f} (\xi) = \varphi_j(\xi_\ell) \widehat{f}(\xi), \]
where $\ell=1,2,3$, $j\in\Z$, $\xi=(\xi_1,\xi_2)\in\R^2$. We shall also make use of a second set of Littlewood--Paley operators $\widetilde{\Delta}^{(\ell)}_j$ associated with multipliers $\xi\mapsto \widetilde{\psi}_j(\xi_\ell)$ chosen so that $\psi_j \widetilde{\psi}_j = \psi_j$, say $\widetilde{\psi}_j(\zeta) = \widetilde{\psi}(2^{-j} \zeta),$ where $\widetilde{\psi}$ is a smooth function that equals one on the support of $\psi$, but is supported on, say, a $\tfrac1{100}$-neighborhood of the support of $\psi$.
Decompose $T=\sum_{j\in\Z} T_j$, where
\[ T_j(f_1, f_2)(x,y) =  \int_{\R} f_1(x+t, y) f_2(x, y+t^2) \psi(2^{j} t) t^{-1} dt. \]
Standard stationary phase considerations motivate the following further decomposition in frequency as
\begin{equation}\label{eqn:basicfreqdecomp}
T = T^{\mathrm{L}} + T^{\mathrm{M}} + T^{\mathrm{H}},
\end{equation}
into components representing the {\em low} (L), {\em mixed} (M) and {\em high} (H) frequency contributions, respectively. Here we write for $\omega\in\{\mathrm{L},\mathrm{M},\mathrm{H}\},$
\[ T^\omega = \sum_{j\in\Z} T^\omega_j,\quad T_j^\omega (f_1, f_2) = \sum_{k \in \mathfrak{F}_\omega} T_j(\Delta^{(1)}_{j+k_1} f_1, \Delta^{(2)}_{2j+k_2} f_2),\]
where $k=(k_1,k_2)\in\Z^2=\mathfrak{F}_\mathrm{L} \cup \mathfrak{F}_\mathrm{M}\cup \mathfrak{F}_\mathrm{H}$ with
\begin{align}
 \mathfrak{F}_\mathrm{L} &= \{ k\in\Z^2\,:\, \max(k_1,k_2)\le 0 \}, \nonumber\\
 \mathfrak{F}_\mathrm{M} &= \{ k\in\Z^2\,:\, \max(k_1,k_2)>0,\,|k_1-k_2|>100 \}, \label{eqn:basic-freq-dec}\\
 \mathfrak{F}_\mathrm{H} &= \{ k\in\Z^2\,:\, \max(k_1,k_2)>0,\,|k_1-k_2|\le 100 \}\nonumber.
\end{align}

The low frequency component $T^\mathrm{L}$ is not oscillatory and can be viewed as a more well-behaved bilinear singular integral, more specifically an anisotropic variant of the twisted paraproduct \cite{Kov12}. This reduction will be done in \S \ref{sec:low-red}. In the mixed frequency component $T^\mathrm{M}$ we make use of rapid decay stemming from a non-stationary phase to similarly reduce to an anisotropic variant of the twisted paraproduct, see \S \ref{sec:mixed-red}. Finally, the high frequency component $T^\mathrm{H}$ requires a more intricate analysis. The key estimate is given in Theorem \ref{mainresult}. In \S \ref{sec:high-red} we will detail the reduction of the high frequency component to Theorem \ref{mainresult}.

\subsection{Shifted maximal function}\label{sec:shiftedmaxfct}
Denote by $\mathscr{M}_{\sigma}$  the shifted (dyadic) maximal function 
\[\mathscr{M}_\sigma g(x) = \sup_{s\in\Z}  {2^{-s}} \int_{[\sigma 2^s, (\sigma+1)2^s]} |g(x+t)| dt\quad (x\in\R), \]
where $\sigma\in\R$.
It is well-known \cite[p. 78]{Ste93} that 
\begin{equation}\label{eqn:shiftedmaxfct-basic}
\|\mathscr{M}_\sigma g\|_p\lesssim \log(2+|\sigma|)^{1/p} \|g\|_p
\end{equation}
for all $1<p\leq \infty$. Similarly, from a variant of the Fefferman--Stein inequality we also have
\begin{equation}\label{eqn:shiftedfs}
\Big\| \Big(\sum_{j\in\Z} (\mathscr{M}_\sigma g_j)^2 \Big)^{1/2} \Big\|_p \lesssim \log(2+|\sigma|)^2 \Big\| \Big(\sum_{j\in\Z} |g_j|^2 \Big)^{1/2}\Big\|_p
\end{equation}
for all $p\in (1,\infty)$ (for a proof of this inequality, see \cite[Theorem 3.1]{GHLR17}). The following lemma provides a pointwise domination for each single scale piece $T_0$ in terms of the shifted maximal function.
 \begin{lemma}\label{lem:shiftedmaxfct} Let $\kappa$ be a positive integer. Then we have the pointwise estimate
\begin{equation}\label{eqn:shiftedmaxfct-main}
|T_0(f_1,\Delta_{\kappa}^{(2)}f_2)| \lesssim \sum_{\iota \in \mathcal{I}} a_\iota (\mathscr{M}^{(1)}_{\sigma_{1,\iota}} f_1) (\mathscr{M}^{(2)}_{\sigma_{2,\iota}}f_2) 
\end{equation}
with $\mathcal{I}$ a countable set and $a_\iota>0$, $\sigma_{\ell,\iota}\in\R$ satisfying \begin{equation}\label{eqn:shiftedmaxfct-coeff}
\sum_{\iota\in\mathcal{I}} a_\iota \log(2+|\sigma_{1,\iota}|)^a\log(2+|\sigma_{2,\iota}|)^b \lesssim \kappa^{a+b}
\end{equation}
for every $a,b>0$.
 \end{lemma} 
By the dilation symmetry, the estimate \eqref{eqn:shiftedmaxfct-main} implies that for every $j,k_1\in\Z$ and $\kappa=k_2\ge 1$,
 \begin{equation}\label{eqn:shiftedmaxfct-j}
|T_j(\Delta_{j+k_1}^{(1)} f_1,\Delta_{2j+k}^{(2)}f_2)| \lesssim \sum_{\iota \in \mathcal{I}} a_\iota (\mathscr{M}^{(1)}_{\sigma_{1,\iota}} \Delta_{j+k_1}^{(1)} f_1) (\mathscr{M}^{(2)}_{\sigma_{2,\iota}}f_2)     
 \end{equation} 
Indeed,
\[ T_j(\Delta_{j+k_1}^{(1)} f_1,\Delta_{2j+k_2}^{(2)}f_2) = D_{(2^j,2^{2j})} T_0 ( D_{(2^{-j},2^{-2j})} \Delta_{j+k_1}^{(1)} f_1, \Delta^{(2)}_{k_2} D_{(2^{-j}, 2^{-2j})} f_2) \]
and $D_{2^j} \mathscr{M} D_{2^{-j}} = \mathscr{M}$, where $D_{(a,b)} f(x,y) = f(ax, by)$ and $D_a g(x) = g(ax)$.

 \begin{proof}[Proof of Lemma \ref{lem:shiftedmaxfct}]
We expand 
\begin{equation}
    \label{T0shift}
    T_0(f_1,\Delta_{\kappa}^{(2)}f_2)(x,y) = \int_{\R^2} f_1(x+t,y)f_2(x,y+t^2-s) (\psi_{\kappa})^\vee (s) \psi(t)\,ds\,dt
\end{equation}
For   an integer $l$ we denote the   dyadic interval
\[I_{l}=[ l2^{-\kappa}, (l+1)2^{-\kappa}] = [0,2^{-\kappa}] + l2^{-\kappa} \]
Splitting the integration in $t$ over these intervals and using rapid decay of $\widecheck{\psi}$ we   estimate \eqref{T0shift} by 
\[
    \sum_{n\in \Z}(1+|n|)^{-N}   2^{-\kappa}\sum_{|l|\le 2^{\kappa+1}}   2^{2\kappa} \int_{I_{l}} \int_{2^{-\kappa}n}^{2^{-\kappa}(n+1)} |f_1|(x+t,y)|f_2|(x,y+t^2-s)\,ds\,dt \]
Suppose that $0< l\leq 2^{\kappa+1}$.
For fixed $n\in \Z$ and $l>0$, the summand can be written as 
\begin{equation}
      \label{est:m0}
   2^{2\kappa} \int_{I_{l}} \int_{2^{-\kappa}n-t^2}^{2^{-\kappa}(n+1)-t^2} |f_1|(x+t,y)|f_2|(x,y-s)\,ds\,dt
\end{equation}
The range of integration in $s$ is contained in 
\[J_l = [2^{-\kappa}n + l^22^{-2\kappa}, 2^{-\kappa}(n+1) + l^2 2^{-2\kappa}+2^{-\kappa+3} ] = [0,2^{\kappa-3}] + \sigma_{l,n}2^{\kappa-3} \]
where $\sigma_{l,n} =2^{-3}n + l^22^{-\kappa+3}  $.  %
  Thus, \eqref{est:m0} is bounded up to a constant multiple by
  \[      {|I_l|^{-1}}  {|J_l|^{-1}} \int_{I_{l}} \int_{J_l} |f_1|(x+t,y)|f_2|(x,y-s)\,ds\,dt  \]
If $-2^{k+1}\le \ell\le 0$ we proceed similarly, defining $\sigma_{l,n}$ so that
\begin{equation}
    \label{m0:finalbd}
 \eqref{T0shift}  \lesssim \sum_{n\in \Z}  (1+|n|)^{-N}   2^{-\kappa}\sum_{|l|\leq 2^{\kappa+1}}   \mathscr{M}_l^{(1)} f_1(x,y) \mathscr{M}^{(2)}_{\sigma_{l,n}}f_2(x,y).
\end{equation}
This gives the   desired estimate for $T_0$. \end{proof}

\subsection{High frequencies}\label{sec:high-red}
First we claim that it suffices to prove that 
there exists $c>0$ such that 
\begin{equation}\label{eqn:HHT0}
\| T_j ( \Delta^{(1)}_{j+k_1} f_1, \Delta^{(2)}_{2j+k_2} f_2) \|_{1} \lesssim 2^{-\delta |k|} \|f_1\|_{2} \|f_2\|_{2}
\end{equation}
for all $k=(k_1,k_2)\in\mathfrak{F}_\mathrm{H}$ and all $j\in\Z$. 
Note that \eqref{eqn:HHT0} holds trivially for $\delta=0$.  The reader may for simplicity assume that $k_1=k_2\ge 1$, but we will not make this assumption in the text. 

To verify the claim let us assume that \eqref{eqn:HHT0} holds. 
Define
\[ T^{(k)} (f_1,f_2)= \sum_{j\in\Z} T_j(\Delta^{(1)}_{j+k_1}  f_1, \Delta^{(2)}_{2j+k_2} f_2). \]
First we use \eqref{eqn:HHT0} to estimate for each $k\in \mathfrak{F}_\mathrm{H}$,
\begin{equation}\label{eqn:Tk-l2-decay}
\|T^{(k)} (f_1,f_2)\|_{1} \lesssim 2^{-c|k|} \sum_{j\in\Z}\|\widetilde{\Delta}^{(1)}_{j+k_1}  f_1\|_{2}\|\widetilde{\Delta}^{(2)}_{2j+k_2} f_2\|_{2} \lesssim 2^{-\delta |k|} \|f_1\|_{2}\|  f_2\|_{2}.
\end{equation}
where $\Delta^{(\ell)}_j \widetilde{\Delta}^{(\ell)}_j = \Delta^{(\ell)}_j$ and we have used the Cauchy--Schwarz inequality and Plancherel's theorem to treat the sum over $j$. 

Next, we use Lemma \ref{lem:shiftedmaxfct} (in the form of its consequence, \eqref{eqn:shiftedmaxfct-j}) to estimate for $p,q\in (1,\infty)$, $r\in [1,\infty)$, $p^{-1}+q^{-1}=r^{-1}$,
\[ \| T^{(k)} (f_1,f_2)\|_{r} \lesssim \sum_{\iota \in \mathcal{I}} a_\iota \Big\| \sum_{j\in\Z} (\mathscr{M}^{(1)}_{\sigma_{1,\iota}} \Delta^{(1)}_{j+k_1} f_1) (\mathscr{M}^{(2)}_{\sigma_{2,\iota}} \widetilde{\Delta}^{(2)}_{2j+k_2}f_2) \Big\|_r.
\]
By the Cauchy--Schwarz inequality applied to the summation in $j$ and H\"older's inequality the previous is
\[\le \sum_{\iota \in \mathcal{I}} a_\iota \Big\| \Big(\sum_{j\in\Z} (\mathscr{M}^{(1)}_{\sigma_{1,\iota}} \Delta^{(1)}_{j+k_1} f_1)^2 \Big)^{1/2}\Big\|_p \Big\|\Big(\sum_{j\in\Z}(\mathscr{M}^{(2)}_{\sigma_{2,\iota}} \widetilde{\Delta}^{(2)}_{2j+k_2}f_2)^2\Big)^{1/2} \Big\|_q. \]
Since $p,q\in (1,\infty)$, by \eqref{eqn:shiftedfs} and Littlewood--Paley theory, the previous is 
\[\lesssim \log(2+\sigma_{1,\iota})^2\log(2+\sigma_{2,\iota})^2 \|f_1\|_p \|f_2\|_q. \]
Consequently,
\[ \|T^{(k)}(f_1,f_2)\|_r \lesssim  \sum_{\iota \in \mathcal{I}} a_\iota \log(2+\sigma_{1,\iota})^{2}\log(2+\sigma_{2,\iota})^{2}  \|f_1\|_p \|f_2\|_q \lesssim |k|^{4} \|f_1\|_p \|f_2\|_q. \]
By interpolation with \eqref{eqn:Tk-l2-decay} we therefore obtain
\[\|T^{(k)} (f_1,f_2)\|_{r} \lesssim 2^{-\delta_{p,q}|k|} \|f_1\|_p \|f_2\|_q \]
for all $p,q\in (1,\infty)$, $r\in [1,\infty)$ with $p^{-1}+q^{-1}=r^{-1}$. Summing over $k\in \mathfrak{F}_\mathrm{H}$ yields the claim.

It remains to prove \eqref{eqn:HHT0}. By a scaling argument it suffices to prove this estimate for $j=0$. Indeed, 
we have
\[ T_j(\Delta^{(1)}_{j+k_1} f_1, \Delta^{(2)}_{2j+k_2} f_2) = D_{(2^j,2^{2j})} T_0(\Delta^{(1)}_{k_1} D_{(2^{-j},2^{-2j})} f_1, \Delta^{(2)}_{k_2} D_{(2^{-j},2^{-2j})} f_2).\]
To apply Theorem \ref{mainresult} we need to perform an additional spatial localization. Choose a smooth non-negative function $\eta$ on $\R^2$ that is supported in a small neighborhood of $[-\tfrac12,\tfrac12]^2$ so that $\sum_{m\in\Z^2} \eta_m = 1$,
where $\eta_m(z) = \eta(z-m)$.
Then 
\[  \| T_0 (\Delta_{k_1}^{(1)} f_1, \Delta_{k_2}^{(2)} f_2 )\|_{L^1(\R^2)} \le \sum_{m\in \Z^2} \int_{\R^2} |T_0 (\Delta_{k_1}^{(1)} f_1, \Delta_{k_2}^{(2)} f_2 )(x,y)|\eta_m(x,y) d(x,y)    \]
Expanding the definition of
\[ T_0 (\Delta_{k_1}^{(1)} f_1, \Delta_{k_2}^{(2)} f_2 )(x,y) \eta_m(x,y)  \]
we obtain the expression
\[\int_{\R}  \widetilde{\eta}_m(x+t,y) \Delta^{(1)}_{k_1} f_1(x+t,y) \widetilde{\eta}_m(x,y+t^2) \Delta^{(2)}_{k_2} f_2(x,y+t^2) \eta_m(x,y) \psi(t) t^{-1} dt,\]
where we have inserted the factors $\widetilde{\eta}_m(x+t,y)$ and $\widetilde{\eta}_m(x,y+t^2)$ which are arranged to equal one for $(x,y,t)$ in the support of $\eta_m(x,y)\psi(t)$. Here the function $\widetilde{\eta}_m$ is an appropriately chosen smooth non-negative function with compact support such that $\|\widetilde{\eta}_m\|_{C^1}\lesssim 1$ and $\sum_{m\in\Z^2} \widetilde{\eta}_m\lesssim 1$.
Define a local operator
\[ \LocTOp (f_1, f_2)(x,y) = \int_\R f_1 (x+t,y) f_2(x,y+t^2) \zeta(x,y,t) dt, \]
where $\zeta$ is a smooth function compactly supported in $\R^2\times (\R\setminus\{0\})$ that we suppress in the notation $\LocTOp(f_1,f_2)$ and which may change from one occurrence of the notation to the next. With this in mind,
\[ \| T_0 (\Delta_{k_1}^{(1)} f_1, \Delta_{k_2}^{(2)} f_2 )\|_{L^1(\R^2)} \le \sum_{m\in\Z^2} \int_{\R^2} |\LocTOp(\widetilde{\eta}_m \Delta^{(1)}_{m} f_1, \widetilde{\eta}_m \Delta^{(2)}_{m} f_2)|  \leq \mathrm{I} + \mathrm{II} + \mathrm{III},\]
where 
\begin{align*}
\mathrm{I} &= \sum_{m\in\Z^2} \int_{\R^2} |\LocTOp(\Delta^{(1)}_{k_1} (\widetilde{\eta}_m f_1), \Delta^{(2)}_{k_2} (\widetilde{\eta}_m f_2)) |,\\
\mathrm{II} &= \sum_{m\in\Z^2} \int_{\R^2} |\LocTOp(\widetilde{\eta}_m \Delta^{(1)}_{k_1} f_1 - \Delta^{(1)}_{k_1} (\widetilde{\eta}_m f_1) , \widetilde{\eta}_m \Delta^{(2)}_{k_2} f_2) \eta_m| ,\\
\mathrm{III} &= \sum_{m\in\Z^2} \int_{\R^2} |\LocTOp(\Delta^{(1)}_{k_1} ( \widetilde{\eta}_m f_1), \widetilde{\eta}_m \Delta^{(2)}_{k_2} f_2 - \Delta^{(2)}_{k_2} (\widetilde{\eta}_m f_2) )  \eta_m| .
\end{align*}

We begin with the main term $\mathrm{I}$. By \eqref{eqn:local-gain1} we obtain
\begin{equation}\label{eqn:local-gain-L2}
\|\LocTOp(\Delta^{(1)}_{k_1} f_1, \Delta^{(2)}_{k_2} f_2)\|_1 \lesssim 2^{-c|k|} \|f_1\|_{2} \|f_2\|_{2}.
\end{equation}
Now we can estimate
\[ \mathrm{I} \lesssim 2^{-c|k|} \sum_{m\in\Z^2} \|\widetilde{\eta}_m f_1\|_2 \|\widetilde{\eta}_m f_2\|_2 %
\lesssim 2^{-c|k|} \|f_1\|_2 \|f_2\|_2, \]
where we have used the Cauchy--Schwarz inequality in $m$ and that $\sum_{m\in\Z} \widetilde{\eta}_m \lesssim 1$.

It remains to consider the error terms $\mathrm{II}$ and $\mathrm{III}$. To treat $\mathrm{II}$ we estimate
\[  |\widetilde{\eta}_m \Delta^{(1)}_{k_1} f_1 - \Delta^{(1)}_{k_1} (\widetilde{\eta}_m f_1)|(x,y) \le \int_\R |(\widetilde{\eta}_m(x,y) - \widetilde{\eta}_m(u,y)) f_1(u,y) \widecheck{\psi_{k_1}}(x-u)| du \]
By the mean value theorem and rapid decay of the Schwartz function $\widecheck{\psi}$ the previous display is
\begin{align*}
\lesssim 2^{-k_1} \int_{\R} |f_1(u,y)| 2^{k_1} (1+2^{k_1} |x-u|)^{-10} du = 2^{-k_1} (|f_1| *_1 \phi_{k_1})(x,y),
\end{align*}
where $\phi_{k_1}(u) = 2^{k_1} (1+2^{k_1}|u|)^{-10}$ and we have made use of the uniform estimate $\|\widetilde{\eta}_m\|_{C^1}\lesssim 1$. Therefore, also using $|\widetilde{\eta}_m \Delta_{k_2}^{(2)} f_2|\lesssim |f_2|*_2\phi_{k_2}$ and $\sum_{m\in\Z^2} \eta_m= 1$,
\[ \mathrm{II} \lesssim 2^{-k_1} \int_{\R^3}\int (|f_1|*_1\phi_{k_1})(x+t,y) (|f_2|*_2\phi_{k_2})(x,y+t^2) |\psi(t) t^{-1}| dt dx\,dy, \]
which by the Cauchy--Schwarz inequality and Young's convolution inequality is
\[ \lesssim 2^{-k_1} \||f_1|*_1\phi_{k_1}\|_2 \||f_2|*_2\phi_{k_2}\|_2 \lesssim 2^{-k_1} \|f_1\|_2 \|f_2\|_2. \]
The same analysis, with the roles of the two coordinates interchanged, gives the estimate
\[ \mathrm{III} \lesssim 2^{-k_2} \|f_1\|_2 \|f_2\|_2. \]

\subsection{Low frequencies}\label{sec:low-red}
By definition we have
\[ T^\mathrm{L} (f_1,f_2)(x,y) = \sum_{j\in\Z} T_j (f_1 *_1 \widecheck{\varphi_j}, f_2 *_2 \widecheck{\varphi_{2j}})(x,y). \]
Expanding the convolutions this becomes
\begin{align*}
 \int_{\R^2} f_1(u,y) f_2(x,v) K(x-u,y-v) du dv,
\end{align*}
where the kernel $K$ takes the form
\[ K(u,v) = \sum_{j\in \mathbb{Z}} 2^{3j} \kappa(2^j u, 2^{2j} v) \] 
for a Schwartz function $\kappa$ given by
\[\kappa(u,v) = \int \widecheck{\varphi}(u+t) \widecheck{\varphi}(v+t^2) \psi(t) t^{-1} dt. \] 
Since $t\mapsto \psi(t) t^{-1}$ has integral zero, the integral of $\kappa$ also vanishes.
This implies that the Fourier transform of $K$ satisfies the standard (anisotropic) symbol estimates
\begin{align}
    \label{sym-est}
    |\partial_\xi^{\alpha}\partial_\eta^{\beta} \widehat{K}(\xi,\eta)| \leq C_{\alpha,\beta} (|\xi| + |\eta|^{1/2})^{-\alpha-2\beta}.
\end{align}
 for all $\alpha,\beta\geq 0$. 
Therefore, Theorem \ref{thm:anisotp} yields that $T^\mathrm{L}$ extends to a bounded operator $L^2\times L^2\to L^1$.

\subsection{Mixed frequencies}\label{sec:mixed-red}
Here we treat the term $T^{\mathrm{M}}$. From the definition,
\[ T^\mathrm{M} (f_1,f_2) = \sum_{j\in\Z} \sum_{k>0} T_j(\Delta^{(1)}_{j+k} f_1, S^{(2)}_{2j+k-101} f_2) + \sum_{j\in\Z} \sum_{k>0} T_j(S^{(1)}_{j+k-101} f_1, \Delta^{(2)}_{2j+k} f_2) = \mathrm{I} + \mathrm{II}. \]
The terms $\mathrm{I}$ and $\mathrm{II}$ will both receive the same treatment as in \S \ref{sec:low-red}. We proceed with the details for $\mathrm{I}$. We again write
\begin{align*}
  \mathrm{I} (x,y) = \int_{\R^2} f_1(u,y) f_2(x,v) K(x-u,y-v) du dv,
\end{align*}
where $K$ is given by
\[ K(u,v) = \sum_{j\in \mathbb{Z}} 2^{3j} \kappa(2^j u, 2^{2j} v) \] 
with
\begin{align*}
    \kappa(u,v) = \sum_{k>0} \int_{\R} (\psi_k)^\vee (u+t) (\varphi_{k-101})^{\vee} (v+t^2) \psi(t)t^{-1} dt. 
\end{align*}
Taking the Fourier transform we have
\begin{align*}
    \widehat{\kappa}(\xi_1, \xi_2) =   \vartheta(\xi)\sum_{k> 0}  \psi(2^{-k}\xi_1) \varphi(2^{-k+101}\xi_2) 
\end{align*}
where $\xi=(\xi_1,\xi_2)$ and
$$\vartheta(\xi) =  \int_{\R} e^{2\pi i (t\xi_1+t^2\xi_2)} \psi(t)t^{-1} dt .$$
Note that $\widehat{\kappa}$ vanishes on a neighborhood of the origin.
Moreover, for $\xi$ in the support of $\widehat{\kappa}$ and $t$ in the support of $\psi$, the derivative $\xi_1+2t\xi_2$ of the phase does not vanish.
Therefore, repeated integration by parts yields
\begin{align}\label{est-statph}
|\partial^{\alpha} \vartheta(\xi)| \lesssim_{\alpha,N} (1+ |\xi|)^{-N}
\end{align}
for all integers $N\ge 0$ and $\alpha\in\mathbb{N}^2_0$. This shows that $\kappa$ is again a Schwartz function with mean zero and the desired bound for $\mathrm{I}$ follows from Theorem \ref{thm:anisotp}. The bound for $\mathrm{II}$ is derived in the same way, interchanging the roles of the two coordinates.

\subsection{Maximal operator}\label{sec:max-red}
To prove Theorem \ref{thm:maxfct} it suffices to prove the claimed bounds for the maximal operator $(f_1,f_2)\mapsto \sup_{j\in\Z} |M_j(f_1,f_2)|$, where
\[ M_j(f_1,f_2)(x,y) = \int_\R f_1(x+t,y) f_2(x,y+t^2) \psi(2^j t) 2^j dt. \]
To do this we decompose as in \eqref{eqn:basicfreqdecomp},
\[ M_j = M_j^\mathrm{L} + M_j^\mathrm{M} + M_j^\mathrm{H}, \]
where for $\omega\in \{\mathrm{L},\mathrm{M},\mathrm{H}\}$,
\[ M_j^\omega(f_1,f_2) = \sum_{k\in\mathfrak{F}_\omega} M_j (\Delta_{j+k_1}^{(1)} f_1, \Delta_{2j+k_2}^{(2)} f_2) \]
and $\mathfrak{F}_\omega$ is defined by \eqref{eqn:basic-freq-dec}. Each of the three components is treated separately. 

We begin with the maximal operators for $M_j^\mathrm{L}$ and $M_j^{\mathrm{M}}$, each of which will be dealt with in the same way. The arguments in \S \ref{sec:low-red} and \S \ref{sec:mixed-red} show that both $M_j^\mathrm{L}$ and $M_j^{\mathrm{M}}$ take the form
\begin{equation}\label{eqn:max-red-1}
\int_{\R^2} f_1(u,y) f_2(x,v) 2^{3j} \kappa(2^j (x-u), 2^{2j} (y-v)) du \,dv
\end{equation} 
with $\kappa$ a fixed Schwartz function. Exploiting the rapid decay of $\kappa$,
\[ |\kappa(u,v)|\lesssim (1+|u|)^{-10} (1+|v|)^{-10}. \]
Thus, \eqref{eqn:max-red-1} is majorized by a constant times
\[ \Big( \int_{\R} |f_1(u,y)| 2^j (1+2^j |x-u|)^{-10} du \Big) \Big( \int_{\R} |f_2(x,v)| 2^{2j} (1+2^{2j} |y-v|)^{-10} dv \Big) \]
\[ \lesssim \mathcal{M}^{(1)} f_1(x,y) \mathcal{M}^{(2)} f_2(x,y), \]
where $\mathcal{M}^{(\ell)}$ denotes the Hardy--Littlewood maximal operator applied in the $\ell$th coordinate. Since both $M_j^\mathrm{L}$ and $M_j^\mathrm{M}$ are of the form \eqref{eqn:max-red-1}, 
H\"older's inequality and the $L^p$ boundedness of $\mathcal{M}^{(\ell)}$ therefore imply
\[ \|\sup_{j\in\Z} |M_j^\mathrm{L}(f_1,f_2)|\|_r +\|\sup_{j\in\Z} |M_j^\mathrm{M}(f_1,f_2)|\|_r \lesssim \|\mathcal{M}^{(1)} f_1\|_p \|\mathcal{M}^{(2)} f_2\|_q \lesssim \|f_1\|_p \|f_2\|_q \]
for all $p,q\in (1,\infty]$, $r\in (0,\infty]$ with $p^{-1}+q^{-1}=r^{-1}$.

It remains to treat the high frequency components $M_j^\mathrm{H}$. The argument in \S \ref{sec:high-red} proves that for each $k\in\mathfrak{F}_\mathrm{H}$ and each $j\in\Z$ we have
\begin{equation}\label{eqn:max-l2-decay}
\|M_j(\Delta_{j+k_1}^{(1)} f_1, \Delta_{2j+k_2}^{(2)} f_2)\|_1 \lesssim 2^{-\delta |k|} \|f_1\|_2 \|f_2\|_2.
\end{equation} 
Indeed, the only difference between $M_j$ and $T_j$ is that $T_j$ features the mean zero bump function $t\mapsto \psi(2^j t)t^{-1}$.
Let us set 
\[M^{(k)}(f_1,f_2) = \sup_{j\in\Z} |M_j(\Delta_{j+k_1}f_1,\Delta_{2j+k_2}f_2)|\]
Then, as in \eqref{eqn:Tk-l2-decay},
\begin{equation}\label{eqn:Mk-l2-decay} \|M^{(k)}(f_1,f_2) \|_1 \lesssim  2^{-\delta|k|} \sum_{j\in\Z} \|\widetilde{\Delta}_{j+k_1}^{(1)} f_1\|_2 \|\widetilde{\Delta}_{2j+k_2}^{(2)} f_2\|_2 \lesssim 2^{-\delta|k|} \|f_1\|_2 \|f_2\|_2.
\end{equation}
To obtain the claimed range of exponents we again make use of the argument in Lemma \ref{lem:shiftedmaxfct}. This gives the pointwise estimate
\begin{equation}\label{eqn:shiftedmax-max}
M^{(k)}(f_1,f_2) \lesssim \sum_{\iota \in \mathcal{I}} a_\iota (\mathscr{M}^{(1)}_{\sigma_{1,\iota}} \mathcal{M}^{(1)} f_1) (\mathscr{M}^{(2)}_{\sigma_{2,\iota}}f_2)
\end{equation}
for all $j\in\Z, k_1\in\Z, k_2\ge 1$ with $a_\iota, \sigma_{\ell,\iota}$ as in Lemma \ref{lem:shiftedmaxfct}. Let $p,q\in (1,\infty]$, $r\in [1,\infty]$ with $p^{-1}+q^{-1}=r^{-1}$. Then \eqref{eqn:shiftedmax-max}  and H\"older's inequality   give
\[ \|M^{(k)} (f_1,f_2)\|_r \lesssim \sum_{\iota\in\mathcal{I}} a_\iota \|\mathscr{M}^{(1)}_{\sigma_{1,\iota}} \mathcal{M}^{(1)} f_1\|_p \|(\mathscr{M}^{(2)}_{\sigma_{2,\iota}}f_2)\|_q \lesssim |k| \|f_1\|_p \|f_2\|_q,  \]
where the last inequality uses \eqref{eqn:shiftedmaxfct-coeff}, \eqref{eqn:shiftedmaxfct-basic} and the Hardy--Littlewood maximal theorem. Interpolation with \eqref{eqn:Mk-l2-decay} shows
\[ \|M^{(k)}(f_1,f_2)\|_r \lesssim 2^{-\delta_{p,q}|k|} \|f_1\|_p \|f_2\|_q \]
for all $p,q\in (1,\infty)$, $r\in [1,\infty)$ with $p^{-1}+q^{-1}=r^{-1}$ and summing over $k\in\mathfrak{F}_\mathrm{H}$ gives the claim.

\section{A trilinear smoothing  inequality}\label{sec:localop}
In this section we prove Theorem \ref{mainresult}.
In place of the trilinear form we find it convenient to dualize and instead work with the bilinear operator
    \begin{equation}\label{eqn:locop-def2}
\LocTOp (f_1, f_2)(x,y) = \int_\R f_1 (x+t,y) f_2(x,y+t^2) \zeta(x,y,t) dt,
\end{equation}
where $\zeta$ a smooth function that we assume is compactly supported in $\R^2\times (0,\infty)$ (the restriction to $t>0$ is no loss of generality). 

We first claim that it suffices to prove existence of $\sigma>0$ so that for all $\lambda\ge 1$,
\begin{equation}\label{eqn:decay-linf} \|\LocTOp(f_1,f_2)\|_1 \lesssim \lambda^{-\sigma} \|f_1\|_\infty \|f_2\|_\infty 
\end{equation}
holds for all $f_j$ so that $\widehat{f_j}$ is supported where $|\xi_j|\asymp \lambda$ for at least one $j=1,2$. 
To prove this claim observe the inequality
\[ \|\LocTOp( f_1,f_2)\|_1 \lesssim \|f_1\|_{3/2} \|f_2\|_{3/2},  \]
valid for arbitrary test functions $f_1,f_2$. To see this, begin with the triangle inequality, Fubini's theorem and a change of variables to see
\[ \|\LocTOp(f_1, f_2)\|_1 \lesssim \int_{\R^2} |f_1(x,y)| \Big( \int_{\R} |f_2(x+t, y+t^2)| \eta(t) dt \Big ) d(x,y),\]
where $\eta$ is a smooth non-negative function compactly supported in $\R\setminus\{0\}$. By H\"older's inequality the previous is
\[\le \|f_1\|_{3/2} \Big( \big|\int_{\R} |f_2(x+t, y+t^2)| \eta(t) dt\big|^3\Big)^{1/3} \lesssim \|f_1\|_{3/2} \|f_2\|_{3/2},  \]
where we have used the  fact that averages along a parabola map $L^{3/2}\to L^3$ (see \cite{Str70}).
By interpolation with \eqref{eqn:decay-linf} and duality we then obtain the claim \eqref{eqn:local-gain1}.

We will now prove \eqref{eqn:decay-linf}. 
Assume that
\begin{equation}\label{eqn:localasm-strict}
\widehat{f_{\ell_*}}(\xi_1,\xi_2)\not=0\quad\Longrightarrow\quad \lambda\le |\xi_{\ell_*}| \le 2\lambda
\end{equation}
either holds for $\ell_*=1$ or $\ell_*=2$.
If $\ell_*=1$, then we will make the additional assumption that
\begin{equation}\label{eqn:localasm-upper-2}
\widehat{f_2}(\xi_1,\xi_2)\not=0\quad\Longrightarrow\quad |\xi_2| \le 2\lambda.
\end{equation} 
This is without loss of generality as can be seen by a Littlewood-Paley decomposition.

We will make use of \eqref{eqn:localasm-strict} only at the end of the argument. Also assume without loss of generality that 
\[ \|f_1\|_\infty = \|f_2\|_\infty = 1. \]
The first step is a further spatial localization.
Let $\eta$ be a smooth and non-negative function that is compactly supported in a small neighborhood of $[-\tfrac12,\tfrac12]^2$ and satisfies $\sum_{m\in\Z^2} \eta((x,y)-m)=1$ for all $(x,y)\in\R^2$. Choose $\gamma\in (\tfrac12,1)$. For $\ell\in \{1,2\}$ decompose
\begin{equation}\label{eqn:flmdef}
f_\ell(x,y) = \sum_{m\in\mathbb{Z}^2} f_{\ell,m}(x,y)
\ \text{ with } \ 
f_{\ell,m}(x,y) = f_\ell(x,y) \eta(\lambda^{\gamma}(x,y)-m).
\end{equation}
The support of the function $f_{\ell,m}$ is contained in a small neighborhood of the cube of sidelength $\lambda^{-\gamma}$ centered at $\lambda^{-\gamma}m\in \lambda^{-\gamma}\Z^2$. Let us denote by $Q_m$ the cube of sidelength $2\lambda^{-\gamma}$ centered at $\lambda^{-\gamma}m$. Let $\mathfrak{I}$ denote the set of all pairs $\mathbf{m}=(m_1,m_2)\in (\Z^2)^2$ so that
\[ \|\LocTOp(f_{1,m_1}, f_{2,m_2})\|_1 \not= 0. \]
Also let $\mathfrak{I}_\ell\subset\Z^2$ denote the set of $m_\ell$ so that there exists $\mathbf{m}=(m_1,m_2)\in\mathfrak{I}$. The compact support of $\zeta$ implies $\#\mathfrak{I_\ell}=O( \lambda^{2\gamma})$.
We claim that inspection of \eqref{eqn:locop-def2} yields
\begin{equation}\label{eqn:interactingcount}
\# \mathfrak{I} = O(\lambda^{3\gamma}).
\end{equation}
Indeed, compact support of $\zeta$ already implies $\#\mathfrak{I}=O(\lambda^{4\gamma})$. To see the additional saving of $O(\lambda^{\gamma})$, suppose $(x+t,y)\in Q_{m_1}$ and $(x,y+t^2)\in Q_{m_2}$. Writing $t^2=(y+t^2) - y = ( (x+t) - x)^2$ we obtain the relation
\[ \lambda^{-\gamma}(m_{2,2} - m_{1,2}) - \lambda^{-2\gamma}(m_{1,1} - m_{2,1})^2  = O(\lambda^{-\gamma}),\]
which implies that say, with $m_1$ and $m_{2,1}$ given (for which there are $O(\lambda^{3\gamma})$ choices), the integer $m_{2,2}$ is determined up to $O(1)$.

\subsection{A basic estimate}\label{sec:local-basicest} 
We use the triangle inequality to estimate
\[ \|\LocTOp(f_1, f_2)\|_1 \le \sum_{\mathbf{m}\in\mathfrak{I}} \| \LocTOp(f_{1,m_1}, f_{2,m_2}) \|_1. \]

Next, we apply the Cauchy--Schwarz inequality to the integration in $(x,y)$ to find that the quantity $\|\LocTOp(f_{1,m_1}, f_{2,m_2})\|_1^2$ is dominated by a constant times
\begin{align}\label{eqn:loc-cs}
\lambda^{-2\gamma}\int_{\R^4} & f_{1,m_1}(x+t+s,y)\overline{f_{1,m_1}(x+t,y)}  \\\nonumber
 & f_{2,m_2}(x,y+(t+s)^2)  \overline{f_{2,m_2}(x,y+t^2)} \zeta_{\mathbf{m}}(x,y,t,s) dx \,dy \,dt ds
\end{align}
where $\zeta_{\mathbf{m}}$ is a smooth non-negative function compactly supported in a cube with side lengths $O(\lambda^{-\gamma})$ contained in $\R^2\times (0,\infty)^2$, which satisfies 
\[\norm{\partial^\alpha\zeta_\mathbf{m}}_{C^0} \lesssim 
\lambda^{\gamma|\alpha|}\]
for each multiindex $\alpha\in\mathbb{N}_0^4$. 
For later use we also fix for each $\mathbf{m}\in\mathfrak{I}$ a point
\begin{equation}\label{cubepointofref}
(\overline{x},\overline{y},\overline{t})=(\overline{x}_\mathbf{m},\overline{y}_\mathbf{m},\overline{t}_\mathbf{m})\quad\text{such that}\quad(\overline{x}+\overline{t},\overline{y})\in Q_{m_1}, (\overline{x},\overline{y}+\overline{t}^2)\in Q_{m_2}.
\end{equation}
Then for each $(x,y,t,s)$ in the support of $\zeta_\mathbf{m}$,
\[ |x-\overline{x}| + |y-\overline{y}| + |t-\overline{t}| \lesssim \lambda^{-\gamma}. \]

If $(x+t+s,y)$ and $(x+t,y)$ both lie in the support of $f_{1,m_1}$
then $|s|=|(x+t+s)-(x+t)| = O(\lambda^{-\gamma})$.
Now by the mean value theorem, for $|s| = O(\lambda^{-\gamma})$ we have
\[ f_{2,m_2}(x,y+(t+s)^2) = f_{2,m_2}(x,y+t^2 + 2s\overline{t}) + O(\lambda^{-2\gamma} \|\partial_2 f_{2,m_2}\|_\infty). \] 
By \eqref{eqn:localasm-upper-2} and \eqref{eqn:flmdef} we have
\[ \|\partial_2 f_{2,m_2}\|_\infty \lesssim \lambda \|f_2\|_\infty=\lambda.  \]
This uses $\gamma\le 1$. Define
\[\delta = \gamma-\tfrac12. \]
Note that $\delta>0$ since $\gamma>\tfrac12$.
For a function $f$ on $\R^2$ let us denote
\begin{align*}%
  \md_s^{(1)}f(x,y) = f(x+s,y) \overline{f(x,y)},\quad
   \md_s^{(2)}f(x,y) = f(x,y+s) \overline{f(x,y)}.
\end{align*}
Therefore we can write%
\[ f_{1,m_1}(x+t+s,y)\overline{f_{1,m_1}(x+t,y)} = \md_s^{(1)} f_{1,m_1}(x+t,y),\]
\[ f_{2,m_2}(x,y+(t+s)^2)\overline{f_{2,m_2}(x,y+t^2)} = \md_{2s\overline{t}}^{(2)}f_{2,m_2} (x,y+t^2) + O(\lambda^{-2\delta}).\]
Thus in all, $\|\LocTOp(f_1,f_2)\|_1$ 
is majorized by a constant times
\begin{equation} \label{mainterm1} 
\lambda^{-\gamma} \sum_{\mathbf{m}\in\mathfrak{I}}  \Big|
\int_{|s|=O(\lambda^{-\gamma})} \Big(\int_{\mathbb{R}^3}
\md_s^{(1)}f_{1,m_1}(x+t,y)
\md_{2s\overline{t}}^{(2)}f_{2,m_2}(x,y+t^2)
\zeta_\mathbf{m}(x,y,t,s)\,dx\,dy\,dt\Big) \,ds\Big|^{\frac12} \end{equation}
plus a constant times
\begin{equation} \label{latterterm} \lambda^{-\gamma} \sum_{\mathbf{m}\in\mathfrak{I}}  \Big(
\int_{|s|=O(\lambda^{-\gamma})} \Big(\int_{\mathbb{R}^3}
\lambda^{-2\delta}
\zeta_\mathbf{m}(x,y,t,s)\,dx\,dy\,dt\Big) \,ds\Big)^{\frac12}. \end{equation}
The term \eqref{latterterm}  is $O(\lambda^{-\delta})$,
since $\#\mathfrak{I}=O(\lambda^{3\gamma})$
and the support of each auxiliary function $\zeta_\mathbf{m}$ 
contained in a cube in $\mathbb{R}^4$ of side length $O(\lambda^{-\gamma})$. 
It remains to estimate the main term \eqref{mainterm1}. 
The Cauchy-Schwarz inequality applied to the summation over $\mathbf{m}$ gives a majorization
for \eqref{mainterm1} of the form
\begin{equation}\label{eqn:sumaftercs}
  \lambda^{\gamma/2} \big( 
\int_{|s|=O(\lambda^{-\gamma})}
\sum_{\mathbf{m}\in\mathfrak{I}} 
\big| \int_{\mathbb{R}^3} \md_s^{(1)}f_{1,m_1}(x+t,y) \md_{2s\overline{t}}^{(2)}f_{2,m_2}(x,y+t^2)
\zeta_\mathbf{m}(x,y,t,s)\,dx\,dy\,dt\big|\,ds\big)^{1/2}.    
\end{equation} 
Fix $\mathbf{m}\in\mathfrak{I}$ and any parameter $s$
satisfying $|s| = O(\lambda^{-\gamma})$, and consider
\begin{equation} \label{next1} 
\int_{\mathbb{R}^3} \md_s^{(1)}f_{1,m_1}(x+t,y) \md_{2s\overline{t}}^{(2)}f_{2,m_2}(x,y+t^2)
\zeta_\mathbf{m}(x,y,t,s)\,dx\,dy\,dt.
\end{equation}

Since for every $s\in\R$ the function $\md_s^{(\ell)} f_{\ell,m_\ell}$ is supported in a cube of side length $\tfrac32\lambda^{-\gamma}$ centrally contained in $Q_{m_\ell}$, we can write it in terms of local Fourier series as %
\begin{equation} \label{localFourierseries}
\md_s^{(\ell)}f_{\ell,m_\ell}(x,y) = \tilde\eta_{m_\ell}(x,y) 
\sum_{k\in\mathbb{Z}^2} a_{\ell,m_\ell,k,s} e^{\pi i \lambda^\gamma k\cdot (x,y)},\quad(\ell=1,2)
\end{equation}
with $\tilde\eta_{m}$ a smooth function that equals one on $Q_{m}$, is supported on a small neighborhood of $Q_{m}$ and satisfies $|\partial^\alpha \tilde\eta_{m}| = O(\lambda^{\gamma|\alpha|})$
for all $\alpha\in\mathbb{N}_0^2$ and $m\in\Z^2$. The Fourier coefficients are given by
\[ a_{\ell,m_\ell,k,s} = \tfrac14 \lambda^{2\gamma} \widehat{\md_s^{(\ell)}f_{\ell,m_\ell}}(\tfrac12 \lambda^{\gamma} k). \]
By Parseval's theorem,
\begin{equation} \label{l2control}
\sum_{k\in\mathbb{Z}^2} |a_{\ell,m_\ell,k,s}|^2 \lesssim \lambda^{2\gamma} \|\md_s^{(\ell)}f_{\ell,m_\ell}\|_2^2 \lesssim 1. \end{equation}
with implicit constants uniform in $\lambda,m_\ell,s$.

As a consequence of the band-limitedness hypotheses \eqref{eqn:localasm-upper-2}
on $f_2$, 
a stronger bound holds for $\ell=2$ when $k$ is large in an appropriate sense. 
Write $k = (k_1,k_2)$. Integration by parts gives
\[ \sum_{k_1\in\mathbb{Z}} |a_{2,m_2,k,s}|^2 \lesssim_{N} \lambda^{(1-\gamma)N} |k_2|^{-N} \]
for every $N\ge 1$, $k_2\not=0$. Fix a small $\epsilon_1>0$. 
Then from the previous display,
\begin{equation}\label{gainfromfreqsupport}
\sum_{|k_2|\ge \lambda^{1-\gamma+\epsilon_1}} \sum_{k_1\in\mathbb{Z}} |a_{2,m_2,k,s}|^2 \lesssim_{N,\epsilon_1} \lambda^{-N}.
\end{equation} 
To further estimate \eqref{next1} we majorize it using \eqref{localFourierseries} and \eqref{gainfromfreqsupport} by
\[O(\lambda^{-N}) + \sum_{\substack{\mathbf{k}\in(\Z^2)^2,\\|k_{2,2}|\le \lambda^{1-\gamma+\epsilon_1}}} |a_{1,m_1,k_1,s} a_{2,m_2,k_2,2s\overline{t}}|\; \Big|\int_{\mathbb{R}^3} e^{\pi i \lambda^{\gamma} (k_1\cdot (x+t,y)+ k_2\cdot (x,y+t^2))} 
\zeta_\mathbf{m}(x,y,t,s)\,dx\,dy\,dt\Big|,\]
where $\mathbf{k}=(k_1,k_2)=((k_{1,1},k_{1,2}), (k_{2,1}, k_{2,2}))\in (\Z^2)^2$.
The gradient of the phase function 
\[((x,y),t)\mapsto \lambda^{\gamma}(k_1\cdot (x+t,y)+k_2\cdot (x,y+t^2))\]
is equal to 
\begin{equation}\label{eqn:gradient}
\lambda^\gamma (k_1+k_2, k_{1,1} + 2t k_{2,2}) = \lambda^\gamma (k_1+k_2, k_{1,1} + 2\overline{t} k_{2,2}) + O(\lambda^{1-\gamma+\epsilon_1}),
\end{equation}
where $\overline{t}=\overline{t}_\mathbf{m}$ is as in \eqref{cubepointofref} and we have used $|k_{2,2}|\le \lambda^{1-\gamma+\epsilon_1}$.
The remainder term $O(\lambda^{1-\gamma+\epsilon_1})$ on the right-hand side of \eqref{eqn:gradient}
is small relative to the factor $\lambda^\gamma$, since $\gamma > \tfrac12$ and $\epsilon_1>0$ is small.
Let $\epsilon_2>0$ be a small exponent
to be chosen below. Integration by parts therefore gives
\[
\Big|\int_{\mathbb{R}^3} e^{\pi i \lambda^{\gamma} (k_1\cdot (x+t,y)+ k_2\cdot (x,y+t^2))} 
\zeta_\mathbf{m}(x,y,t,s)\,dx\,dy\,dt\Big| \lesssim_N \lambda^{-N}
\]
for all $N\ge 0$ unless 
\begin{equation}\label{eqn:stationary1}
|k_1+k_2|\lesssim \lambda^{\epsilon_2}\quad\text{and}\quad|k_{1,1}+2\overline{t}k_{2,2}|\lesssim \lambda^{\epsilon_2}.
\end{equation}
On the other hand, if \eqref{eqn:stationary1} holds, then $k_{1,1}$ determines $k_{1,2}$ and $k_2$ up to additive ambiguity $O(\lambda^{\epsilon_2})$. In the same way, $k_{2,2}$ determines $k_1,k_{2,1}$ up to $O(\lambda^{\epsilon_2})$.
Thus far, we have shown that \eqref{next1} is majorized by
\[
O(\lambda^{-N}) + O(\lambda^{-3\gamma})\sum_{\substack{\mathbf{k}\in (\Z^2)^2,\\ \eqref{eqn:stationary1}\text{ holds}}} |a_{1,m_1,k_1,s} a_{2,m_2,k_2,2s\overline{t}}|
\]
for every $N\ge 0$.
By the Cauchy--Schwarz inequality, \eqref{l2control} and the previous display, we have now shown that for $\ell=1$,
\begin{equation} \label{graphappears1}
|\eqref{next1}|\lesssim \lambda^{-N} + \lambda^{-3\gamma+2\epsilon_2} 
\big( \sum_{\substack{k_1\in \Z^2,\\ |2\overline{t}k_{1,2}-k_{1,1}|=O(\lambda^{\epsilon_2})}} |a_{1, m_1, k_1, s}|^2 \big)^{\frac12}
\end{equation}
and for $\ell=2$,
\begin{equation} \label{graphappears2}
|\eqref{next1}|\lesssim \lambda^{-N} + \lambda^{-3\gamma+2\epsilon_2} 
\big( \sum_{\substack{k_2\in \Z^2,\\ |2\overline{t}k_{2,2}-k_{2,1}|=O(\lambda^{\epsilon_2})}} |a_{2, m_2, k_2, 2s\overline{t}}|^2 \big)^{\frac12}.
\end{equation}
Plugging these estimates into \eqref{eqn:sumaftercs}, applying the Cauchy--Schwarz inequality in $s$ and $\mathbf{m}$ and keeping in mind \eqref{latterterm} gives for each $\ell\in\{1,2\}$,
\begin{equation}\label{eqn:first-penult}
\|\LocTOp(f_1,f_2)\|_1 \lesssim \lambda^{-\delta}+ %
\lambda^{-\frac{\gamma}2+\epsilon_2}\big( 
\int_{|s|=O(\lambda^{-\gamma})}
\sum_{\mathbf{m}\in\mathfrak{I}} \sum_{\substack{k_\ell\in \Z^2,\\ |2\overline{t}_\mathbf{m}k_{\ell,2}-k_{\ell,1}|=O(\lambda^{\epsilon_2})}} |a_{\ell, m_\ell, k_\ell, s}|^2 \,ds\big)^{1/4}.
\end{equation} 
Note that in the case $\ell=2$ this is obtained from \eqref{graphappears2} by a change of variables $2s\overline{t}_{\mathbf{m}}\mapsto s$ which only incurs a Jacobian factor of $O(1)$.

We pause to point out a certain novelty in \eqref{eqn:stationary1}. 
Rather than deducing that a favorable bound holds unless a function has a 
significantly large Fourier coefficient,
as in the work of Roth \cite{Rot53},
we have deduced such a bound unless the $\ell^2$
norm of the restriction of the Fourier transform to (the union of
a relatively small number of cosets of) a certain rank $1$ subgroup is significantly large.

For the moment suppose that $\ell=1$.
Writing $\mathbf{m}=(m_1,m_2)$ and holding $m_1$ fixed there are $O(\lambda^{\gamma})$ values of $m_2$ so that $(m_1,m_2)\in\mathfrak{I}$.
Moreover, holding also $k_1=(k_{1,1},k_{1,2})\in\Z^2$ fixed we claim that there are at most
\begin{equation}\label{eqn:improvedmbound}
1+O(\lambda^{\gamma+\epsilon_2} |k_{1,1}|^{-1})
\end{equation} 
choices of $m_2$ such that
\begin{equation}\label{eqn:k1condition}
|2\overline{t}_\mathbf{m}k_{1,2}-k_{1,1}|=O(\lambda^{\epsilon_2}).
\end{equation} 
To see this note that \eqref{eqn:k1condition} is equivalent to $|2k_{1,2}-(\overline{t}_\mathbf{m})^{-1}k_{1,1}|=O(\lambda^{\epsilon_2})$ since $\overline{t}$ is bounded from above and below by positive constants. Then, from the mean value theorem we see that changing $m_2$ by $1$ effects a change of $\asymp \lambda^{-\gamma}|k_{1,1}|$ in $|2k_{1,2}-(\overline{t}_\mathbf{m})^{-1}k_{1,1}|$, which implies the claim.
In the case $\ell=2$ we argue similarly that for every fixed $m_2$ and $k_2=(k_{2,1},k_{2,2})$, there are at most 
\begin{equation}\label{eqn:improvedmbound2}
1+ O(\lambda^{\gamma+\epsilon_2} |k_{2,2}|^{-1})
\end{equation}
choices of $m_1$ such that $|2\overline{t}_\mathbf{m}k_{2,2}-k_{2,1}|=O(\lambda^{\epsilon_2})$.
The bound \eqref{eqn:improvedmbound} (resp. \eqref{eqn:improvedmbound2}) is an improvement over $O(\lambda^{\gamma})$ if $|k_{1,1}|$ (resp. $|k_{2,2}|$) is large enough.
To quantify this we introduce another auxiliary exponent $\kappa\in (\epsilon_2, 1-\gamma)$. %
Then for each fixed $s$ and $\ell\in \{1,2\}$,
\[
\sum_{\mathbf{m}\in\mathfrak{I}} \sum_{\substack{k_\ell\in \Z^2,\\ |2\overline{t}_\mathbf{m}k_{\ell,2}-k_{\ell,1}|=O(\lambda^{\epsilon_2})}} |a_{\ell, m_\ell, k_\ell, s}|^2 \lesssim \lambda^\gamma \sum_{m_\ell\in\mathfrak{I}_\ell} \sum_{\substack{k_\ell\in \Z^2,\\ |k_{\ell,\ell}|\le \lambda^{\kappa}}} |a_{\ell, m_\ell, k_\ell, s}|^2 + \mathrm{R}_{s,\ell},
\]
where %
\[\mathrm{R}_{s,\ell} = (1+\lambda^{\gamma+\epsilon_2-\kappa}) \sum_{m_\ell\in\mathfrak{I}_\ell} \sum_{\substack{k_\ell\in \Z^2}} |a_{\ell, m_\ell, k_\ell, s}|^2 \lesssim \lambda^{3\gamma+\epsilon_2-\kappa}, \]
where the last inequality uses \eqref{l2control} and $\#\mathfrak{I}_\ell\lesssim \lambda^{2\gamma}$. Combining this with \eqref{eqn:first-penult} we have now proved that for $\ell=1,2$, %
\begin{equation}\label{eqn:first-final}
 \|\LocTOp(f_1,f_2)\|_{1} \lesssim \lambda^{\frac12-\gamma} + \lambda^{-\frac{\kappa}4 + \frac54\epsilon_2} + \lambda^{-\frac{\gamma}4+\epsilon_2} \big( 
\int_{|s|=O(\lambda^{-\gamma})}
\sum_{m_\ell\in\mathfrak{I}_\ell} \sum_{\substack{k_\ell\in \Z^2,\\ |k_{\ell,\ell}|\le \lambda^{\kappa}}} |a_{\ell, m_\ell, k_\ell, s}|^2 \,ds\big)^{1/4}.
\end{equation}

\noindent {\em Remark.} The argument that gives \eqref{eqn:improvedmbound}, \eqref{eqn:improvedmbound2} works in both components, so the main term could be modified by replacing $|k_{\ell,\ell}|\le \lambda^\kappa$ with $|k_\ell|\le \lambda^{\kappa}$ but we won't need that.

\subsection{A structural decomposition} \label{section:structure}
We begin with a decomposition of an $L^2(\R^d)$ function $f$ that relates $f$ to $\md_sf$,  where  
 $\md_sf(x) = f(x+s) \overline{f(x)}$. 
This will be stated in $\R^d$, though we will only make use of it in the case $d=1$.
\begin{lemma} \label{lemma:structure}
Let $f\in L^2(\R^d)$, $\rho\in (0,1)$ and $R>0$.
Suppose that
\[ \int_{\mathbb{R}^d} \int_{|\xi|\le R}
|\widehat{\md_sf}|^2(\xi) \,d\xi\,ds
\ge \rho \norm{f}_{{L^2}}^4.\]
Then there exists an orthogonal decomposition $f=g+h$ 
with $\widehat{g}$ supported in some ball of radius $R$, 
$g\perp h$, and $\norm{g}_{{L^2}} \ge \tfrac12 \rho^{1/2} \norm{f}_{{L^2}}$.
\end{lemma}

\begin{proof}
We compute
\[ \widehat{\md_s f}(\xi) = \int_{\R^d} e^{2\pi i s\cdot (\xi + \xi')} \widehat{f}(\xi+\xi') \overline{\widehat{f}} (\xi') d\xi' \]
and hence
\[ |\widehat{\md_sf}(\xi)|^2
= \int_{\R^d} \int_{\R^d}
e^{2\pi is\cdot (\xi'-\xi'')}
\widehat{f}(\xi+ \xi') \overline{\widehat{f}}(\xi')
\overline{\widehat{f}}(\xi+ \xi'') \widehat{f}(\xi'')
\,d\xi'\,d\xi''
\]
whence
\[ \int_{\mathbb{R}^d} |\widehat{\md_sf}(\xi)|^2\,ds
= \int_{\R^d} 
|\widehat{f}(\xi+ \xi')|^2 |\widehat{f}(\xi')|^2
\,d\xi'\]
and finally
\[
\begin{aligned}
\int_{\mathbb{R}^d} \int_{|\xi|\le R} |\widehat{\md_sf}(\xi)|^2\,d\xi\,ds
&=  \iint_{|\xi-\xi'|\le R} 
|\widehat{f}(\xi)|^2 |\widehat{f}(\xi')|^2
\,d\xi\,d\xi'
\\&
\le \norm{f}_{{L^2}}^2 
\ \sup_{B} \int_B |\widehat{f}|^2.
\end{aligned}\]
where supremum is over all balls $B$ of radius $R$ in $\R^d$.
Choose $B$ to essentially realize this supremum. The desired decomposition is obtained by defining $g\in L^2$   via $\widehat{g} = \mathbf{1}_B \widehat{f}$ and $h=f-g$.
\end{proof}

\begin{lemma} \label{lemma:sharpflat}
Let $R\ge 1$ and $\varrho\in (0,1)$. Let $f\in{L^2}(\mathbb{R})$. 
There exists a decomposition
\[ f = f_\sharp + f_\flat \]
with the following properties.
 \begin{enumerate}
     \item One has
     \[
     \norm{f_\sharp}_{L^2}+\norm{f_\flat}_{L^2} \lesssim  \norm{f}_{L^2}
     \]
     
\item The function $f_\sharp$ admits a decomposition 
\[
f_\sharp(x) = \sum_{n=1}^{\mathcal{N}} h_n(x) e^{i \alpha_{n}x}
\]
with each  $\alpha_{n}\in\mathbb{R}$, 
$h_n$ a smooth function satisfying $\|\partial^N h_n\|_{\infty}\lesssim_N R^N \|f\|_\infty$ for all integers $N\ge 0$ and 
$\norm{h_{n}}_{L^2} \lesssim \norm{f}_{{L^2}}$,
$h_n$ is Fourier supported in $[-R,R]$,
and $\mathcal{N}\lesssim \varrho^{-1}. $ 
Moreover, the support of $\widehat{f_\sharp}$ is contained in the support of $\widehat{f}$.
\item One has the bound
\begin{equation} \label{gflat}
\int_{\mathbb{R}} \int_{|\xi| \le R} 
|\widehat{\md_s f_\flat}(\xi)|^2\,d\xi\,ds
\lesssim \varrho\norm{f}_{{L^2}}^4.
\end{equation}
 \end{enumerate}
All implicit constants do not depend on $R,\varrho,f$.
\end{lemma}

\begin{proof}
Let $\varphi\in C^\infty_0(\mathbb{R})$
be supported in $(-1,1)$ and satisfy
$\sum_{n\in\mathbb{Z}} \varphi(\xi+n)=1$ for every $\xi\in\mathbb{R}$.  Denote $\varphi_n(\xi)=\varphi(\xi+n)$. Let  $\mathfrak{N}$ be the set of all
   $n\in \Z$ for which
there exists an interval $I$ of length $R$ intersecting support of $\varphi_n(R^{-1}\cdot)$ so that
\[
\int_{I} |\widehat{f}(\xi)|^2\,d\xi \ge \varrho \norm{f}_{{L^2}}^2.
\]
Observe that $\#\mathfrak{N}$ is $O(\varrho^{-1})$.
The desired decomposition is obtained by defining $f_\flat$ and $f_\sharp$ via
\begin{align*}&\widehat{f_\sharp}(\xi) = \sum_{n\in \mathfrak{N}} \varphi_n(\lambda^{-\tau}\xi)
\widehat{f}(\xi)
, \quad \widehat{f_\flat}(\xi) = \widehat{f}(\xi)- \widehat{f_\sharp}(\xi) = \sum_{n\in \Z\setminus \mathfrak{N}} \varphi_n(\lambda^{-\tau}\xi)
\widehat{f}(\xi).  
\end{align*}

Indeed, we claim that Lemma \ref{lemma:structure} implies that $f_\flat$ satisfies \eqref{gflat}. Suppose not.
Then by the lemma applied with $\rho=c \varrho$ there exists an orthogonal decomposition $f_\flat = g + h$ with $\widehat{g}$ supported on an interval $I$ of length $R$ so that
\[ \varrho^{1/2}\ge \|\widehat{f_\flat}\mathbf{1}_I\|_2\ge \|g\|_2\ge \tfrac 12{c}^{\frac12} \varrho^{1/2}  \]
This is a contradiction if $c$ is large.
\end{proof}

\subsection{Applying the structural decomposition}\label{sec:local-applystructure}

Let us use the decomposition from Lemma \ref{lemma:sharpflat} to refine the initial spatial decomposition from \eqref{eqn:flmdef}. For $\ell\in \{1,2\}$, choose smooth functions $\psi^{(\ell)}$, each of the form $\lambda\psi_0(\lambda \cdot)$ so that $f_\ell *_\ell \psi^{(\ell)}=f_\ell$. For $m\in\Z^2$ denote
\[ \eta_m (x) = \eta(\lambda^{\gamma} x - m) \]
Then with $f_{\ell,m} = \psi^{(\ell)} *_\ell (\eta_{m} f_\ell)$, 
\[ f_\ell = \sum_{m\in \Z^2} \eta_{m} f_\ell = \sum_{m\in\Z^2} f_{\ell,m} + \sum_{m\in \Z^2} \eta_{m} (\psi^{(\ell)}*_\ell f_\ell) - \psi^{(\ell)}*_\ell (\eta_{m} f_\ell). \] %
By the mean value theorem, the summand in the second sum on the right hand side is pointwise bounded by a constant times $\lambda^{\gamma-1}$.

Fix $\tau = \gamma+\kappa$ and a small parameter $\delta'>0$.
Suppose for the moment that $\ell=1$.
Fix $y\in\R$. Applying Lemma \ref{lemma:sharpflat} to each function $x\mapsto f_{1,m}(x,y)$, we obtain a decomposition
\[ f_{1,m} = f_{1,m,\flat} + f_{1,m,\sharp}, \]
where
\[
f_{1,m,\sharp}(x,y) = \sum_{n=1}^\mathcal{N} h_{1,n,m}(x,y) e^{i \alpha_{n,m}(y) x}
\]
with $\mathcal{N}=C\lambda^{\delta'}$, $\alpha_{n,m}$ measurable real-valued functions (which satisfy $|\alpha_{n,m}|\asymp \lambda$ if \eqref{eqn:localasm-strict} holds for $\ell_*=1$) and $h_{1,n,m}$ measurable functions with smooth fibers $h_{1,n,m}(\cdot,y)$ satisfying $|\partial_x^N h_{1,n,m}|\lesssim_N \lambda^{\tau N}$ (uniformly in $n,m$), 
and %
\[
\int_{\mathbb{R}} \int_{\R^2} \mathbf{1}_{|\xi_1| \le \lambda^\tau} 
|\widehat{\md^{(1)}_s f_{1,m,\flat}}(\xi)|^2\,d\xi\,ds
\lesssim 
\lambda^{-\deltaprime-3\gamma}. %
\]
Similarly, for $\ell=2$ we fix $x\in\R$ and apply Lemma \ref{lemma:sharpflat} to each function $y\mapsto f_{2,m}(x,y)$ to obtain a decomposition
\[ f_{2,m} = f_{2,m,\flat} + f_{2,m,\sharp}, \]
where
\[
f_{2,m,\sharp}(x,y) = \sum_{n=1}^{\mathcal{N}} h_{2,n,m}(x,y) e^{i \beta_{n,m}(x) y}
\]
with $\beta_{n,m}$ measurable real-valued functions satisfying $|\beta_{n,m}|\lesssim \lambda$ by \eqref{eqn:localasm-upper-2} (and $|\beta_{n,m}|\asymp \lambda$ if \eqref{eqn:localasm-strict} holds with $\ell_*=2$) and $h_{2,n,m}$ measurable functions with smooth fibers $h_{2,n,m}(x,\cdot)$ satisfying $|\partial_y^N h_{2,n,m}|\lesssim_N \lambda^{\tau N}$ (uniformly in $n,m$), 
and %
\[
\int_{\mathbb{R}} \int_{\R^2} \mathbf{1}_{|\xi_2| \le \lambda^\tau} 
|\widehat{\md^{(2)}_s f_{2,m,\flat}}(\xi)|^2\,d\xi\,ds
\lesssim 
\lambda^{-\deltaprime-3\gamma}. %
\]
It is convenient to spatially localize the functions $f_{\ell,m,\flat}, f_{\ell,m,\sharp}$.
Let $\widetilde{\eta}$ denote a smooth function that equals one on the support of $\eta$ and has an only slightly larger support than $\eta$. Denote $\widetilde{\eta}_m(x,y) = \widetilde{\eta}(\lambda^\gamma (x,y) -m)$.
Then write
\[ f_{\ell,m} = \widetilde{\eta}_{m} f_{\ell,m,\flat} + \widetilde{\eta}_{m} f_{\ell,m,\sharp} + (1-\widetilde{\eta}_m) (\psi^{(\ell)}*_\ell (\eta_{m} f_\ell)) \]
From the mean value theorem, 
 the third summand on the right-hand side enjoys a pointwise gain of $O(\lambda^{\gamma-1})$.
To summarize, we have arrived at a decomposition
\[
	f_\ell = f_{\ell,\flat} + f_{\ell,\sharp} + f_{\ell,\mathrm{err}}, 
	\]
	with
	\begin{align}
	 f_{\ell,\flat} &= \sum_{m\in\Z^2} \widetilde{\eta}_{m} f_{\ell,m,\flat},
	\label{eqn:sharp-f1}
	\\
	f_{1,\sharp}(x,y) &= \sum_{m\in\Z^2} \sum_{n=1}^\mathcal{N} 
	\widetilde{\eta}_m(x,y) {h}_{1,n,m}(x,y) e^{i\alpha_{n,m}(y) x}, 
	\\
	f_{2,\sharp}(x,y) &= \sum_{m\in\Z^2} \sum_{n=1}^\mathcal{N} 
	\widetilde{\eta}_m(x,y) {h}_{2,n,m}(x,y) e^{i\beta_{n,m}(x) y},
	\label{eqn:sharp-f2}
	\\
	f_{\ell,\mathrm{err}} &= \sum_{m\in\Z^2} f_{\ell,m,\mathrm{err}},%
	\label{eqn:ferrptwise}
	\end{align}
where the functions
\[f_{\ell,m,\mathrm{err}} = \eta_m(\psi^{(\ell)}*_\ell f_\ell) - \psi^{(\ell)}*_\ell (\eta_{m} f_\ell) +  (1-\widetilde{\eta}_m) (\psi^{(\ell)}*_\ell (\eta_{m} f_\ell))\]
satisfy $\|f_{\ell,m,\mathrm{err}}\|_\infty \lesssim \lambda^{\gamma-1}$ uniformly in $m$.
We use this decomposition to analyze our operator as follows:
\[ \LocTOp(f_1, f_2) = \mathrm{T}_\flat + \mathrm{T}_\sharp + \mathrm{T}_\mathrm{err}, \]
where
\begin{align*}
\mathrm{T}_\sharp &= \LocTOp(f_{1,\sharp}, f_{2,\sharp}), 
\\
\mathrm{T}_\flat &= \LocTOp(f_{1,\flat}, f_2) + \LocTOp(f_{1,\sharp}, f_{2,\flat}),
\\
\mathrm{T}_{\mathrm{err}} &= \LocTOp(f_{1,\sharp}, f_{2,\mathrm{err}}) + \LocTOp(f_{1,\mathrm{err}}, f_2). 
\end{align*}

The basic estimate \eqref{eqn:first-final} implies 
\begin{equation}\label{eqn:local-flat-final}
\|\mathrm{T}_\flat\|_1 \lesssim \lambda^{\frac12-\gamma} + \lambda^{-\frac{\kappa}4+\frac{5}{4}\epsilon_2} + \lambda^{-\frac{\delta'}4+\epsilon_2}.
\end{equation}
The triangle inequality and the pointwise bound for $f_{\ell,m,\mathrm{err}}$ give 
\begin{equation}\label{eqn:local-err-final}
\|\mathrm{T}_\mathrm{err}\|_1 \lesssim \lambda^{\gamma-1}.
\end{equation}
It only remains to treat the term $\mathrm{T}_\sharp$.

\subsection{Conclusion of proof of Theorem~\ref{mainresult}}\label{sec:coupdegrace}

Here we estimate the remaining term $\mathrm{T}_\sharp = \LocTOp(f_{1,\sharp}, f_{2,\sharp})$. From \eqref{eqn:sharp-f1}, \eqref{eqn:sharp-f2}, we have that $\|\mathrm{T_\sharp}\|_1$ is dominated by a sum of $O(\lambda^{2\delta'})$ terms of the form
\begin{equation}\label{eqn:cdg-mainterm}
\sum_{\mathbf{m}\in (\Z^2)^2} \iint \Big| \int e^{i (\alpha_{m_1}(y) t +  \beta_{m_2}(x) t^2)} H_\mathbf{m}(x,y,t) dt \Big| dx \,dy,
\end{equation}
where $\mathbf{m}=(m_1,m_2)\in \Z^2\times \Z^2$, $\alpha_{m_1}, \beta_{m_2}$ are measurable real-valued functions which by construction satisfy
\[ |\beta_{m_2}|\lesssim \lambda \quad \text{by \eqref{eqn:localasm-upper-2} and}\] 
\[ |\alpha_{m_1}| \asymp \lambda \quad \text{if \eqref{eqn:localasm-strict} holds with } \ell_*=1,\] 
\[ |\beta_{m_2}| \asymp \lambda \quad \text{if \eqref{eqn:localasm-strict} holds with } \ell_*=2.\] 
So far we have not exploited \eqref{eqn:localasm-strict}. This assumption will be crucial in this section.
Moreover, the function $H_{\mathbf{m}}$ takes the form
\begin{equation}\label{eqn:cdg-Hmdef}
H_{\mathbf{m}}(x,y,t) = (\widetilde{\eta}_{m_1} h_{1,m_1})(x+t,y) (\widetilde{\eta}_{m_2} h_{2,m_2})(x,y+t^2)\zeta(x,y,t),
\end{equation}
with $\widetilde{\eta}_{m}=\widetilde{\eta}(-m+\lambda^{\gamma} \cdot)$ supported in a cube $\widetilde{Q}_m$ of sidelength, say $10\lambda^{-\gamma}$ centered at $\lambda^{-\gamma}m$ and $h_{1,m}, h_{2,m}$ measurable functions with smooth fibers $h_{1,m}(\cdot,y)$, $h_{2,m}(x,\cdot)$ satisfying $\|\partial_\ell^N h_{\ell,m}\|_{\infty} \lesssim_N \lambda^{\tau N}$ uniformly in $m$, and $\zeta$ the compactly supported smooth function stemming from the definition of $\LocTOp$.

The problem of estimating \eqref{eqn:cdg-mainterm} is {\em global} in nature: that is, for a fixed $\mathbf{m}\in\widetilde{\mathfrak{I}}$ it is possible that the integration in $(x,y,t)$ yields no gain over the trivial bound  $O(\lambda^{-3\gamma})$ from the size of the support of $H_\mathbf{m}$ (one can see this by choosing $\alpha_{m_1}(y)=\alpha, \beta_{m_2}(x)=\beta$ to be constant functions so that $\alpha + 2\overline{t}\beta=0$, where $\overline{t}$ is some fixed point from the $t$-support of the integrand).
The challenge will be to show that this cannot happen for too many of the $\mathbf{m}$.

We begin with some technical preparations.
The arguments above utilized cubes 
$\{\widetilde{Q}_m\,:\,m\in\Z^2\}$ that are not pairwise disjoint, but their overlapping
is now inconvenient. Since this collection is finitely overlapping, it can be partitioned into finitely many subcollections, each of which consists of pairwise disjoint cubes. Thus \eqref{eqn:cdg-mainterm} is dominated by a sum of $O(1)$ terms of the form 
\begin{equation}\label{eqn:cdg-mainterm2}
\sum_{\substack{\mathbf{m}\in (\Z^2)^2,\\\widetilde{Q}_{m_1}\in \mathcal{Q}, \widetilde{Q}_{m_2}\in \widetilde{\mathcal{Q}}}} \iint \Big| \int e^{i (\alpha_{m_1}(y) t + \beta_{m_2}(x) t^2)} H_{m_1,m_2}(x,y,t) dt \Big| dx \,dy,
\end{equation}
where each of the collections $\mathcal{Q}, \widetilde{\mathcal{Q}}$ consists of pairwise disjoint cubes.
By compact support of the function $\zeta$, the sum over $\mathbf{m}$ is finite.
Indeed, define $\widetilde{\mathfrak{I}}$ as the set of $\mathbf{m}\in(\Z^2)^2$ such that $\widetilde{Q}_{m_1}\in\mathcal{Q}$, $\widetilde{Q}_{m_2}\in\widetilde{\mathcal{Q}}$ and \[ \iint \Big| \int e^{i (\alpha_{m_1}(y) t + \beta_{m_2}(x) t^2)} H_{\mathbf{m}}(x,y,t) dt \Big| dx \,dy \not=0. \]
Then by the same reasoning as in \eqref{eqn:interactingcount} we have
\[ \#\widetilde{\mathfrak{I}} \lesssim \lambda^{3\gamma}. \]
For each $\mathbf{m}\in \widetilde{\mathfrak{I}}$ we fix $(\overline{x}_\mathbf{m},\overline{y}_\mathbf{m},\overline{t}_\mathbf{m})$ in the support of $H_\mathbf{m}$ (the support is a cube of sidelength $O(\lambda^{-\gamma})$). For each $(x,y,t)$ in the support of $H_{\mathbf{m}}$ we have that the $t$-derivative of the phase function in \eqref{eqn:cdg-mainterm} is
\begin{equation}\label{eqn:cdgphasecomp1}
\alpha_{m_1}(y) + 2t\beta_{m_2}(x) = \alpha_{m_1}(y) + 2\overline{t}_\mathbf{m}\beta_{m_2}(x) + O(\lambda^{1-\gamma}).
\end{equation}
Here we used that $|\beta_{m_2}(x)|\lesssim \lambda$ by \eqref{eqn:localasm-upper-2}. Choose a small parameter $\rho>0$ and suppose that $(x,y,\mathbf{m})$ are such that
\begin{equation}\label{eqn:cdgnonstat}
|\alpha_{m_1}(y) + 2\overline{t}_\mathbf{m}\beta_{m_2}(x)| \ge \lambda^{\tau + \rho}.
\end{equation}
Using \eqref{eqn:cdgphasecomp1}, $\tau>\gamma>\tfrac12>1-\gamma$ and $\|\partial_t^N H_\mathbf{m}\|_\infty = O(\lambda^{\tau N})$, we can integrate by parts in the $t$-integral to obtain
\[ \Big| \int e^{i (\alpha_{m_1}(y) t + \beta_{m_2}(x) t^2)} H_{\mathbf{m}}(x,y,t) dt \Big| \lesssim_N \lambda^{-N} \]
for every $N>0$.
Therefore, \eqref{eqn:cdg-mainterm2} is majorized by
\begin{equation}\label{eqn:cdg-nonstat-term}
\sum_{\mathbf{m}\in\widetilde{\mathfrak{I}}} \iint O(\lambda^{-N}) \mathbf{1}_{|\alpha_{m_1}(y) + 2\overline{t}_\mathbf{m}\beta_{m_2}(x)| \ge \lambda^{\tau + \rho}} dx \,dy
\end{equation}
plus
\begin{equation}\label{eqn:cdg-themainterm}
\sum_{\mathbf{m}\in\widetilde{\mathfrak{I}}} \iint \Big| \int e^{i (\alpha_{m_1}(y) t +  \beta_{m_2}(x) t^2)} H_{\mathbf{m}}(x,y,t) dt \Big| \mathbf{1}_{|\alpha_{m_1}(y) + 2\overline{t}_\mathbf{m}\beta_{m_2}(x)| \le \lambda^{\tau + \rho}} dx \,dy.
\end{equation}
The term \eqref{eqn:cdg-nonstat-term} is $\lesssim_N \lambda^{-N}$ for all $N>0$. It remains to estimate \eqref{eqn:cdg-themainterm}. In this term 
there is no reason to expect cancellation from the $t$--integration. 
Instead, we will demonstrate 
that the scenario 
$|\alpha_{m_1}(y) + 2\overline{t}_\mathbf{m}\beta_{m_2}(x)| \le \lambda^{\tau + \rho}$ 
occurs for only relatively few values of $\mathbf{m}$.

For this purpose it will be convenient to change perspective and interchange the order of the  $(x,y,t)$ integration with the summation over $\mathbf{m}$. Using the triangle inequality, \eqref{eqn:cdgphasecomp1} and \eqref{eqn:cdg-Hmdef}, we estimate
\[ \eqref{eqn:cdg-themainterm} \lesssim \iiint_K \sum_{\mathbf{m}\in\widetilde{\mathfrak{I}}} \mathbf{1}_{(x+t,y)\in Q_{m_1}} \mathbf{1}_{(x,y+t^2)\in Q_{m_2}} \mathbf{1}_{|\alpha_{m_1}(y) + 2t\beta_{m_2}(x)| \le 2\lambda^{\tau + \rho}} dx \,dy \,dt,   \]
where $K$ denotes the compact support of $\zeta$.
Recall that for each $\mathbf{m}\in\widetilde{\mathfrak{I}}$ both of the cubes $\widetilde{Q}_{m_1}\in\mathcal{Q}$ and $\widetilde{Q}_{m_2}\in\widetilde{\mathcal{Q}}$ are chosen from collections consisting of pairwise disjoint cubes. 
As a consequence, for each fixed $(u,v)\in\R^2$ there exists at most one $m$ with $(u,v)\in\widetilde{Q}_{m}\in\mathcal{Q}$. This defines a measurable function \[\mathfrak{m}:\R^2\to \Z^2\quad\text{with}\quad \mathfrak{m}(u,v) = m\;\text{if}\; (u,v)\in \widetilde{Q}_m \in \mathcal{Q} \] 
and $\mathfrak{m}(u,v)=0$ if no such $m$ exists. In the same way we define a function $\widetilde{\mathfrak{m}}:\R^2\to \Z^2$ so that $\widetilde{\mathfrak{m}}(u,v)$ equals the unique $m$ so that $(u,v)\in\widetilde{Q}_{m}\in \widetilde{Q}$, or $0$ if no such $m$ exists.
Then for every $(x,y,t)\in K$,
\[ \sum_{\mathbf{m}\in\widetilde{\mathfrak{I}}} \mathbf{1}_{(x+t,y)\in Q_{m_1}} \mathbf{1}_{(x,y+t^2)\in Q_{m_2}} \mathbf{1}_{|\alpha_{m_1}(y) + 2t\beta_{m_2}(x)| \le 2\lambda^{\tau + \rho}} \le  \mathbf{1}_{|\widetilde{\alpha}(x+t,y) - t\widetilde{\beta}(x,y+t^2)| \le \varepsilon },\]
where we have set
\[ \widetilde{\alpha}(x+t,y) = \lambda^{-1} \alpha_{\mathfrak{m}(x+t,y)}(y),\quad \widetilde{\beta}(x,y+t^2) = - \lambda^{-1} \beta_{\widetilde{\mathfrak{m}}(x,y+t^2)}(x),\quad \varepsilon= 2\lambda^{\tau+\rho-1}. \]
Therefore,
\begin{equation}\label{eqn:cdg-themainterm-2}
\eqref{eqn:cdg-themainterm}\lesssim |\{ (x,y,t)\in K\,:\, |\widetilde{\alpha}(x+t,y) - 2t\widetilde{\beta}(x,y+t^2)| \le \varepsilon\}|.
\end{equation}
Thus we aim to estimate the measure of the set on the 
right-hand side of \eqref{eqn:cdg-themainterm-2}. The following result is proved in \S \ref{sec:sublevel} below.

\begin{lemma} \label{lemma:sublevel} %
Let $K\subset\mathbb{R}^2\times (0,\infty)$ be a compact set and $\alpha,\beta$ measurable functions $\R^2\to\R$. Suppose that either $|\alpha|\asymp 1$ or $|\beta|\asymp 1$. 
Then there exist $\sigma, C\in (0,\infty)$ such that for all $\varepsilon\in (0,1]$,
\begin{equation} \label{lastsublevel}
\big|\{(x,y,t)\in K: |\alpha(x+t,y) - 2t\beta(x,y+t^2)| \le\eps\}\big|
\le C\eps^\sigma.  \end{equation}
The constants $C$ and $\sigma$ only depend on $K$ and not on the measurable functions $\alpha,\beta$.
\end{lemma}

\begin{remarka}
Under the assumptions of the lemma, the additional condition $|\alpha|+|\beta|\lesssim 1$ can be assumed without loss of generality:
for example, if we assume $|\alpha|\asymp 1$, then in order for the inequality $|\alpha(x+t,y)-2t\beta(x,y+t^2)|\le \varepsilon$ to hold for some $(x,y,t)\in K$, we must have $|\beta|\lesssim 1$. Thus we may replace $\beta$ by $\beta$ multiplied with a suitable characteristic function.
\end{remarka}

Combining the previous estimates \eqref{eqn:cdg-nonstat-term}, \eqref{eqn:cdg-themainterm-2} and Lemma \ref{lemma:sublevel} we obtain
\[ \|\mathrm{T}_\sharp\|_1 \lesssim \lambda^{2\delta' + \sigma(\tau+\rho-1)}. \]
Here we used \eqref{eqn:localasm-strict}. Note that the exponent is negative if $\delta'>0$, $\rho>0$, $\tau=\gamma+\kappa$ are chosen small enough. Together with \eqref{eqn:local-flat-final} and \eqref{eqn:local-err-final} this concludes the proof of Theorem \ref{mainresult}.

\subsection{Proof of the sublevel set estimate Lemma \ref{lemma:sublevel}}
\label{sec:sublevel}
The proof is in the spirit of arguments in \cite[\S 11]{Chr20}, but involves certain idiosyncrasies that are somewhat technical.
We first give the proof under the assumption that $|\alpha|\asymp 1$.

Write $z=(x,y)\in\R^2$.
Changing variables $(x,y,t)\mapsto (x-t,y,t)$ it suffices to show that the set
\begin{equation} \label{lastsublevel2}
\mathcal{E} = \{(z,t)\in K: 
|\alpha(z) - 2t\beta(z+(-t,t^2))| \le \eps\}
\end{equation}
satisfies $|\mathcal{E}| \lesssim \eps^\tau$.
Moreover, it will be convenient to assume that $K=[0,1]^2\times I$, where $I\subset (0,\infty)$ is a closed interval of length one. This can be achieved by an affine transformation in $z$ and a rescaling in $t$. We may also assume that $|\mathcal{E}|>0$ since otherwise there is nothing to show.\\

\noindent {\em Claim.}
There exist a point $\overline{z}= (\overline{x},\overline{y})\in\mathbb{R}^2$,
and a measurable set $\mathscr{A}\subset I^3$ so that $|\mathscr{A}|\gtrsim |\mathcal{E}|^7$ and for every $(t_1,t_2,t_3)\in \mathscr{A}$,
\begin{equation} \label{iterated}
\left\{
\begin{aligned}
&|\alpha(\overline{z})-2t_1\beta(\overline{z}+(-t_1,t_1^2))| \le \varepsilon,
\\
&|\alpha(\overline{z}+(-t_1,t_1^2)-(-t_2,t_2^2))-2t_2\beta(\overline{z}+ (-t_1,t_1^2))| \le \varepsilon,
\\
&|\alpha(\overline{z}+(-t_1,t_1^2)-(-t_2,t_2^2))-2t_3\beta(\overline{z}+(-t_1,t_1^2)-(-t_2,t_2^2)+(-t_3,t_3^2))| \le \varepsilon.
\end{aligned} \right.  \end{equation}

\begin{proof}[Proof of claim]
Define $\mathcal{E}'_0\subset[0,1]^2$ to be 
\[ 
\mathcal{E}'_0 = \big\{
z\in \R^2:
\big|\{t\in I: (z,t)\in\mathcal{E}\} \big| \ge \tfrac12 |\mathcal{E}|\big\}.
\]
We have $|\mathcal{E}'_0|\ge \tfrac12 |\mathcal{E}|$ since by Fubini's theorem,
\begin{equation}\label{eqn:sublevelpf-fubini1}
|\mathcal{E}| = \int_{\mathcal{E}_0'} |\{t\in I: (z,t)\in\mathcal{E}\}| dz + \int_{[0,1]^2\setminus \mathcal{E}'_0} |\{t\in I: (z,t)\in\mathcal{E}\}| dz
\end{equation}
\[ \le |\mathcal{E}_0'| + \tfrac12  |\mathcal{E}|. \]
Define
\[ \mathcal{E}_1 = \{(w,t)\in\mathcal{E}: w\in\mathcal{E}'_0 \}.\]
By Fubini's theorem and by definition of $\mathcal{E}_0'$,
\begin{equation}\label{eqn:sublevelpf-fubini2}
|\mathcal{E}_1| = \int_{\R^2} \Big( \int_I \mathbf{1}_{\mathcal{E}}(w,t) dt \Big) \mathbf{1}_{\mathcal{E}_0'}(w) dw \ge \tfrac12 |\mathcal{E}| \cdot |\mathcal{E}_0'| \ge \tfrac 14 |\mathcal{E}|^2.
\end{equation}
Next consider 
\[ \mathcal{E}'_1 = \big\{w\in \R^2:
\big|\{t\in I: (w+(-t,t^2),t)\in\mathcal{E}_1\} \big| 
\ge \tfrac12 |\mathcal{E}_1|\big\}.\] 
By a similar argument as in \eqref{eqn:sublevelpf-fubini1} we obtain $|\mathcal{E}_1'|\ge \tfrac12 |\mathcal{E}_1|\ge \tfrac18 |\mathcal{E}|^2$. 
Similarly, define
\[ \mathcal{E}_2 = \{(w,t)\in\mathcal{E}_1: w\in\mathcal{E}'_1\}\]
and observe from the same argument as in \eqref{eqn:sublevelpf-fubini2} that $|\mathcal{E}_2|\ge \frac12 |\mathcal{E}_1| \cdot |\mathcal{E}_1'|\ge 2^{-6} |\mathcal{E}|^4$.
Finally, 
\[ \mathcal{E}'_2 = \big\{w\in\mathbb{R}^2:
\big|\{t\in I: (w-(-t,t^2),t)\in\mathcal{E}_2\}\big|
\ge \tfrac12 |\mathcal{E}_2|\big\} \]
satisfies $|\mathcal{E}'_2|\ge \frac12 |\mathcal{E}_2| \ge 2^{-7} |\mathcal{E}|^4>0$. In particular, $\mathcal{E}'_2$ is nonempty.

Now choose any $\overline{z}\in\mathcal{E}'_2$ and define
\begin{gather}
U=\{t\in I: (\overline{z}-(-t,t^2),t)\in\mathcal{E}_2\},
\\
U_{t_1} = \{t\in I: (\overline{z}-(-t_1,t_1^2)+(-t,t^2),t)\in\mathcal{E}_1\}
\ \text{ for each $t_1\in U$},
\\
U_{t_1,t_2}=\{t\in I: (\overline{z}-(-t_1,t_1^2)+(-t_2,t_2^2),t)\in\mathcal{E} \}
\ \text{ for each $t_1\in U$ and $t_2\in U_{t_1}$. }  
\end{gather}
Finally, let 
\[ \mathscr{A} = \{ (t_1,t_2,t_3)\in I^3\,:\, t_1\in U,\, t_2\in U_{t_1},\, t_3\in U_{t_1,t_2} \}. \]
By Fubini's theorem, 
\[ |\mathscr{A}| = \int_I \mathbf{1}_{U}(t_1) \int_I \mathbf{1}_{U_{t_1}}(t_2) \int_I \mathbf{1}_{U_{t_1,t_2}}(t_3) dt_3 dt_2 dt_1\ge |U| \cdot (\inf_{t_1\in U} |U_{t_1}| ) \cdot (\inf_{t_1\in U, t_2\in U_{t_1}} |U_{t_1,t_2}|), \]
which is $\gtrsim |\mathcal{E}|^4 |\mathcal{E}|^2 |\mathcal{E}| = |\mathcal{E}|^7$, concluding the proof of the claim.
\end{proof}

Set $\alpha_0 = \alpha(\overline{z})$. Then $|\alpha_0|\asymp 1$ by assumption. Let us write $\mathbf{t}=(t_1,t_2,t_3)\in\R^3$.
By \eqref{iterated}, the function
\[
F(\mathbf{t}) =  \alpha_0 t_3^{-1}t_2t_1^{-1} 
- \beta(\overline{x} -t_1+t_2-t_3,\overline{y} + t_1^2-t_2^2+t_3^2)
\]
satisfies
$\big| F(\mathbf{t}) \big|\lesssim \eps$ for every $\mathbf{t}\in\mathscr{A}$.

We will show that sublevel sets of such functions $F$
are small, uniformly in all measurable functions $\beta$. Define
\[ \theta_1(\mathbf{t})  = -t_1+t_2-t_3
\ \  \text{ and } \ \ 
\theta_2(\mathbf{t}) = t_1^2 - t_2^2  + t_3^2
\ \  \text{ and } \ \ 
\vartheta(\mathbf{t}) = t_3^{-1}t_2t_1^{-1}.  \]
Consider the vector field in $\R^3$ defined by
\[V(\mathbf{t}) = (\nabla \theta_1 \times \nabla\theta_2) (\mathbf{t}) = 2(t_3-t_2)\partial_{t_1} + 2(t_3-t_1)\partial_{t_2} + 2(t_2-t_1)\partial_{t_3}.\]
$V$ vanishes identically on the line 
$\Delta=\{(t_1,t_2,t_3): t_1=t_2=t_3\}\subset\mathbb{R}^3$,
but vanishes nowhere else.

Let $\gamma:\R^3\times \R\to \R^3$ denote the flow associated to $V$. 
For generic $\mathbf{t}$, the function
\[ s\mapsto F( \gamma(\mathbf{t}_0, s) ) = \alpha_0 \vartheta(\gamma(\mathbf{t}_0,s)) - \mathrm{constant} \]
is analytic and nonconstant.
Indeed, the functions $\theta_1, \theta_2$ are by construction constant along the integral curves of $V$, hence so is the function
$\mathbf{t}\mapsto \beta(\overline{x} + \theta_1(\mathbf{t}), 
\overline{y} + \theta_2(\mathbf{t}))$.
On the other hand,
\begin{equation} \label{Vvartheta}
V\vartheta(\mathbf{t}) = 2t_1^{-2}t_3^{-2}
\big( (t_2-t_3)t_1^2 + (t_3-t_1)t_2^2 + (t_1-t_2)t_3^2 \big)
\end{equation}
is generically nonzero.
Since we are working in a bounded region in which the absolute values of the coordinates
$t_j$ are bounded below by a strictly positive quantity,
the factor $t_1^{-2}t_3^{-2}$ is bounded above and below
by finite positive constants.

To diagonalize the system underlying the vector field, we introduce the coordinates $\mathbf{u} = J\mathbf{t}$ with $J\in\R^{3\times 3}$ given by 
\[ (u_1,u_2,u_3)= (t_2-t_1, t_2 - t_3, t_1 - t_2 + t_3).\]
Observe that $|\det J|=1$. Let $\Omega=J (I^3)\subset [-1,1]^2\times (I-I+I)$. In $\mathbf{u}$-coordinates, the integral curves of the vector field degenerate along the line $\{u_1=u_2=0\}$. We perform an additional dyadic decomposition of the region $\Omega$ by
\[ \mathbf{1}_\Omega \le \sum_{k=0}^\infty \sum_{|d|\lesssim 2^k} \mathbf{1}_{\Omega_{k,d}}, \]
where for integers $k$ and $d$,
\[\Omega_{k,d} = \{ \mathbf{u}\in\Omega\,:\, 2^{-k-1}\le \max(|u_1|,|u_2|)\le 2^{-k},\, |u_3-2^{-k}d|\le 2^{-k} \}. \]
This decomposition becomes necessary because, as we will see below, the integral curves of the vector field are hyperbolae, which causes an unfortunate lack of compactness. Each region $\Omega_{k,d}$ has measure $\approx 2^{-3k}$ and will be treated separately. That is, we begin by estimating
\[ |\mathscr{A}| = \int_{\R^3} \mathbf{1}_\mathscr{A} = \int_{\R^3} \mathbf{1}_{\Omega\cap J\mathscr{A}} \le \sum_{k=0}^\infty \sum_{|d|\lesssim 2^k} |\Omega_{k,d}\cap J\mathscr{A}|. \]
Fix $k\ge 0$ and $|d|\lesssim 2^k$. To estimate $|\Omega_{k,d}\cap J\mathscr{A}|$ we introduce normalized coordinates 
\[ \mathbf{v}= (v_1,v_2,v_3) = 2^k (u_1,u_2,u_3 - 2^{-k} d) = \Lambda_{k,d} \mathbf{u} \]
so that the $\mathbf{u}$-region $\Omega_{k,d}$ is mapped into the $\mathbf{v}$-region 
\[\Box=\{ \mathbf{v}=(v_1,v_2,v_3)\,:\,2^{-1}\le \max(|v_1|,|v_2|)\le 1,\, |v_3|\le 1\}\supset \Lambda_{k,d} \Omega_{k,d}.\]
In $\mathbf{v}$-coordinates, the vector field $V$ is given by
\[ \tfrac12 \tilde{V}(\mathbf{v}) = v_1 \partial_{v_1} - v_2 \partial_{v_2}, \]
so the integral curves are given in $\mathbf{v}$--coordinates by
\[ \gamma:\R^3\times \R\to \R^3,\quad \gamma(c; s) = (c_1 e^{2s}, c_2 e^{-2s}, c_3). \]
If we set $c_\ell=\pm 1$ for $\ell=1,2$, we obtain a foliation of the half-space $\{ \pm v_\ell>0 \}$ by one-dimensional integral curves (each of which is a hyperbola). Define coordinate transformations
\[ \varphi_{1,\pm} : \R^3 \to \{ \pm v_1 > 0\},\quad (w_1,w_2,w_3)\mapsto \gamma( (\pm 1,w_2,w_3); w_1), \]
\[ \varphi_{2,\pm} : \R^3 \to \{ \pm v_2 >0 \},\quad (w_1,w_2,w_3)\mapsto \gamma( (w_2,\pm 1,w_3); w_1). \]
One verifies that each $\varphi_{\ell,\pm}$ is a diffeomorphism and $\det D\varphi_{\ell,\pm} \equiv \pm 2$. 
We decompose $\Box$ into four rectangular boxes by
\[ \Box = \bigcup_{\ell\in\{1,2\}, \pm} \Box_{\ell,\pm}\quad\text{with}\quad \Box_{\ell,\pm} = \Box\cap \{ \pm v_\ell \ge 2^{-1} \}. \]
The motivation for this decomposition (and the previous $(k,d)$ decomposition) is that we now have compactness in the $\mathbf{w}$-coordinates. For instance,
\[ \varphi^{-1}_{1,+}(\Box_{1,+})\subset [-\tfrac12\log(2),0] \times [-1,1]^2.\]
We are ready to calculate
\[ |\Omega_{k,d}\cap J\mathscr{A}| = 2^{-3k} |\Lambda_{k,d}\Omega_{k,d}\cap \Lambda_{k,d} J \mathscr{A}| = 2^{-3k+1} \sum_{\ell\in\{1,2\},\pm} |\varphi^{-1}_{\ell,\pm}(\Lambda_{k,d}\Omega_{k,d}\cap \Box_{\ell,\pm})\cap  \varphi^{-1}_{\ell,\pm}  \Lambda_{k,d} J \mathscr{A}|, \]
Recalling the definition of $\mathscr{A}$ and setting
\[ G_{\ell,\pm,k,d} = F\circ J^{-1}\circ \Lambda_{k,d}^{-1}\circ \varphi_{\ell,\pm},\]
we obtain that $|\Omega_{k,d}\cap J\mathscr{A}|$ is dominated by
\[ 2^{-3k+1} \sum_{\ell\in\{1,2\},\pm}  |\{w\in\mathcal{K}_{\ell,\pm,k,d}\,:\,|G_{\ell,\pm,k,d}(w)|\lesssim \varepsilon \}|, \]
where $\mathcal{K}_{\ell,\pm,k,d} = \varphi^{-1}_{\ell,\pm}(\Lambda_{k,d}\Omega_{k,d}\cap \Box_{\ell,\pm})\subset \varphi_{\ell,\pm}^{-1}(\Box_{\ell,\pm})$ is a compact set.
By construction we have
\begin{equation}\label{eqn:sublevel-keyderiv0}
\partial_{w_1} G_{\ell,\pm,k,d}(w) = \alpha_0 \tilde{V} (\vartheta\circ J^{-1}\circ \Lambda_{k,d}^{-1} )|_{\varphi_{\ell,\pm}(w)}.    
\end{equation} 
The quantity $\tilde{V} (\vartheta\circ J^{-1}\circ \Lambda_{k,d}^{-1} )|_\mathbf{v}$ can be seen to equal
\begin{equation}\label{eqn:sublevel-keyderiv1}
2^{-3k} R_{k,d}(\mathbf{v}) P(\mathbf{v}),
\end{equation}
where $P$ is a homogeneous polynomial of degree $3$ that is independent
of $k$ and does not vanish identically, and
\[R_{k,d}(\mathbf{v}) = ( 2^{-k} (v_2+v_3+d))^{-2} (2^{-k} (v_1+v_3+d))^{-2}.\]
Observe that $|R_{k,d}(\mathbf{v})|\asymp 1$ if $\mathbf{v}\in \Lambda_{k,d}\Omega_{k,d}\subset \Lambda_{k,d} J I^3$. As a consequence,
\begin{equation}\label{eqn:sublevel-keyderivest}
|\partial_{w_1} G_{\ell,\pm,k,d}(w)| \approx 2^{-3k} |P(\varphi_{\ell,\pm}(w))|, \end{equation}
if $w\in \mathcal{K}_{\ell,\pm,k,d}$. 
By a well-known estimate in the spirit of van der Corput's lemma there exists $a>0$ so that
\begin{equation}\label{eqn:sublevel-deriv1}
|\{ w\in\mathcal{K}_{\ell,\pm}\,:\,|P(\varphi_{\ell,\pm}(w))|\le \varepsilon \}| \lesssim \varepsilon^a.
\end{equation}
We claim that the estimate \eqref{eqn:sublevel-deriv1} implies
\begin{equation}\label{eqn:sublevel-penult}
|\{ w\in \mathcal{K}_{\ell,\pm,k,d}\,:\, |G_{\ell,\pm,k,d}(w)| \lesssim\varepsilon\}|  \lesssim 2^{Ck} \varepsilon^c,
\end{equation}
uniformly in $k,d$, where $c,C>0$ are some fixed constants. This would imply
\[ |\Omega_{k,d}\cap J\mathscr{A}| \lesssim 2^{Ck} \varepsilon^c \]
for some fixed constants $c,C>0$ (recall that $c,C$ may vary from line to line). Together with the trivial estimate $|\Omega_{k,d}\cap J\mathscr{A}|\lesssim 2^{-3k}$ we would then obtain
\[ |\mathscr{A}| \lesssim \sum_{k=0}^\infty \sum_{|d|\lesssim 2^k} \min( 2^{Ck} \varepsilon^{c}, 2^{-3k}) \lesssim \varepsilon^{c'} \]
and thus $|\mathcal{E}|\lesssim |\mathscr{A}|^{1/7}\lesssim \varepsilon^{c'/7}.$ This concludes the proof up to the verification of \eqref{eqn:sublevel-penult}.

\begin{proof}[Proof of \eqref{eqn:sublevel-penult}]
To simplify notation we assume $(\ell,\pm)=(1,+)$ in the following and suppress $\ell,\pm$ from subscripts (in particular, we write $\varphi=\varphi_{1,+}$, $\mathcal{K}_{k,d}=\mathcal{K}_{\ell,\pm,k,d}$, {\em etc.}). The other three cases of $\ell,\pm$ follow in the same way. There are two delicate issues that we need to deal with. One is that while $G_{k,d}$ is analytic in $w_1$, it is only measurable as a function of all three variables $(w_1,w_2,w_3)$. We deal with this by observing that $\partial_1 G_{k,d}$ is analytic in all three variables and make use of {\L}ojasiewicz's inequality. The second issue is that we need to track the dependence of the constants on $k$. A convenient way to do this is to apply {\L}ojasiewicz's inequality to $P\circ \varphi$ rather than $\partial_1 G_{k,d}$, motivated by \eqref{eqn:sublevel-keyderivest}. A subtlety here is that \eqref{eqn:sublevel-keyderivest} only applies on $\mathcal{K}_{k,d}$. %

We have $\mathcal{K}_{k,d}\subset \mathcal{K}=[-\tfrac12\log(2),0]\times [-1,1]^2$. Let ${\mathcal{K}^*}$ denote an open $2^{-10}$-neighborhood of $\mathcal{K}$. The zero set of $P\circ\varphi$ will be denoted 
\[ \mathcal{Z} = \{ w\in {\mathcal{K}^*}\,: P(\varphi(w))=0 \}.  \]
By {\L}ojasiewicz's inequality \cite{Loja1959}, there exists $b>0$ so that
\begin{equation}\label{eqn:loj1}
|P(\varphi(w))| \gtrsim \mathrm{dist}(w, \mathcal{Z})^b
\end{equation}
for all $w\in \mathcal{K}$.
Cover $\mathcal{K}_{k,d}$ with a grid of closed (axis-aligned) cubes $Q\subset \mathcal{K}$ with pairwise disjoint interiors, each of sidelength $\rho$, where $\varepsilon$ is sufficiently small and $\rho$ is to be determined with $\varepsilon\ll \rho\ll 1$.
Denote the collection of all these cubes by $\mathcal{Q}$ (this depends on $k,d$, but in a harmless way). Decompose
\[ \mathcal{Q} = \mathcal{Q}_{\mathrm{near}} \cup \mathcal{Q}_{\mathrm{far}} \cup \mathcal{Q}_\mathrm{bdry}, \]
where $\mathcal{Q}_\mathrm{near}$ consists of all $Q\in\mathcal{Q}$ so that $Q\subset \mathcal{K}_{k,d}$ and $\mathrm{dist}(Q,\mathcal{Z})\le \rho$, $\mathcal{Q}_\mathrm{far}$ consists of all $Q\in\mathcal{Q}$ so that $Q\subset \mathcal{K}_{k,d}$ and $\mathrm{dist}(Q,\mathcal{Z})>\rho$ and $\mathcal{Q}_{\mathrm{bdry}}$ consists of the remaining cubes (those intersecting the boundary of $\mathcal{K}_{k,d}$).

If $Q\in\mathcal{Q}_{\mathrm{near}}$, then by the mean value theorem  $|P(\varphi(w))|\lesssim \rho$ for all $w\in Q$. 
It may happen that $G_{k,d}$ is small on such cubes. However, by \eqref{eqn:sublevel-deriv1}, the total volume of these cubes is
\[ |\bigcup \mathcal{Q}_{\mathrm{near}}| \lesssim \rho^a. \]

Suppose $Q\in\mathcal{Q}_{\mathrm{far}}$. Then \eqref{eqn:loj1} gives $|P(\varphi(w))|\gtrsim \rho^b$ for every $w\in Q$.
Using \eqref{eqn:sublevel-keyderivest} and $Q\subset \mathcal{K}_{k,d}$ this yields $|\partial_{w_1} G_{k,d}(w)|\gtrsim 2^{-3k} \rho^b$. 
This implies that for every $(w_2,w_3)\in [-1,1]^2$,
\[ |\{w_1\in\R\,:\,(w_1,w_2,w_3)\in Q,\,|G_{k,d}(w)|\le\varepsilon\}| \lesssim 2^{3k} \varepsilon\rho^{-b}, \]
with implicit constant independent of $k,d,w_2,w_3$.
Fubini's theorem implies
\[ |\{ w\in Q\,:\,|G_{k,d}(w)|\le\varepsilon \}|\lesssim 2^{3k} \varepsilon \rho^{2-b}. \]
Since there are $O(\rho^{-3})$ cubes, we obtain
\[ \bigcup_{Q\in\mathcal{Q}_\mathrm{far}} |\{ w\in Q\,:\, |F(w)|\le \varepsilon \}| \lesssim 2^{3k} \varepsilon \rho^{-1-b}. \]
Finally, $\mathcal{Q}_\mathrm{bdry}$ is harmless because it contains only $O(\rho^{-2})$ cubes, so 
\[ |\bigcup \mathcal{Q}_\mathrm{bdry}| \lesssim \rho. \]
Altogether we obtain
\[ |\{ w\in\mathcal{K}\,:\,|F(w)|\le \varepsilon \}| \lesssim \rho^a + 2^{3k} \varepsilon \rho^{-1-b} + \rho, \]
which is $\lesssim 2^{Ck} \varepsilon^c$ after choosing $\rho=\varepsilon^{c'}$ appropriately.
\end{proof}

\section{The smooth case}\label{sec:anisotp}
In this section we prove Theorem \ref{thm:anisotp}.

\subsection{Cone decomposition}
First we perform a cone decomposition of the symbol which is adapted to its anisotropic structure.   
Throughout this section we will denote
\begin{align*}
\g(x)=e^{-\pi x^2}\quad \textup{and} \quad \h(x)=\g'(x) = -2\pi x e^{-\pi x^2}.
\end{align*} 
We also write $\rho_s(x) = s^{-1}\rho(s^{-1}x)$
for a function $\rho$ on $\mathbb{R}$. 
Let ${\psi}$ be a smooth even function supported in $[-2,-1] \cup [1,2]$.
Note that  since ${\psi}$ vanishes near the origin,
$$\int_{0}^\infty  {\psi}(t^\alpha \xi) {\h}(t^\alpha\xi)^2   \frac{dt}{t} =C$$
is the same finite constant $C$ for any $\xi\neq 0$. 
Therefore, $m(\xi,\eta)$ is equal to a constant times
\begin{align}\label{mdecompose}
\int_0^\infty \int_0^\infty m(\xi,\eta){\psi}(t^\alpha \xi){\h}(t^\alpha\xi)^2 {\psi}(s^\beta \eta){\h}(s^\beta\eta)^2  \frac{ds}{s} \frac{dt}{t} .
\end{align}
Let us define $\varphi$ by
$${\varphi}(\xi)=\int_1^\infty {\psi}(s^\alpha \xi) {\h}(s^\alpha \xi)^2 \frac{ds}{s}.$$
Note that ${\varphi}$ is a smooth  function and supported in $[-2,2]$. Splitting the integration in \eqref{mdecompose} into the regions    $t\leq s$  and $s\leq t$ and integrating in the larger parameter
we obtain  that $m(\xi,\eta)$ equals a constant times
\begin{align}\label{mdecompose2}
\int_0^\infty  m(\xi,\eta) {\varphi}(t^\alpha\xi) {\psi}(t^\beta\eta) {\h}(t^\beta\eta)^2  \frac{dt}{t}  + \int_0^\infty  m(\xi,\eta) {\psi}(t^\alpha\xi) {\h}(t^\alpha\xi)^2 {\varphi}(t^\beta\eta)  \frac{dt}{t}  .
\end{align}

We discuss the first term only. The same discussion applies to the second, which is obtained from the first by interchanging the roles of the variables $\xi, \eta$.
Let us denote 
$$m^{(t)}(\xi,\eta) = m(t^{-\alpha}\xi,t^{-\beta}\eta) {\widetilde{\varphi}}(\xi)\g(\xi)^{-1} \widetilde{\psi}(\eta)$$
where  ${\widetilde{\varphi}}$ is a smooth function supported in $[-2.1,2.1]$ and constantly equal to $1$ on the support of $\varphi$, while $\widetilde{\psi}$ is a smooth function supported in $[-2.1,-0.9]\cup [0.9,2.1]$ and constantly equal to $1$ on the support of $\psi$.
Then the first term in \eqref{mdecompose2} can be written as 
\begin{align*}
\int_0^\infty  m^{(t)}(t^\alpha\xi,t^\beta\eta) {\varphi}(t^\alpha\xi)\g(t^\alpha \xi) {\psi}(t^\beta\eta) {\h}(t^\beta\eta)^2  \frac{dt}{t} .
\end{align*}
Expanding  $m^{(t)}$ into (rescaled) Fourier series we obtain for the last display
\begin{align*}
 \sum_{(u,v)\in \Z^2} \int_0^\infty  c_t(u,v) \varphi(t^\alpha\xi)\g(t^\alpha \xi)e^{c\pi i ut^\alpha \xi} \psi(t^\beta \eta){\h}(t^\beta\eta)^2 e^{c\pi i v t^\beta\eta}     \frac{dt}{t}, 
\end{align*}
where $c>0$ is a fixed constant and $c_t(u,v)$ are Fourier coeffients. 
Integrating  by parts and using  the symbol estimates \eqref{symest}  we obtain 
$$|c_t(u,v)|\lesssim_N (1+|u|)^{-N}(1+|v|)^{-N}$$
for $N>0$, uniformly in $t$.  

We will use the notation
\[ \varphi_{t,p} (x) = t^{-1} \varphi(t^{-1} (x-p)), \]
and note that $\varphi_t = \varphi_{t,0}$.
Normalizing, passing to the spatial side and using $\widecheck{\h}=i\,\h$, the preceding discussion leads to 
 model operators of the form
\begin{align*} %
T^{(u,v)}(f_1,f_2)(x,y) = \int_0^\infty  c(t) (f_1*_1(\g*\widecheck{\varphi}_{1,u})_{t^\alpha})(x,y) (f_2*_2(\h * \h * \widecheck{\psi}_{1,u})_{t^\beta})(x,y)  \frac{dt}{t}
\end{align*} 
with $|c(t)|\leq 1$. 
\begin{proposition}\label{prop:twisted}
Let $1<p,q< \infty$,  $\tfrac12< r< 2$ and $p^{-1}+q^{-1}=r^{-1}$.
Then for every $u,v\in\R$,
\begin{align}
\label{estimate-cone}
\|T^{(u,v)}(f_1,f_2)\|_{r} \lesssim_{p,q} C_{u,v}  \|f_1\|_p \|f_2\|_q,
\end{align}
where $C_{u,v}=(1+|u|+|v|)^{100}$.
\end{proposition}
Theorem \ref{thm:anisotp} follows by applying this proposition to the operator corresponding to each fixed $(u,v)$ and finally summing in $(u,v)\in \Z^2$.
Note that the range in Theorem \ref{thm:anisotp} is a subset of the range in Proposition \ref{prop:twisted}.
This is because we split the symbol $m$   into two parts in \eqref{mdecompose2}.   Proposition \ref{prop:twisted} gives bounds for the operator corresponding to the first part. The estimates for the second part follow by symmetry with the roles of $f_1$ and $f_2$ interchanged, with appropriate ranges for $p$ and $q$. Taking the intersections of the ranges corresponding to the two terms in \eqref{mdecompose2} then yields  Theorem~\ref{thm:anisotp}. 
 
We will first prove Proposition \ref{prop:twisted} for the case $r>1$ and $2<p,q<\infty $. Then we will extend the range by a fiber-wise Calder\'on--Zygmund decomposition from \cite{Ber12}.
The proof will make use of a localization procedure involving trees of dyadic rectangles, analogous to \cite{Kov12}. This requires some notations and definitions.

For $k\in \Z$ we consider dyadic rectangles  of the form $Q=I\times J$ where $I$ is a dyadic interval of length $2^{\alpha k}$ and $J$ a dyadic interval of length $2^{\beta k}$, where $\alpha,\beta$ are positive integers. Let $\ell(Q)=2^k$. Each such rectangle partitions into $2^{\alpha+\beta}$ many children $Q_i$, each with $\ell(Q_i)=2^{k-1}$. For each $k$, the rectangles partition $\R^2$. 
For any collection of such rectangles $\mathcal{Q}$ we introduce
$$\Omega_\mathcal{Q} = \bigcup_{Q\in \mathcal{Q}} Q\times \big [\tfrac{\ell(Q)}{2},\ell(Q) \big ],$$
which is  a region in  $\R^{3}_+$. 
A finite collection of dyadic rectangles $\mathcal{T}$ is called a {\em tree} if there exists $Q_\mathcal{T}\in \mathcal{T}$, called the {\em root} of the tree, which  satisfies  $Q\subseteq Q_\mathcal{T}$ for all $Q\in \mathcal{T}$. A tree is called {\em convex}, if   for any dyadic rectangles $Q,Q',Q''$ we have that  $Q \subseteq Q' \subseteq Q''$ and $Q,Q''\in \mathcal{T}$ imply $Q'\in \mathcal{T}$. 
By $\mathcal{L}(\mathcal{T})$ we denote the set of {\em leaves} of a tree $\mathcal{T}$, i.e. those rectangles which are not contained in $\mathcal{T}$, but whose parent is. Note that the leaves $\mathcal{L}(\mathcal{T})$ partition the root $Q_\mathcal{T}$.
We fix the function
\[\theta(x)= (1+|x|)^{-10}\quad \text{for}\;x\in\R.\]
Given a collection $\mathcal{Q}$ of dyadic rectangles, let
\begin{equation}\label{def:mq}
 \mathcal{M}_{\mathcal{Q}}(f)=%
\sup_{Q\in\mathcal{Q}} \Big(\;|f|^{2} * (\theta_{\ell(Q)^\alpha}\otimes \theta_{\ell(Q)^\beta}) (c(Q)) \Big )^{1/2} ,
\end{equation}
where $c(Q)$ denotes the center of $Q$.
Observe that 
\[ \mathcal{M}_{\mathcal{Q}}(f) \asymp%
\sup_{\substack{(x,y,t)\in \Omega_{\mathcal{Q}}}}\Big(\;|f|^{2} * (\theta_{t^\alpha}\otimes \theta_{t^\beta}) (x,y) \Big )^{1/2},  \]
since if $(x,y,t)\in Q \times [\ell(Q)/2, \ell(Q)]$, then $(1+t^{-\alpha} | x-c(Q)_1|)^{10} (1+t^{-\beta}|y-c(Q)_2|)^{10} \lesssim 1$.
Finally, we define  
\begin{align*}
    \Lambda^{(u,v)}_{\mathcal{Q}}(f_1,f_2,f_3,f_4) = \int_{\Omega_{\mathcal{Q}}}  c(t) \int_{\R^4} & f_1(x',y)f_2(x,y')f_3(x,y) f_4(x',y')
 (\widecheck{\varphi}_{1,u})_{t^\alpha,\p}(x) \g_{t^\alpha,\p}(x')\\
&   (\h*\widecheck{\psi}_{1,v})_{t^\beta,\q}(y) \h_{t^\beta,\q}(y') 
 dx\,dy\,dx'\,dy'\, d\p\,d\q\, \frac{dt}{t} .
\end{align*}
Note that by specifying 
 $f_4=1$ we obtain a  (localized) version of a form dual to $T^{(u,v)}$. The key estimate is as follows.

\begin{proposition}[Tree estimate]
\label{prop:tree}
For   any convex tree $\mathcal{T}$ it holds
\begin{align*}
 | \Lambda^{(u,v)}_{\mathcal{T}}(f_1,f_2,f_3,f_4)   |  \leq C_{u,v}   |Q_\mathcal{T}| \prod_{j=1}^4 \mathcal{M}_\mathcal{T}(f_j),
\end{align*}
where $C_{u,v}=(1+|u|+|v|)^{100}$.
\end{proposition}

To prove the tree estimate we will combine repeated applications of the Cauchy-Schwarz inequality  with a certain telescoping identity. To state this identity we need to introduce some more notation.
Given a convex tree $\mathcal{T}$ and parameters $\lambda\geq 1$, $\r\in\R$, let us define  the forms 
\begin{align*}
    \Theta_{\mathcal{T},\lambda,\r}^{(1)}(f_1,f_2,f_3,f_4) =\int_{\Omega_{\mathcal{T}}} \int_{\R^4} &f_1(x',y)f_2(x,y')f_3(x,y)f_4(x',y')\\
&    \h_{\lambda t^\alpha,\p}(x)\h_{\lambda t^\alpha,\p}(x')\g_{t^\beta,\q+\r t^\beta}(y)\g_{t^\beta,\q+\r t^\beta}(y')   dx\,dy\,dx'\,dy'\,d\p\,d\q\,\frac{dt}{t},\\
 \Theta_{\mathcal{T},\lambda,\r}^{(2)}(f_1,f_2,f_3,f_4) =\int_{\Omega_{\mathcal{T}}} \int_{\R^4} &f_1(x',y)f_2(x,y')f_3(x,y)f_4(x',y')\\
&    \g_{\lambda t^\alpha,\p}(x)\g_{\lambda t^\alpha,\p}(x')\h_{t^\beta,\q+\r t^\beta}(y)\h_{t^\beta,\q+\r t^\beta}(y')   dx\,dy\,dx'\,dy'\,d\p\,d\q\,\frac{dt}{t}.
\end{align*}
Moreover, for any collections of rectangles $\mathcal{Q}$ we define 
\begin{align*}
\Xi_{\mathcal{Q},\lambda,\r}(f_1,f_2,f_3,f_4) & = \pi  \sum_{Q\in \mathcal{Q}}\int_{Q} \int_{\R^4} f_1(x',y)f_2(x,y')f_3(x,y)f_4(x',y')\\
&\g_{\lambda \ell(Q)^\alpha,p}(x)\g_{\lambda \ell(Q)^\alpha,p}(x')\g_{\ell(Q)^\beta,q+r\ell(Q)^\beta}(y)\g_{\ell(Q)^\beta,q+r\ell(Q)^\beta}(y')   dx\,dy\,dx'\,dy'\,d\p\,dq.
\end{align*}

\begin{lemma}[Telescoping identity]
For every convex tree $\mathcal{T}$ and $\lambda>1$, $\r\in\R$ we have
\label{lemma:telescoping}
\begin{equation}\label{eqn:telescoping}
\alpha \Theta^{(1)}_{\mathcal{T},\lambda,\r}  + \beta \Theta^{(2)}_{\mathcal{T},\lambda,\r}  = \Xi_{\{Q_\mathcal{T}\},\lambda,\r} - \Xi_{\mathcal{L}(\mathcal{T}),\lambda,\r} + \mathcal{B}_{\mathcal{T},\lambda,\r},
\end{equation}
where $\mathcal{B}_{\mathcal{T},\lambda,\r}$ takes the form $\mathcal{B}_{\mathcal{T},\lambda,\r} = \sum_{Q\in\mathcal{T}} \mathcal{B}_{\{Q\},\lambda,\r}$ and satisfies \begin{equation}\label{eqn:barkest-final} |\mathcal{B}_{\mathcal{T},\lambda,\r}(f_1,f_2,f_3,f_4)| \le C_{\lambda,\r} |Q_\mathcal{T}| \prod_{j=1}^4 \mathcal{M}_\mathcal{T}(f_j)
\end{equation}
with a constant $C_{\lambda,\r}=O(\lambda^{11} (1+|\r|)^{11})$ not depending on $\mathcal{T}$.
\end{lemma}
The first two terms on the right-hand side also enjoy an estimate
\begin{align*}
|\Xi_{\{Q_\mathcal{T}\},\lambda,\r}(f_1,f_2,f_3,f_4)|+ |\Xi_{\mathcal{L}(\mathcal{T}),\lambda,\r} (f_1,f_2,f_3,f_4)| \le  C_{\lambda,\r} |Q_{\mathcal{T}}| \prod_{j=1}^4 \mathcal{M}_{\mathcal{T}}(f_j)
\end{align*}
Indeed, this follows from 
the estimates
\begin{equation}
    \label{domfct}
    \g_{\lambda} \lesssim \lambda^{10} \theta,\; \g_{1,\r} \lesssim (1+|\r|)^{10} \theta,\;|\h_{\lambda}| \lesssim \lambda^{10} \theta,\; |\h_{1,\r}| \lesssim (1+|\r|)^{10} \theta,
\end{equation}
the identity
$$\varphi_{t^\beta,\q+\r t^\beta} = (\varphi_{1,\r})_{t^\beta,\q},$$
and the following simple result.
\begin{lemma}  
\label{lem:bdbymax}
For any $(p,q,t)\in \R^3_+$ and non-negative Schwartz functions $f_1,f_2,f_3,f_4$ on $\R^2$,
\begin{align*}
  \int_{\R^{4}} & f_1(x',y)f_2(x,y')f_3(x,y)f_4(x',y')\\ &\theta_{t^\alpha,\p}(x)\theta_{t^\alpha,\p}(x') \theta_{t^\beta,\q}(y)\theta_{t^\beta,\q}(y')    dx\, dy\,dx'\,dy'   \leq \prod_{j=1}^4 \Big (f_j^{2} *  (\theta_{t^\alpha} \otimes \theta_{t^\beta})(\p,\q)\Big)^{1/2}.
\end{align*}
\end{lemma}

\begin{proof}
Interchanging the order of integration we write the left-hand side as
 \[
 \int_{\R^2} \Big( \int_{\R}  f_2(x,y')f_3(x,y)\theta_{t^\alpha,\p}(x) dx \Big) \Big( \int_{\R} f_1(x',y)f_4(x',y') \theta_{t^\alpha,\p}(x')dx' \Big)  \theta_{t^\beta,\q}(y)\theta_{t^\beta,\q}(y')   dy\,dy'   .
 \]
By the Cauchy-Schwarz inequality in $x,x'$ we obtain\chk
\begin{align*}
     \int_{\R^2} & 
      \Big( \int_{\R}  f_1(x,y)^2\theta_{t^\alpha,\p}(x) dx \Big)^{1/2}
     \Big( \int_{\R}  f_2(x,y')^2\theta_{t^\alpha,\p}(x) dx \Big)^{1/2}
    \\
&     \Big( \int_{\R}  f_3(x,y)^2\theta_{t^\alpha,\p}(x) dx \Big)^{1/2}
     \Big( \int_{\R}  f_4(x,y')^2\theta_{t^\alpha,\p}(x) dx \Big)^{1/2}
     \theta_{t^\beta,\q}(y)\theta_{t^\beta,\q}(y')   dy\,dy'  . 
\end{align*}
Another application of the Cauchy-Schwarz inquality, this time in $y,y'$ 
yields the desired bound.
\end{proof}

The proof of   Lemma \ref{lemma:telescoping} is contained in \S \ref{sec:telescoping}. In \S \ref{sec:tree} we use the telescoping identity to prove Proposition \ref{prop:tree}.
In \S \ref{sec:planting} we use the tree estimate to prove \eqref{estimate-cone} in the case $r>1, p,q>2$. Finally, in \S \ref{sec:fwcz} we obtain the range claimed in Proposition \ref{prop:twisted} using a fiber-wise Calder\'on--Zygmund decomposition.

\subsection{Proof of the tree estimate}\label{sec:tree}
Here we prove Proposition \ref{prop:tree} using Lemma \ref{lemma:telescoping}.
By splitting into real and imaginary parts and using multilinearity we may assume that the functions are real-valued.
We rewrite $\Lambda_\mathcal{T}^{(u,v)}(f_1,f_2,f_3,f_4)$ as 
\begin{align*}
\int_{\Omega_\mathcal{T}} c(t) \int_{\mathbb{R}^3} \Big ( \int_{\mathbb{R}}f_1(x',y) f_3(x,y)  \h_{t^\beta,\q}(y) & dy \Big )   \Big( \int_{\mathbb{R}}  f_2(x,y') f_4(x',y')
\h_{t^\beta,\q+\r t^\beta}(y')
 dy' \Big )\\ & \times \g_{t^\alpha,\p}(x)  (\widecheck{\varphi}_{1,u})_{t^\alpha,\p}(x')  \widecheck{\psi}_{1,v}(\r) dx\,dx'\,d\r\,d\p\,d\q\, \frac{dt}{t}.
\end{align*}
Taking the triangle inequality, invoking the bound $|c(t)|\le 1$ and applying the Cauchy--Schwarz inequality in $x,x',r,p,t$ we obtain the geometric mean of 
 \begin{align}\label{aftercs1}
\int_{\Omega_\mathcal{T}} \int_{\mathbb{R}^3} \Big ( \int_{\mathbb{R}}f_1(x',y) f_3(x,y)  \h_{t^\beta,\q}(y)  dy \Big )^2  
\g_{t^\alpha,\p}(x) |\widecheck{\varphi}_{1,u}|_{t^\alpha,\p}(x') |\widecheck{\psi}_{1,v}|(\r) dx\,dx'\,d\r\,d\p\,d\q\, \frac{dt}{t}
\end{align}
and
\begin{align}\label{aftercs2}
\int_{\Omega_\mathcal{T}}  \int_{\mathbb{R}^3}  \Big( \int_{\mathbb{R}} f_2(x,y')f_4(x',y') \h_{t^\beta,\q+\r t^\beta}(y') dy' \Big )^2 
\g_{t^\alpha,\p}(x) |\widecheck{\varphi}_{1,u}|_{t^\alpha,\p}(x') |\widecheck{\psi}_{1,v}|(\r) dx\,dx'\,d\r\,d\p\,d\q\, \frac{dt}{t}.
\end{align}
Let us first consider \eqref{aftercs1}.  Using the rapid decay of the bump functions we dominate 
\begin{equation}\label{eqn:dominatebygaussians}
| \g(x)\widecheck{\varphi}_{1,u}(x')| \lesssim_N (1+|u|)^{N} \int_1^\infty\g_{\lambda}(x) \g_{\lambda}(x') \lambda^{-N+1} d\lambda.
\end{equation} 
This is because $|\g(x)\widecheck{\varphi}_{1,u}(x')|\lesssim_N (1+|u|)^N (1+|(x,x')|)^{-N}$ and
\[ \int_1^\infty e^{-\lambda^{-2} (x^2+(x')^2)} \lambda^{-N-1} d\lambda= |(x,x')|^{-N} \int_{0}^{|(x,x')|} e^{-\omega^2} \omega^{N-1} d\omega. \]
We will use \eqref{eqn:dominatebygaussians} with, say $N=50$, to estimate \eqref{aftercs1}.
Expanding the square in \eqref{aftercs1} and integrating in $\r$  leads us to consider for each fixed $\lambda$,
\begin{align*}%
 \int_{\Omega_\mathcal{T}} \int_{\mathbb{R}^4}   f_1(x',y) f_3(x,y) f_1(x',y')  f_3(x,y') \g_{\lambda t^\alpha,\p}(x)  \g_{\lambda t^\alpha,\p}(x')  \h_{t^\beta,\q} (y) {\h}_{t^\beta,\q} (y')  dx\,dx'\,dy\,dy'\,d\p\,d\q\, \frac{dt}{t},
\end{align*}
which can be recognized as $\Theta_{\mathcal{T},\lambda,1}^{(2)}(f_1,f_3,f_3,f_1)$.
The identity from  Lemma \ref{lemma:telescoping} yields
\begin{align*}
   \alpha\Theta_{\mathcal{T},\lambda,1}^{(1)}(f_1,f_3,f_3,f_1) + \beta\Theta_{\mathcal{T},\lambda,1}^{(2)}(f_1,f_3,f_3,f_1)  = R_{\mathcal{T},\lambda}(f_1,f_3,f_3,f_1),
\end{align*}
where $R_{\mathcal{T},\lambda} = \Xi_{\{Q_\mathcal{T}\},\lambda,1} - \Xi_{\mathcal{L}(\mathcal{T}),\lambda,1} + \mathcal{B}_{\mathcal{T},\lambda,1}$ and satisfies the bound
\[|R_{\mathcal{T},\lambda}(f_1,f_3,f_3,f_1)|\lesssim \lambda^{11} |Q_\mathcal{T}| \mathcal{M}_\mathcal{T}(f_1)^2 \mathcal{M}_\mathcal{T}(f_3)^2.\]
We proceed with bounding $\Theta^{(1)}_{\mathcal{T},\lambda,1}(f_1,f_3,f_3,f_1)$, which is
 \begin{align*}%
 \int_{\Omega_\mathcal{T}} \int_{\mathbb{R}^4}   f_1(x',y) f_3(x,y) f_1(x',y')  f_3(x,y') \h_{\lambda t^\alpha,\p}(x)  \h_{\lambda t^\alpha,\p}(x')  \g_{t^\beta,\q} (y) {\g}_{t^\beta,\q} (y')  dx\,dx'\,dy\,dy'\,d\p\,d\q\, \frac{dt}{t}.
\end{align*}
Applying the Cauchy--Schwarz inequality again, we obtain the geometric mean of 
\begin{align*}
\int_{\Omega_{\mathcal{T}}} \int_{\mathbb{R}^2} \Big ( \int_{\mathbb{R}}f_1(x',y) f_1(x',y')  \h_{ \lambda t^\alpha,\p}(x')  dx' \Big )^2  \g_{t^\beta,\q}(y)\g_{t^\beta,\q}(y')  dy\,dy'\,d\p\,d\q\,\frac{dt}{t}
\end{align*}
and
\begin{align*}
\int_{\Omega_{\mathcal{T}}} \int_{\mathbb{R}^2} \Big ( \int_{\mathbb{R}}f_3(x,y) f_3(x,y')  \h_{\lambda t^\alpha,\p}(x)  dx \Big )^2  \g_{t^\beta,\q}(y)\g_{t^\beta,\q}(y')  dy\,dy'\,d\p\,d\q\,\frac{dt}{t}.
\end{align*}
Expanding the square, the first term can be recognized as $\Theta^{(1)}_{\mathcal{T},\lambda,1}(f_1,f_1,f_1,f_1)$, while the second term is $\Theta^{(1)}_{\mathcal{T},\lambda,1}(f_3,f_3,f_3,f_3)$. Applying Lemma \ref{lemma:telescoping} again we have 
\[\alpha\Theta^{(1)}_{\mathcal{T},\lambda,1}(f_j,f_j,f_j,f_j) + \beta\Theta^{(2)}_{\mathcal{T},\lambda,1}(f_j,f_j,f_j,f_j) = R_{\mathcal{T},\lambda}(f_j,f_j,f_j,f_j), \]
where \[|R_{\mathcal{T},\lambda}(f_j,f_j,f_j,f_j)| \lesssim \lambda^{11} |Q_\mathcal{T}| \mathcal{M}_\mathcal{T}(f_j)^4. \]
Since $\Theta^{(2)}_{\mathcal{T},\lambda,1}(f_j,f_j,f_j,f_j)\ge 0$ we have now proved that 
\[ |\Theta^{(1)}_{\mathcal{T},\lambda,1}(f_1,f_3,f_3,f_1)| \lesssim \lambda^{11} |Q_\mathcal{T}| \mathcal{M}_\mathcal{T}(f_1)^2 \mathcal{M}_\mathcal{T}(f_3)^2, \]
which completes our estimate of \eqref{aftercs1}.
It remains to treat \eqref{aftercs2}. 
Using \eqref{eqn:dominatebygaussians} and rapid decay of $\widecheck{\psi}_{1,v}$ we estimate \eqref{aftercs2} by a uniform constant times
\begin{align*}
(1+|u|)^{50} (1+|v|)^{50}  \int_{\R} \int_1^\infty  \Theta^{(2)}_{\mathcal{T},\lambda,\r}(f_4,f_2,f_2,f_4)  \lambda^{-49}d\lambda  (1+|\r|)^{-50}  d\r.
\end{align*}
Performing the analogous steps as for \eqref{aftercs1} gives
\[\Theta^{(2)}_{\mathcal{T},\lambda,\r}(f_4,f_2,f_2,f_4) \lesssim (\lambda(1+|\r|))^{11} |Q_\mathcal{T}| \mathcal{M}_\mathcal{T}(f_2)^2 \mathcal{M}_\mathcal{T}(f_4)^2. \]
In the end it remains to integrate the bounds in $\lambda$ and  $r$.
This  finishes the proof of the tree estimate.
 
\subsection{Proof of telescoping identity}\label{sec:telescoping}
Here we prove Lemma \ref{lemma:telescoping}.
Let us begin with the case that the tree $\mathcal{T}$ consists of a single dyadic rectangle $Q=I\times J$.
Denote $\ell(Q)$ by $2^k$. Then \[\Omega_\mathcal{T} = [2^{k-1}, 2^k] \times I\times J. \]
Fix $(x,x',y,y')\in \R^4$. 
For an interval $I=[a,b]$ and a function $f$ on $I$ we write
\[ [f]_I = [f(t)]_{t\in I} = f(b) - f(a). \]
By the fundamental theorem of calculus in $t$ we obtain \chk
\begin{align}\nonumber
& \int_{2^{k-1}}^{2^k} -t\partial_t(\g_{\lambda t^\alpha,\p}(x)\g_{\lambda t^\alpha,\p}(x')\g_{t^\beta,\q+\r t^\beta}(y)\g_{t^\beta,\q+\r t^\beta}(y') ) \frac{dt}{t}\\ \label{ftc1}
& = [ -\g_{\lambda t^\alpha,\p}(x)\g_{\lambda t^\alpha,\p}(x')\g_{t^\beta,\q+\r t^\beta}(y)\g_{t^\beta,\q+\r t^\beta}(y') ]_{t\in [2^{k-1},2^k]}.
\end{align}
We will use the product rule on the left-hand side of the identity, yielding a sum of four terms. 
To analyze each term we first make some preliminary computations.

For a function $\phi$ on $\R$ we denote ${\phi}^\sharp (u)=u\phi(u)$. Then we observe the identity \chk
\begin{align*}%
-t\partial_t \phi_{t^\alpha}(u) = \alpha 	\phi_{t^\alpha}(u) + \alpha ((\phi')^\sharp)_{t^\alpha}(u) .
\end{align*}
Integrating by parts we deduce for every bounded interval $I$ and $t\in (0,\infty)$, \chk
\begin{align*}\nonumber
\int_{I}  ((\phi')^\sharp)_{t^\alpha,\p}(x)  \phi_{t^\alpha,\p}(x') \,  d\p  = & 
-t^\alpha \big[ (\phi^\sharp)_{t^\alpha,\p}(x) \,  \phi_{t^\alpha,\p}(x')\big ]_{\p\in I}\\ \nonumber
& - \int_{I}  \phi_{t^\alpha,\p}(x) \phi_{t^\alpha,\p}(x') d\p -  \int_{I}   (\phi^\sharp)_{t^\alpha,\p}(x)  (\phi')_{t^\alpha,\p}(x') d\p.
\end{align*}
Combining the last two identities we obtain \chk
\begin{align}\label{eqn:telescoping-pre}
 \int_{I} \big(  -t\partial_t \phi_{t^\alpha,\p}(x)\big ) \phi_{t^\alpha,\p}(x')d\p   = &
 -\alpha t^\alpha \big[ (\phi^\sharp)_{t^\alpha,\p}(x) \,  \phi_{t^\alpha,\p}(x')   \big ]_{\p\in I}  \\\nonumber
& -  \alpha \int_{I}   (\phi^\sharp)_{t^\alpha,\p}(x)  (\phi')_{t^\alpha,\p}(x') d\p.\label{piinp}
\end{align}
Note that \chk 
\begin{equation}\label{eqn:gh-idts}
((\g_{1,s})_{\mu})' = \mu^{-1}(\h_{1,s})_{\mu}\quad\text{and}\quad((\g_{1,s})_{\mu})^\sharp(x)=  -\mu (2\pi)^{-1}(\h_{1,s})_\mu + \mu s (\g_{1,s})_\mu.
\end{equation} We will use these relations in the cases $(\mu,s)=(\lambda,0)$ and $(\mu,s)=(1,r)$.
From \eqref{eqn:telescoping-pre}, \eqref{eqn:gh-idts}, and the product rule,
\begin{align}\label{eqn:telescoping-pre1}
\int_Q \eqref{ftc1}\,d\p\,dq  &= \alpha \pi^{-1} \int_{\Omega_\mathcal{T}} \h_{\lambda t^\alpha,\p}(x) \h_{\lambda t^\alpha,\p}(x') \g_{t^\beta,\q+\r t^\beta}(y) \g_{t^\beta,\q+\r t^\beta}(y') d\p\,d\q\,\frac{dt}t   \\\nonumber
& + \beta \pi^{-1} \int_{\Omega_\mathcal{T}} \g_{\lambda t^\alpha,\p}(x) \g_{\lambda t^\alpha,\p}(x') \h_{t^\beta,\q+\r t^\beta}(y) \h_{t^\beta,\q+\r t^\beta}(y') d\p\, d\q\, \frac{dt}t + \mathscr{B} ,
\end{align}
where the remainder term $\mathscr{B}$ consists of terms corresponding to the ``bark'' of the tree, as well as two additional terms that vanish when $r=0$: it is of the form $\mathscr{B} = \mathscr{B}_{1} + \mathscr{B}_2$, where
\begin{align}\label{eqn:barkterm}
\mathscr{B}_1 &= \alpha \lambda (2\pi)^{-1} \int_{2^{k-1}}^{2^k} \int_{J} [ \h_{\lambda t^\alpha,\p}(x) \g_{\lambda t^\alpha,\p}(x') ]_{\p\in I}(\g_{1,\r})_{t^\beta,\q}(y) (\g_{1,\r})_{t^\beta,\q}(y') d\q\,t^{\alpha-1} dt\\\nonumber
& +  \alpha \lambda (2\pi)^{-1} \int_{2^{k-1}}^{2^k} \int_{J} [ \g_{\lambda t^\alpha,\p}(x) \h_{\lambda t^\alpha,\p}(x') ]_{\p\in I}(\g_{1,\r})_{t^\beta,\q}(y) (\g_{1,\r})_{t^\beta,\q}(y') d\q\,t^{\alpha-1} dt\\\nonumber
& +    \beta \lambda (2\pi)^{-1} \int_{2^{k-1}}^{2^k} \int_{I}  \g_{\lambda t^\alpha,\p}(x) \g_{\lambda t^\alpha,\p}(x') [(\h_{1,\r})_{t^\beta,\q}(y) (\g_{1,\r})_{t^\beta,\q}(y')]_{\q\in J} d\p\,t^{\beta-1} dt\\\nonumber
& +  \beta \lambda (2\pi)^{-1} \int_{2^{k-1}}^{2^k} \int_{I}  \g_{\lambda t^\alpha,\p}(x) \g_{\lambda t^\alpha,\p}(x') [(\g_{1,\r})_{t^\beta,\q}(y) (\h_{1,\r})_{t^\beta,\q}(y')]_{\q\in J} d\p\,t^{\beta-1} dt,
\end{align} 
\begin{align}\nonumber
    \mathscr{B}_2 &= - 2 \beta r \int_{2^{k-1}}^{2^k} \int_{I} \g_{\lambda t^\alpha,\p}(x) \g_{\lambda t^\alpha,\p}(x') [ (\g_{1,\r})_{t^\beta,\q}(y) (\g_{1,\r})_{t^\beta,\q}(y') ]_{\q\in J} d\p\, t^{\beta-1}\, dt\\\nonumber%
& - \beta r \int_{\Omega_\mathcal{T}} \g_{\lambda t^\alpha,\p}(x) \g_{\lambda t^\alpha,\p}(x') \Big( (\g_{1,\r})_{t^\beta,\q}(y) (\h_{1,\r})_{t^\beta,\q}(y') +
(\h_{1,\r})_{t^\beta,\q}(y)
(\g_{1,\r})_{t^\beta,\q}(y') \Big)\,d\p\,d\q\,\frac{dt}t.
\end{align}
By another integration by parts in $q$ we see that %
\[
\mathscr{B}_2= -\beta r \int_{2^{k-1}}^{2^k} \int_{I}    \g_{\lambda t^\alpha,\p}(x)   \g_{\lambda t^\alpha,p }(x') [(\g_{1,\r})_{t^\beta,\q}(y)  (\g_{1,\r})_{t^\beta,\q}(y')]_{\q\in J}
   d\p\,t^{\beta-1}\,dt.
\]

So far we have proved that in the case $\mathcal{T}=\{Q\}$, the identity \eqref{eqn:telescoping} holds with $\mathcal{B}_{\{Q\},\lambda,\r}$ taking the form
\[ \int_{\R^4} f_1(x',y) f_2(x,y') f_3(x,y) f_4(x',y') \mathscr{B}\, dx\, dx'\, dy\, dy' \]
where $\mathscr{B}$ depends on $x,x',y,y',\lambda,r$ and is given as above.

Given a general convex tree $\mathcal{T}$ we now {\em define} $\mathcal{B}_{\mathcal{T},\lambda,\r} = \sum_{Q\in\mathcal{T}} \mathcal{B}_{\{Q\},\lambda,\r}$
and observe that \eqref{eqn:telescoping} continues to hold.
This is because convexity of the tree allows us to write the right hand side of \eqref{eqn:telescoping-pre} as 
\[\Xi_{\{Q_\mathcal{T}\},\lambda,\r} - \Xi_{\mathcal{L}(\mathcal{T}),\lambda,\r} + \mathcal{B}_{\mathcal{T},\lambda,\r}    =  \sum_{\substack{Q\in \mathcal{T}}} \Big(\;   \Xi_{\{Q\},\lambda,\r}  -   \sum_{Q'\textup{ child of }Q} \Xi_{\{Q'\},\lambda,\r} \; + \mathcal{B}_{\{Q\},\lambda,\r}\Big ).\]
Applying the single rectangle case of \eqref{eqn:telescoping} this becomes
\[ \sum_{Q\in\mathcal{T}} \Big( \alpha \Theta_{\{Q\},\lambda,\r}^{(1)} + \beta \Theta^{(2)}_{\{Q\},\lambda,\r} \Big) = \alpha \Theta_{\mathcal{T},\lambda,\r}^{(1)} + \beta \Theta^{(2)}_{\mathcal{T},\lambda,\r}. \]
It remains to prove \eqref{eqn:barkest-final}. %
From the above computation we know that $\mathcal{B}_{\mathcal{T},\lambda,\r}$ is a linear combination of finitely many terms each of which either takes the form
\begin{align}\label{eqn:barkestpf1}
\sum_{k\in \mathbb{Z}}  \sum_{\substack{Q=I\times J\in \mathcal{T}\\ \ell(Q)=2^{k}}}   & \int_{2^{k-1}}^{2^{k}} \int_{J}  \int_{\R^{4}} f_1(x',y)f_2(x,y')f_3(x,y)f_4(x',y')\\\nonumber &  [(\phi_1)_{t^\alpha,\p}(x)(\phi_2)_{t^\alpha,\p} (x')]_{\p\in I} (\phi_3)_{t^\beta,\q}(y)(\phi_4)_{t^\beta,\q}(y') dx\,dx'\,dy\,dy'\, d\q\, t^{\alpha-1} dt ,
\end{align}
or
\begin{align}\label{eqn:barkestpf2}
\sum_{k\in \mathbb{Z}}  \sum_{\substack{Q=I\times J\in \mathcal{T}\\ \ell(Q)=2^{k}} } &  \int_{2^{k-1}}^{2^{k}}  \int_{I}   \int_{\R^{4}} f_1(x',y)f_2(x,y')f_3(x,y)f_4(x',y')\\\nonumber &  (\phi_1)_{t^\alpha,\p}(x)(\phi_2)_{t^\alpha,\p} (x') [(\phi_3)_{t^\beta,\q}(y)(\phi_4)_{t^\beta,\q}(y')]_{\q\in J} dx\,dx'\,dy\,dy'\, d\p\, t^{\beta-1} dt ,
\end{align}
with coefficients that are $O(\lambda (1+|\r|))$ (with constants depending only on $\alpha,\beta$) and each function $\phi_i$ being one of $\g_{\lambda,0}, \g_{1,\r},\h_{\lambda,0},\h_{1,\r}$.
We claim that \eqref{eqn:barkestpf1} and \eqref{eqn:barkestpf2} are both bounded by
\[C_{\lambda,\r} |Q_{\mathcal{T}}| \prod_{j=1}^4 \mathcal{M}_{\mathcal{T}}(f_j).\]
Since the proofs are identical up to notational changes we only give the proof of this claim for terms of the form \eqref{eqn:barkestpf1}.
Summing over $Q=I\times J$ and noting the telescoping sum in $I$ we estimate the expressions in question up to an absolute  constant by
\begin{align*}
 \sum_{k\in \mathbb{Z}} \int_{2^{k-1}}^{2^{k}} \sum_{J:\, |J|=2^{k\beta}}  \int_{J} 
& \sum_{p: \{p\} \times J \subseteq \partial T_k}   \int_{\R^{4}}f_1(x',y)f_2(x,y')f_3(x,y)f_4(x',y')\\
& (\phi_1)_{t^\alpha,\p}(x)(\phi_2)_{t^\alpha,\p} (x') (\phi_3)_{t^\beta,\q}(y)(\phi_4)_{t^\beta,\q}(y') dx\,dx'\,dy\,dy'\, d\q\, t^{\alpha-1}\, dt.
\end{align*}
Here $\partial T_k$ denotes the topological boundary of
$T_k = \cup_{Q\in \mathcal{T}, \, \ell(Q)= 2^{k}} Q$
in $\R^2$. By Lemma \ref{lem:bdbymax} and the pointwise estimates \eqref{domfct}, this can be further estimated by $O((\lambda(1+|\r|))^N)$ times
\begin{align*}
& \Big ( \prod_{j=1}^4 \mathcal{M}_\mathcal{T}(f_j) \Big )  \sum_{k\in \mathbb{Z}} 2^{k(\alpha+\beta)} \sum_{J:\, |J|=2^{k\beta }}  
\# \{p: \{p\}\times J \subseteq \partial T_k\} \, %
\\
& \lesssim \Big ( \prod_{j=1}^4 \mathcal{M}_\mathcal{T}(f_j) \Big )  \sum_{k\in \mathbb{Z}} 2^{k(\alpha+\beta)} \# \{\partial T_k \cap (2^{k\alpha} \Z \times 2^{k\beta}\Z)\}.
\end{align*}
To finish the proof it remains to show that
\begin{equation}\label{eqn:treecountingest}
\sum_{k\in \mathbb{Z}}  2^{k (\alpha+\beta)} \# \{\partial T_k \cap ((2^{\alpha k}\Z)\times (2^{\beta k}\Z) )\}\lesssim |Q_{\mathcal{T}}|.
\end{equation}
If $(\p,\q)\in \partial T_k \cap ((2^{\alpha k}\Z)\times (2^{\beta k}\Z)$, then there exist $\epsilon,\epsilon'\in \{\pm 1\}$ so that
\[ T_k\cap (p + \epsilon 2^{k\alpha}, p) \times (q + \epsilon' 2^{k\beta}, q) = \emptyset. \]
(We write $(a,b)$ to denote the open interval $(\mathrm{min}(a,b),\mathrm{max}(a,b))$.) This is because $T_k$ is a union of dyadic rectangles $Q$ with $\ell(Q)=2^k$. Without loss of generality we only consider the case $\epsilon=\epsilon'=-1$.
For each such point consider  the open dyadic rectangle 
\begin{align*}
Q(\p,\q,k)=  (\p-2^{k\alpha }, \p-2^{(k-1)\alpha}(2^\alpha-1)) \times  (\q-2^{k\beta}, \q-2^{(k-1)\beta}(2^\beta-1))
\end{align*}
which has   area $2^{(k-1)(\alpha+\beta)}$. We claim that rectangles of this form are pairwise disjoint. This implies \eqref{eqn:treecountingest}, because each $Q(\p,\q,k)$ is contained in the union of $Q_{\mathcal{T}}$ and neighboring dyadic rectangles of the same size. To see the claim, suppose that  $Q(\p,\q,k)$ and $Q(\p',\q',k')$ intersect in a set of positive measure. If $k=k'$, then they must coincide since they are dyadic and of the same scale.
So suppose that $k<k'$,   hence  $Q(\p,\q,k)$ is contained in $Q(\p',\q',k')$. Then the point $(\p,\q)$ is contained in  $(\p'-2^{k'\alpha},p')\times (\q'-2^{k'\beta}, \q')$, which is disjoint from $T_{k'}$.  This shows that $(\p,\q)\in T_k$  but $(\p,\q)\notin T_{k'}$,
contradicting convexity of $\mathcal{T}$.
The proof of Lemma \ref{lemma:telescoping} is now complete.
 
 \subsection{Combining the trees}\label{sec:planting} %
Here we complete the proof of Proposition \ref{prop:twisted} in the case $2<p,q<\infty$, $1<r<2$. This closely follows arguments from \cite{Thi95}, \cite[\S 4]{Kov12}. Write $p_1=p$, $p_2=q$, and $p_3=r'=\tfrac{r}{r-1}$. By homogeneity we  normalize 
\[ \|f_j\|_{p_j} = 1 \]
for all $j=1,2,3$. 
Fix a finite collection $\mathcal{Q}_0$ of dyadic rectangles. For every tuple of integers  $\mathbf{n} = (n_1,n_2,n_3)$ we define the collection of dyadic cubes
\[ \mathcal{P}_{\mathbf{n}} = \{ Q\in\mathcal{Q}_0 \,:\, 2^{n_j-1}< \sup_{\mathcal{Q}_0\ni Q'\supseteq Q} \mathcal{M}_{\{Q'\}}(f_j) \le 2^{n_j}\;\text{for}\;j=1,2,3 \},\]
Let $\mathcal{P}^{\mathrm{max}}_\mathbf{n}$ be the collection of maximal cubes in $\mathcal{P}_{\mathbf{n}}$ with respect to  inclusion. For $Q\in \mathcal{P}^{\mathrm{max}}_\mathbf{n}$   set
$$\mathcal{T}_Q =\{ Q'\in \mathcal{P}_\mathbf{n} : Q' \subseteq Q\}.$$
where $\mathcal{M}_{\mathcal{Q}}$  has been defined in \eqref{def:mq}.
This is a finite convex tree with root $Q$ \chk and for different $Q\in \mathcal{P}^{\mathrm{max}}_\mathbf{n}$ the corresponding trees are disjoint. \chk
By the tree estimate (that is, Proposition \ref{prop:tree}), \chk
\begin{align*}
| \Lambda^{(u,v)}_{\mathcal{T}_Q}(f_1,f_2,f_3,1) | \le C_{u,v}   |Q| \prod_{j=1}^3 \mathcal{M}_{\mathcal{T}_Q}(f_j) %
\leq  C_{u,v} |Q| \, 2^{n_1+n_2+n_3}.
\end{align*}
Note that $\mathcal{M}_{\{{Q'}\}}(f_j)>0$ since $f_j$ does not vanish identically. The region of the integration of our form 
is therefore partitioned as follows:
\begin{align*}
\Omega_{\mathcal{Q}_0}
= \bigcup_{\mathbf{n}\in \Z^3}  \bigcup_{Q\in \mathcal{P}^\mathrm{max}_\mathbf{n}} {\Omega_{\mathcal{T}_Q}}
\end{align*} 
This yields 
\begin{align} \label{bdbyk}
|\Lambda^{(u,v)}_{\mathcal{Q}_0}(f_1,f_2,f_3,1)|\le C_{u,v}    \sum_{\mathbf{n}\in\Z^3} \Big ( \sum_{Q\in \mathcal{P}^\mathrm{max}_\mathbf{n}}|Q| \Big )  \, 2^{n_1+n_2+n_3} . \nonumber
\end{align}
For a function $f$ on $\R^2$ we  consider  the maximal function \[\mathcal{M}_{\mathrm{HL}} f(x,y) =\sup_{t>0}\Big( |f|^2*(\theta_{t^\alpha}\otimes\theta_{t^\beta})(x,y) \Big)^{1/2}.\]
It is pointwise dominated by an anisotropic variant of the Hardy--Littlewood maximal function, which is bounded on ${L}^p$ for $2<p\leq\infty$. 
We split $\Z^{3}=\cup_{j=1}^3\mathcal{N}_{j} $, 
where 
\[\mathcal{N}_{j}=\{\mathbf{n}=(n_1,n_2,n_3) : p_{j}n_{j}\geq p_{j'}n_{j'}\;\; \mathrm{for\;every}\; 1\leq j' \leq 3 \}.\]
For every $(x,y)\in Q\in \mathcal{P}_{\mathbf{n}}$ we have  that $\mathcal{M}_\mathrm{HL}f_{j}(x,y)\ge c\, 2^{n_{j}}$, where $c\in (0,\infty)$ is a constant. %
Also, the cubes in $\mathcal{P}^{\max}_{\mathbf{n}}$ are by definition disjoint, so
\begin{align}\label{pt2}
\sum_{Q\in \mathcal{P}^{\max}_{\mathbf{n}}}|Q| = \Big |\bigcup_{Q\in \mathcal{P}^{\max}_{\mathbf{n}}} Q \, \Big | \le |\{\mathcal{M}_\mathrm{HL} f_{j} \ge c\, 2^{n_{j}} \}|.
\end{align}
for every $j=1,2,3$.
Combining everything and using $\sum_{j=1}^3 p_j^{-1} =1$  we obtain that \[\Lambda^{(u,v)}_{\mathcal{Q}_0}(f_1,f_2,f_3,1)\] is bounded up to a constant $C_{u,v}$ by a sum over $j=1,2,3$ of\chk
\[\sum_{\mathbf{n} \in \mathcal{N}_{j}} 2^{n_1+n_2+n_3} |\{\mathcal{M}_\mathrm{HL} f_{j} \geq c\, 2^{n_{j}} \}|  
   =  \sum_{n_{j} \in \Z} 2^{p_{j}n_{j}} |\{\mathcal{M}_\mathrm{HL}f_{j} \geq c\,2^{n_{j}} \}|\prod_{j'\neq j}\sum_{\substack{n_{j'}\in\Z\\ p_{j'}n_{j'}\leq p_{j}n_{j}}} 2^{n_{j'}-\frac{p_{j}n_{j}}{p_{j'}}}
  \]
\[\lesssim  \sum_{n_{j} \in \Z} 2^{p_{j}n_{j}} |\{\mathcal{M}_\mathrm{HL}f_{j} \geq c\, 2^{n_{j}} \}| \lesssim  \|\mathcal{M}_\mathrm{HL} f_{j} \|_{p_{j}}^{p_{j}} \lesssim  \| f_{j}\|_{p_{j}}^{p_{j}} \lesssim 1.\]
An application of the monotone convergence theorem removes the restriction to finite collections $\mathcal{Q}_0$ and finishes the proof.

\subsection{Fiber-wise Calder\'on--Zygmund decomposition}\label{sec:fwcz}
In this section we extend the range of exponents to the final range.
To achieve this we employ the fiber-wise Calder{\'o}n--Zygmund decomposition of Bernicot \cite{Ber12}.
Since the argument below will not depend on the cancellation properties of $\varphi,\psi$
it will suffice to prove bounds for the operator of the form 
\begin{align*}
U(f_1,f_2)(x,y) = \int_0^\infty c(t) (f_1*_1\phi_{t^\alpha,t^\alpha u})(x,y)  (f_2*_2\rho_{t^\beta,t^\beta u})(x,y)   \frac{dt}{t}
\end{align*}
where $\phi,\rho$ are   Schwartz functions  
and $|c(t)|\leq 1$. 
By the result from \S \ref{sec:planting}, we  may assume the bound
\begin{align*}
\|U(f_1,f_2)\|_{r_0} \le C_{u,v} \|f_1\|_{p_0} \|f_2\|_{q_0}.
\end{align*}
whenever $2<p_0,q_0<\infty$, $1<r_0<2$.
We will apply the fiber-wise Calder\'on--Zygmund decomposition in the fiber $f_1(\cdot,y)$ to show the weak $L^p\times L^q \rightarrow L^{r,\infty}$   estimates whenever $1\leq p \leq  p_0, q=q_0,$ and $p^{-1}+q^{-1}=r^{-1}$.
That is,  for every $\lambda>0$,   
\begin{align*}
|\{(x,y)\in \R^2: |U(f_1,f_2)(x,y)|>\lambda\}| \lesssim_{p,q} C_{u,v} \|f_1\|^r_p \|f_2\|^r_q\lambda^{-r}.
\end{align*}
Multilinear interpolation (see for instance \cite{MS13}) then yields strong-type estimates in the range \begin{equation}\label{fwrange-1}
1<p< \infty,\, 2<q<\infty,\, \tfrac{1}{2}<r<2.
\end{equation}
Analogous considerations in the second fiber $f_2(x,\cdot)$ finally yield estimates whenever
\[1<p <\infty,\,1<q<\infty,\, \tfrac{1}{2}<r<2.\]
Fix exponents $2<p_0,q_0<\infty$ and
choose   $p,q,r$ such that $1\leq p \leq p_0,\,q=q_0,\, 1/p+1/q=1/r$. Fix the functions $f_1,\,f_2$. By homogeneity we may assume 
$$\|f_1\|_p=\|f_2\|_q=1.$$
Performing a Calder{\'o}n-Zygmund decomposition of the function $x\mapsto f_{1,y}(x)=  f_1(x,y)$ for each $y$ at level $\lambda^{r/p}$ we obtain disjoint dyadic intervals $\{I_{y,j}\}$
such that  $|f_{1,y}(x)| \leq \lambda^{r/p}$ for a.e. $x\not\in \cup_j I_{y,j}$, $$|\cup_{j} I_{y,j}|\leq \lambda^{-r} \|f_{1,y}\|^p_{L^{p}(\mathbb{R})} $$ and $\lambda^r <|I_{y,j}|^{-1/p} \|f_{1,y}\|_{L^p(I_{y,j})} \leq 2\lambda^{r}$. We split $f_{1,y} = g_y + b_y$, where  $b_y= \sum_j b_{y,j}$ and 
\begin{align*}
g_y(x) = \left\{ \begin{array}{ll}
f_{1,y}(x) & ;x\not\in \cup_j I_{y,j}\\
|I_{y,j}|^{-1/p}  \|f_{1,y}\|_{L^p(I_{y,j})} & ;x\in I_{y,j}
\end{array} \right .
\end{align*}
and  
$$b_{y,j}(x) = \Big( f_1(x) - |I_{y,j}|^{-1/p}  \|f_{1,y}\|_{L^p(I_{y,j})} \Big )\, \mathbf{1}_{I_{y,j}}. $$

Denote $g = (x,y) \mapsto g_y(x)$  and   consider first the good part $U(g,f_2)$.    From   
$\|g\|_{\infty} \leq \lambda^{r/p}$ and $\|g\|_{p} \leq 1$  it follows that for every $p_1>p$ 
$$\|g\|_{p_1}\leq \lambda^{r(1/p-1/p_1)}.$$
From the strong $L^{p_0}\times L^{q}$ to $L^{r_0}$
we thus obtain 
\begin{align*}
C_{u,v}^{-1}|\{(x,y)\in \R^2: |U(g,f_2)(x,y)|>\lambda \} |\lesssim \lambda^{-r_0}  \|U(g,f_2)\|^{r_0}_{{r_0}} & \lesssim   \lambda^{-r_0} \|g\|^{r_0}_{{p_0}}\|f_2\|^{r_0}_{q} \\
& \lesssim      \lambda^{-r_0+ r\cdot r_0(1/p-1/p_0)} =    \lambda^{-r},
\end{align*}
which is the desired estimate for the good part.
It remains to consider the bad part $U(b,f_2)$, where $b =(x,y) \mapsto b_y(x)$.  Since
\begin{align*}
|\{(x,y) \in    (\cup_j 2I_{y,j} ) \times \mathbb{R} : |U(b,f_2)| > \lambda \}| \leq \int_{\mathbb{R}} |\cup_j 2I_{y,j}| dy \lesssim  \lambda^{-r} \|f_2\|_{L^p(\mathbb{R}^2)}^p  = \lambda^{-r},
\end{align*}
it suffices to estimate
$$|\{(x,y) \in    (\cup_j 2I_{y,j} )^c \times \mathbb{R} : |U(b,f_2)| > \lambda \}|.$$
Note that 
$$(b(\cdot, y) *_1 \phi_{t^\alpha,t^\alpha u})(x) = b_y * \phi_{t^\alpha,t^\alpha u} (x) = \sum_j b_{y,j} * \phi_{t^\alpha,t^\alpha u}(x).$$
Denote by $c_{y,j}$ the center of $I_{y,j}$.
By the mean zero property of $b_{y,j}$ and the fact that they are supported in $I_{y,j}$ we obtain  
\begin{align*}
| b_{y,j} * \phi_{t^\alpha,t^\alpha u} (x) | &\leq \int_{I_{y,j}} |b_{y,j}(w)| |\phi_{t^\alpha,t^\alpha u}(x-w) - \phi_{t^\alpha,t^\alpha u}(x-c_{y,j}) | dw\\
&\lesssim C_{u,v}  \int_{I_{y,j}} |b_{y,j}(w)||I_{y,j}|t^{-2\alpha} (1+t^{-\alpha}|x-c_{y,j,w}|)^{{-100}}  dw
\end{align*}
for a suitable $\min(w,c_{y,j})\leq  c_{y,j,w} \leq \max(w,c_{y,j}),$ 
by the mean value theorem. Since $|x-c_{y,j,w}|>|x-c_{y,j}|/2$ we bound the last display by 
\begin{align*}
& \lesssim C_{u,v} |I_{y,j}|t^{-2\alpha}  (1+t^{-\alpha} |x-c_{y,j}|)^{{-100}}  \int_{I_{y,j}} |b_{y,j}(w)| dw.
\end{align*}
If   $p^{-1}+(p')^{-1}=1$, from $ \|b_{y,j}\|_{L^{p}(\mathbb{R})} \leq |I_{y,j}|^{1/p} \lambda^{r/p}$  we have by H\"older's inequality
$$\|b_{y,j}\|_{L^{1}(\mathbb{R})} \leq |I_{y,j}|^{\frac{1}{p'}}  \|b_{y,j}\|_{L^{p}(\mathbb{R})} \leq  |I_{y,j}| \lambda^{r/p}  .$$
Thus we obtain the bound
\begin{align*}
 | b_{y,j} * \phi_{t^\alpha,t^\alpha u} (x) |\lesssim C_{u,v} \lambda^{r/p}  |I_{y,j}|^2 t^{-2\alpha}  (1+t^{-\alpha} |x-c_{y,j}|)^{-10}   . 
\end{align*}
Integrating  in $t$ we have
\begin{align*}
&   |I_{y,j}|^2 |x-c_{y,j}|^{-2}  \int_0^\infty t^{-2\alpha}    |x-c_{y,j}|^{2}  (1+t^{-\alpha} |x-c_{y,j}|)^{-10}  \frac{dt}{t}
 \\
 & \lesssim   |I_{y,j}|^2 |x-c_{y,j}|^{-2}\\
 & \lesssim \Big( 1+ \frac{|x-c|}{|I_{y,j}| } \Big )^{-2},
\end{align*}
where the last estimate follows since
  $|I_{y,j}|\leq 2|x-c_{i,j}|$.
Therefore, we have shown that whenever $x\in (\cup_j I_{y,j})^c$ one has the bound
\begin{align*}
|U(b,f_2)|\lesssim C_{u,v} \lambda^{r/p} h(x,y) \mathcal{M}^{(2)}f_2(x,y),
\end{align*}
where 
\begin{align*}
h(x,y) = \sum_{j} \Big( 1+ \frac{|x-c|}{|I_{y,j}|} \Big )^{-2}
\end{align*}
and $\mathcal{M}^{(2)}f_2$ denotes the Hardy--Littlewood maximal function in the second fiber.
One has (see \cite[Exercise 4.6.6]{Gra}),
\begin{align*}
\|h(x,y)\|_{L^p_x(\R)} \lesssim_p  \Big( \sum_{j}|I_{y,j}|\Big)^{1/p}.
\end{align*}
Therefore,
\begin{align*}
\|h\|_{L^p(\mathbb{R}^2)} \lesssim_p \Big\| \Big( \sum_{j}|I_{y,j}|\Big)^{1/p} \Big\|_{L^p_y(\mathbb{R})}\lesssim \lambda^{-r/p} \| \|f_{1,y}(x)\|_{L^p_x(\mathbb{R})} \|_{L^p_y(\mathbb{R})}= \lambda^{-r/p}.
\end{align*}
Putting everything together we obtain
\begin{align*}
&C_{u,v}^{-1}|\{(x,y) \in    (\cup_j 2I_{y,j} )^c \times \mathbb{R} : |U(b,f_2)| > \lambda \}| \\
&\lesssim   |\{(x,y) \in    (\cup_j 2I_{y,j} )^c \times \mathbb{R} :   h(x,y) \mathcal{M}^{(2)}f_2(x,y) > \lambda^{1-r/p} \}| \\
&\leq   \lambda^{-r+r^2/p} \|h \cdot \mathcal{M}^{(2)}f_2\|_r^r \\
&\leq  \lambda^{-r+r^2/p} \|h\|^r_p \|\mathcal{M}^{(2)}f_2\|_q^r \lesssim_p  \lambda^{-r+r^2/p-r^2/p} = \lambda^{-r},
\end{align*}
which finishes the proof.

\section{Applications}\label{sec:applications}

\subsection{Proof of Theorem \ref{thm:patterns}}\label{sec:patterns}
We will prove that for every measurable function $f$ on $[0,1]^2$ satisfying $0\le f\le 1$ and $\int_{[0,1]^2} f\ge \varepsilon$,
\begin{equation}\label{eqn:patternpenult}
\int_{[0,1]^3} f(x,y) f(x+t,y) f(x,y+t^2) dx \,dy \,dt > \delta,
\end{equation}
where $\delta=\delta(\varepsilon)=\exp(-\exp(\varepsilon^{-C}))$ for some $C>0$. This implies Theorem \ref{thm:patterns} by setting $f=\mathbf{1}_E$. To prove \eqref{eqn:patternpenult} we will need an upper and a lower bound. The upper bound is provided by Theorem \ref{mainresult} 
and the lower bound is the content of the following.

\begin{lemma}\label{lem:lowerbd}
Let $\vartheta\geq 0$ be an even smooth function which is supported in $[-2,2]$,  constant    on $[-1,1]$,  monotone on $[1,2]$ and normalized such that $\widehat{\vartheta}(0)=1$. Let $\vartheta_k(x) = 2^k \vartheta(2^k x)$.
For any function $f$ on $\R^2$ that is supported in $[0,1]^2$ and satisfies $0\le f\le 1$ and any $k,l\in \mathbb{N}$ we have 
\[\int_{[0,1]^2} f(f*_1\vartheta_k)(f*_2\vartheta_l) \geq c_0 \Big(\int_{[0,1]^2} f\Big)^4 \]
for some constant $c_0>0$ depending only on $\vartheta$. 
\end{lemma}
This is an analogue of \cite[Lemma 6]{Bou88}. We postpone the proof of the lemma to the end of this section. 

\begin{proof}[Proof of Theorem \ref{thm:patterns}]
This follows by the exact same argument as given by Bourgain in \cite{Bou88}, but we provide details for the sake of completeness.
Write
\[I=\int_{[0,1]^3} f(x,y) f(x+t,y) f(x,y+t^2) dx \,dy \,dt. \]

Let $\tau$ be smooth, supported on $[\frac12,2]$, taking values in $[0,1]$ and $\int_\R \tau=1$. Set $\tau_k(x)=2^k \tau(2^k x)$ for $k\in\mathbb{N}$.
Fix natural numbers $1<k<k'<k''$.
Estimate    
\[ 2^{k'} I \ge \int_{[0,1]^3} f(x,y) f(x+t,y) f(x,y+t^2) \tau_{k'}(t) dx \,dy  \,dt = I_1 + I_2 + I_3, \]
where
{\allowdisplaybreaks
\begin{align*}
   I_1 &= \int_{[0,1]^3} f(x,y) f(x+t,y) (f*_2\vartheta_k)(x,y+t^2) \tau_{k'}(t) dx \,dy  \,dt,\\
   I_2 &= \int_{[0,1]^3} f(x,y) f(x+t,y) (f*_2\vartheta_{k''}-f*_2\vartheta_k) (x,y+t^2) \tau_{k'}(t) dx \,dy  \,dt,\\
    I_3 &=\int_{[0,1]^3} f(x,y) f(x+t,y) (f-f*_2\vartheta_{k''})  (x,y+t^2) \tau_{k'}(t) dx \,dy  \,dt.
\end{align*}}
By decomposing  ${f}-f*_2\vartheta_{k''}$ in frequency into dyadic blocks and applying Theorem \ref{mainresult} to each block, there exists $\sigma>0$ so that
\[ |I_3|\lesssim 2^{2\sigma k'-\sigma k''}.\]
By the Cauchy--Schwarz inequality in $(x,y)$ we also have 
\[ |I_2| \le  \|f*_2\vartheta_{k''} - f*_2\vartheta_k\|_2. \]
To estimate $I_1$ we decompose further into
\[ I_1 = I_4 + I_5 + I_6, \]
where
\begin{align*}
    I_4 &= \int_{[0,1]^3} f(x,y) f(x+t,y) (f*_2\vartheta_k)(x,y+t^2) \tau_{k'}(t) dx \,dy  \,dt -\int_{[0,1]^2} f (f*_1 \tau_{k'}) (f*_2\vartheta_k),\\
    I_5 &= \int_{[0,1]^2} f (f*_1 \tau_{k'}) (f*_2\vartheta_k) - \int_{[0,1]^2} f (f*_1 \vartheta_{k}) (f*_2\vartheta_k),\\
    I_6 &= \int_{[0,1]^2} f (f*_1 \vartheta_{k}) (f*_2\vartheta_k).
\end{align*}
By the mean value theorem applied to $f*_2\vartheta_k$ we have
\[ |I_4| \lesssim 2^{k-k'}. \]
Further, we estimate
\[ |I_5| \leq \|f*_1\tau_{k'} - f*_1\vartheta_{k}\|_2.\]
The right-hand side is bounded by
\[
\|f*_1\tau_{k'}*_1\vartheta_{k''} - f*_1\vartheta_{k}*_1\tau_{k'}\|_2 + \|\tau_{k'}-\tau_{k'}*\vartheta_{k''}\|_1 + \|\vartheta_k - \vartheta_k*\tau_{k'}\|_1.
\]
Using $\int \tau=1$ and applying the mean value theorem applied to the last two terms finally gives 
\[
|I_5| \le \|f*_1\vartheta_{k''} - f*_1\vartheta_{k}\|_2 +
O(2^{k'-k''} + 2^{k-k'}).\]
Finally, Lemma \ref{lem:lowerbd} implies
\[ |I_6| \ge c_0 \varepsilon^4.  \]
for some constant $c_0\in (0,\infty)$.

Putting together the estimates for  $I_1,\ldots, I_5$ and choosing $k'$ sufficiently large with respect to $k$ and $k''$ sufficiently large with respect to $k'$  we obtain
\[c_0 \varepsilon^4 \le 2^{k'}I +  \|f*_2\vartheta_{k''} - f*_2\vartheta_k\|_2 + \|f*_1\vartheta_{k''} - f*_1\vartheta_{k}\|_2 + 2^{-100}c_0\varepsilon^4. \]
Thus either
 $I>2^{-k'-10}c_0\varepsilon^4$ or 
 \[ \|f*_2\vartheta_{k''} - f*_2\vartheta_k\|_2 + \|f*_1\vartheta_{k''} - f*_1\vartheta_{k}\|_2 > 2^{-10}c_0\varepsilon^4.\]

Given a sufficiently large constant $M$ and an initial value $k_0$, we recursively construct a sequence $k_0<k_1<\ldots$ by $k_{l+1}= M k_l$ such that for each $l$ either
 \begin{align}\label{dychotomy1}
     I>2^{-k_{l+1}-10}c_0\varepsilon^4
 \end{align}
 or 
 \begin{align}\label{dychotomy2}
   \|f*_2\vartheta_{k_{l+1}} - f*_2\vartheta_{k_l}\|_2 + \|f*_1\vartheta_{k_{l+1}} - f*_1\vartheta_{k_l}\|_2 > 2^{-10}c_0\varepsilon^4.
 \end{align}
 The constant $M$ can be chosen independent of $\varepsilon$ by making $k_0$ sufficiently large (specifically, $k_0\gtrsim \log(\varepsilon^{-1})$).
 
 Suppose that \eqref{dychotomy1} fails for all $l=0,\dots,L$.
By \eqref{dychotomy2} and Plancherel's theorem, since $\|f\|_2\lesssim 1$,
\[ L \cdot 2^{-10} c_0 \varepsilon^4\le \sum_{l=0}^L \|f*_2\vartheta_{k_{l+1}} - f*_2\vartheta_{k_l}\|^2_2 +   \|f*_1\vartheta_{k_{l+1}} - f*_1\vartheta_{k_l}\|^2_2\lesssim 1.\]
Thus, if \eqref{dychotomy1} fails for all $l=0,\dots,L$, then $L \lesssim \varepsilon^{-4}$. In other words, there exists $l$ with $l\lesssim \varepsilon^{-4}$ so that \eqref{dychotomy1} holds. Since $k_l\le M^l k_0$ we obtain 
\[ I> \exp(-\exp(\varepsilon^{-C})) \]
for some $C<\infty$.
\end{proof}

It remains to give the proof of Lemma \ref{lem:lowerbd}.
\begin{proof}[Proof of Lemma \ref{lem:lowerbd}]
For $k\in \Z$ denote by $\mathcal{D}_k$ the set of all dyadic intervals of size $2^{-k}$ which are contained in $[0,1]$. For a function $g$ on $[0,1]$ we denote the  martingale averages
\[E_kg =\sum_{I\in \mathcal{D}_k} \Big(   |I|^{-1}\int_I g \Big ) \mathbf{1}_I.\]
Because of the pointwise bound $E_kg \lesssim g*\vartheta_k$
it suffices to show 
\[\int_{[0,1]^2} f(E_k^{(1)}f)(E_l^{(2)}f) \geq  \Big(\int_{[0,1]^2} f\Big)^4, \]
where  we have denoted $E_k^{(1)}f(x,y) = (E_kf(\cdot,y))(x)$ and   $E_k^{(2)}f(x,y) = (E_kf(x,\cdot))(y)$.

This follows by two applications of the Cauchy--Schwarz inequality. Indeed, observe that
\[
\Big( \int_{[0,1]^2}f(x,y)dx\,dy\Big)^4 = \Big( \int_{[0,1]}\Big( \sum_{I\in\mathcal{D}_k}|I|^{-\frac{1}{2}}|I|^{\frac{1}{2}} \int_If(x,y)dx\Big) dy\Big)^4.
\]
 By the Cauchy--Schwarz inequality, the right--hand side is bounded by
\[\Big( \int_0^1 \sum_{I\in \mathcal{D}_k} |I|^{-1}\Big( \int_I f(x,y)dx\Big)^2 dy\Big)^2.  \]
Expanding the square, using Fubini's theorem and splitting the integration in $y$ over dyadic intervals of scale $2^{-l}$, this can be written as
\[\Big( \sum_{I\in \mathcal{D}_k}|I|^{-1}\int_I\int_I \Big( \sum_{J\in \mathcal{D}_l} |J|^{-1/2}|J|^{1/2} \int_J f(x,y)f(x',y)dy \Big) dx\,dx'\Big)^2. \]
By two more applications of the Cauchy--Schwarz inequality we obtain a bound by 
\[\sum_{I\in \mathcal{D}_k}|I|^{-1}\int_I\int_I  \sum_{J\in \mathcal{D}_l} |J|^{-1}\Big( \int_J f(x,y)f(x',y)dy \Big)^2 dx\,dx', \]
which equals
\[\sum_{I\in \mathcal{D}_k} \sum_{J\in \mathcal{D}_l} |I|^{-1}|J|^{-1} \int_I\int_I   \int_J \int_J f(x,y)f(x',y)f(x,y')f(x',y')  dy\,dy'\,dx\,dx' .\]
Using the upper bound $f\le 1$ and changing the order of integrations produces a majorization by
\[\int_{[0,1]} \int_{[0,1]} \sum_{I\in \mathcal{D}_k} \sum_{J\in \mathcal{D}_l} |I|^{-1}|J|^{-1} \int_I    \int_J f(x,y)f(x',y) f(x,y') \mathbf{1}_I(x)\mathbf{1}_J(y) dy\,dy'\,dx\,dx',  \]
which can be recognized as 
\[\int_{[0,1]^2} f(E_k^{(1)}f)(E_l^{(2)}f),\]
as desired.
\end{proof}
\subsection{Bilinear Hilbert transform with curvature}\label{sec:bht}
In this section we show how Theorem \ref{thm:singint} implies the $L^2\times L^2\rightarrow L^1$ estimate for the operator
\begin{equation}\label{eqn:bht}
(g_1,g_2)\mapsto \mathrm{p.v.}\int_{\R} g_1(x+t)g_2(x+t^2)\,t^{-1}\,dt.
\end{equation}
This estimate was first proved in \cite{Li13}.
Let $g_1,g_2,g_3$ be test functions on $\R$. Let $\varphi$ be non-negative, smooth and compactly supported function with $\|\varphi\|_{L^{2}(\R)}=1$
and $\lambda>0$. Set
\begin{align*}
&f_1(x,y) = g_1(x+y)\lambda^{1/2}\varphi(\lambda y)\\
&f_2(x,y) = g_2(x+y)\lambda^{1/2}\varphi(\lambda y)\\
&f_3(x,y) = g_3(x+y)
\end{align*}
Note that 
$\|f_i\|_{L^{p_i}(\R^2)} = \|g_i\|_{L^{p_i}(\R)}$ for $i=1,2,3$ for $(p_1,p_2,p_3)=(2,2,\infty)$.
With $T$ given by \eqref{eqn:thtc} we obtain 
\begin{align*}
\int_{\R^2}T(f_1,f_2)f_3& =\int_{\R^2}\textup{p.v.} \int_{\R}  g_1(x+y+t)g_2(x+y+t^2)g_3(x+y) \lambda\varphi(\lambda y)\varphi(\lambda (y+t^2))\frac{dt}{t}\,dx\,dy.
\end{align*}
Changing variables $x\mapsto x-y$, rescaling in $y$ and  changing the order of integration we obtain 
\begin{align*}
\mathrm{p.v.} \int_{\R}\int_{\R}g_1(x+t)g_2(x+t^2)g_3(x)\Phi(\lambda^{1/2}t) dx\, \frac{dt}{t}, 
\end{align*}
where $\Phi$ is a Schwartz function given by
\begin{equation}\label{eqn:bht-red-phi}
\Phi(t) = \int_{\R}\varphi(y)\varphi(y+t^2)dy.
\end{equation}
Taking $\lambda\rightarrow 0$ we obtain a  form dual to \eqref{eqn:bht} (the limit is justified by considering truncations of the $t$-integration and using the dominated convergence theorem). Thus, 
 the bound from Theorem \ref{thm:singint} implies the corresponding bound for \eqref{eqn:bht}.

\subsection{A maximal singular oscillatory integral of Stein--Wainger type}\label{sec:sw}
Consider the operator
\[ g\mapsto \sup_{N\in\R} \Big| \mathrm{p.v.} \int_{\R} g(x+t) e^{iN t^2} \frac{dt}t\Big|. \]
Stein \cite{Ste95} proved that this operator is bounded $L^2\to L^2$ (also see work of Stein and Wainger \cite{SW01}). Here we show how this bound can be obtained as a consequence of Theorem \ref{thm:singint}. First note that it is equivalent to prove that for every measurable function $N:\R\to\R$ the bilinear form
\[ (g_1, g_2)\mapsto \int_\R \mathrm{p.v.} \int_{\R} g_1(x+t) g_2(x) e^{iN(x) t^2} \frac{dt}t dx \]
is bounded $L^2\times L^2\rightarrow \mathbb{C}$ with constant uniform in $N(\cdot)$.
Let $\varphi$ be a non-negative smooth compactly supported function with $\|\varphi\|_{L^{2}(\R)}=1$ and $\lambda> 0$. Let
\begin{align*}
&f_1(x,y) = g_1(x) \lambda^{1/2}\varphi(\lambda y),\\
&f_2(x,y) = g_2(x)e^{iN(x)y} \lambda^{1/2}\varphi(\lambda y),\\
&f_3(x,y) = e^{-iN(x)y}.
\end{align*} 
Then $\|f_i\|_{L^2(\R^2)} = \|g_i\|_{L^2(\R)}$ for $i=1,2$, while $\norm{f_3}_\infty=1$. Moreover, 
\begin{align*}
\int_\R T(f_1,f_2)f_3 =\mathrm{p.v.} \int_{\R} \int_\R g_1(x+t) g_2(x) e^{iN(x)t^2}\Phi(\lambda^{1/2}t) dx\,\frac{dt}{t}, 
\end{align*}
where $\Phi$ is a Schwartz function (given by \eqref{eqn:bht-red-phi}). Taking a limit $\lambda\to 0$ and applying Theorem \ref{thm:singint} implies the claim.

\begin{remarka}
The reductions in \S \ref{sec:bht} and \S \ref{sec:sw} are adapted from \cite[Appendix B]{KTZ15}. The arguments are not restricted to $L^2$ bounds.
\end{remarka}

\section{Some open problems} \label{section:openproblems}
1.\ It would be interesting to study analogues in higher dimensions and/or with higher orders of multilinearity. Among these are the multilinear operators
\[ S_\alpha  (\mathbf{f})(x)= \mathrm{p.v.}\int_{\R^d} \prod_{j=1}^d f_j(x+t^{\alpha_j} e_j)\,\frac{dt}t\quad(x\in\R^d), \]
where $\mathbf{f}=(f_1,\dots,f_d)$ with $f_j:\mathbb{R}^d\to \mathbb{C}$, $\alpha\in\mathbb{N}^d$, and $e_j\in\R^d$ denoting the $j$th standard unit vector. The operators $S_{(1,\dots,1)}$ are known as simplex Hilbert transforms (\cite{Zor17}, \cite{DKT16}, \cite{DR18}). 
Theorem~\ref{thm:singint} concerns the case $d=2$ (the operator we study is $T=S_{(1,2)}$), and is the first positive result to be established when $d\ge 2$. The analysis developed here, as it currently stands, does not suffice to treat $d\ge 3$. 

2.\ Another variant is
\begin{equation} \label{anothervariant}
(f_1,f_2,f_3)\mapsto \int_\R f_1 (x+t) f_2(x+t^2) f_3(x+t^3)\, \frac{dt}t, 
\end{equation}
with $f_j:\R^1\to\mathbb{C}$, which can be viewed as a trilinear Hilbert transform with curvature. This operator was suggested by Lie \cite[\S 7]{Lie15} as a model for the trilinear Hilbert transform. Bounds for this object would be implied by bounds for $S_{(1,2,3)}$. No $L^p$ bounds are currently known for the trilinear Hilbert transform, though some cancellation was established in \cite{Tao16}, \cite{Zor17}, \cite{DKT16}. Both the operator $S_{(1,2)}$ treated in this paper,
and the variant \eqref{anothervariant},
have characters intermediate between those
of $S_{(1,2,3)}$,
and the operator \eqref{eqn:bht0}.

3.\ 
It would be desirable to extend the range of exponents in Theorem~\ref{thm:anisotp} to $r\ge 2$ (this would immediately give a corresponding extension of Theorem~\ref{thm:singint}).
It would also be interesting to extend Theorem~\ref{thm:singint} 
and Theorem~\ref{thm:maxfct} to a large range of exponents with $r<1$. A careful inspection of the arguments in \S \ref{sec:prelim} shows that  we are actually able to obtain some bounds with $r > 1-\varepsilon$ for some small $\varepsilon>0$ that is unlikely to be sharp.
We remark that the conclusions of Theorems \ref{thm:singint} and \ref{thm:anisotp} fail if $p=\infty$ or $q=\infty$.

4.\	It would be natural to study the analogous discrete maximal function
\[
M_{\Z}(f_1,f_2)(x,y) = \sup_{N\in\mathbb{N}} N^{-1}  
\sum_{n=1}^N \big| f_1(x+n,y)f_2(x,y+n^2)\big|
\]
with $f_j:\mathbb{Z}^2\to\mathbb{C}$ and $(x,y)\in\mathbb{Z}^2$. On a related note, we refer the reader to a very recent work by Krause, Mirek and Tao \cite{KMT20} concerning  pointwise ergodic theorems for certain bilinear polynomial averages.


\end{document}